\documentclass[11pt, oneside]{article}   	
\usepackage{geometry}                		
\usepackage{graphicx}				
\usepackage{amssymb}
\usepackage{amsmath}
\usepackage{amsthm}
\usepackage{url}
\usepackage{enumerate} 

\newcommand{\Ker}{\mathop{\mathrm{Ker}}}
\newcommand{\sgn}{\mathop{\mathrm{sgn}}} 
\newcommand{\ac}{\mathrm{ac}} 

\newcommand{\re}{\mathop{\mathrm{Re}}} 
\newcommand{\im}{\mathop{\mathrm{Im}}} 
\newcommand{\tr}{\mathop{\mathrm{Tr}}}
\newcommand{\supp}{\mathop{\mathrm{supp}}}

\DeclareMathOperator{\slim}{s-lim}

\newcommand{\N}{\mathbb{N}} 
\newcommand{\Z}{\mathbb{Z}}
\newcommand{\R}{\mathbb{R}} 
\newcommand{\C}{\mathbb{C}}

\numberwithin{equation}{section}

\theoremstyle{plain}
\newtheorem{thm}{Theorem}[section]
\newtheorem{proposition}[thm]{Proposition}
\newtheorem{lemma}[thm]{Lemma} 
\newtheorem{corollary}[thm]{Corollary}

\theoremstyle{definition} 

\newtheorem{defn}[thm]{Definition} 
\newtheorem{example}[thm]{Example}

 \newtheorem{remark}[thm]{Remark}

 \newtheorem*{remarks*}{Remarks}
\newtheorem*{remark*}{Remark}

\title{Uniform resolvent and orthonormal Strichartz estimates for repulsive Hamiltonian}
\author{Akitoshi Hoshiya\thanks{Graduate School of Mathematical Sciences, The University of Tokyo, 3-8-1 Komaba, Meguro-ku, Tokyo 153-8914, Japan \\
 Email address: hoshiya@ms.u-tokyo.ac.jp}}

\begin{document}
\maketitle

\begin{abstract}
We consider the uniform resolvent and orthonormal Strichartz estimates for the Schr\"odinger operator. First we prove the Keel-Tao type theorem for the orthonormal Strichartz estimates, which means that the dispersive estimates yield the orthonormal Strichartz estimates for strongly continuous unitary groups. This result applies to many Schr\"odinger propagators which are difficult to treat by the smooth perturbation theory, for example, local-in-time estimates for the Schr\"odinger operator with unbounded electromagnetic potentials, the $(k, a)$-generalized Laguerre operators and global-in-time estimates for the Schr\"odinger operator with scaling critical magnetic potentials including the Aharonov-Bohm potentials. Next we observe mapping properties of resolvents for the repulsive Hamiltonian and apply to the orthonormal Strichartz estimates. We prove the Kato-Yajima type uniform resolvent estimates with logarithmic decaying weight functions. This is new even when without perturbations. The proof is dependent on the microlocal analysis and the Mourre theory. We also discuss mapping properties on the Schwartz class and the Lebesgue space.  
\end{abstract}

\section{Introduction}\label{23110221}
The Strichartz estimates have been studied by many authors for dispersive and hyperbolic equations. For the Schr\"odinger equation:
\[
\left\{
\begin{array}{l}
i\partial_t u= Hu, \\
u(0)=u_0,
\end{array}
\right.
\]    
where $H$ is a self-adjoint operator on $L^2 (\R^d)$, they are inequalities like
\begin{align}
\|e^{-itH} P_{\ac} (H)u\|_{L^q _t L^r _x} \lesssim \|u\|_2 \label{2406221304}
\end{align} 
where $L^q _t L^r _x = L^q (\R_t; L^r (\R^d _x))$ and $q, r \in [1, \infty]$ satisfy certain conditions. For the free Hamiltonian $H=-\Delta$, (\ref{2406221304}) is proved for any $\frac{d}{2}$ admissible pair $(q, r)$ (see Notations below) by \cite{St}, \cite{GV}, \cite{Y1} and \cite{KT}. Afterwards, (\ref{2406221304}) is extended to the Schr\"odinger operator $H=-\Delta +V$, for example, by \cite{BPST}, \cite{BM}, \cite{D2}, \cite{DF}, \cite{EGS}, \cite{M1}, \cite{M2}, \cite{RS} and \cite{Ta1} (see also \cite{GVV} for a negative result). Furthermore, \cite{D2}, \cite{DF}, \cite{EGS} and \cite{GYZZ} proved (\ref{2406221304}) for the magnetic Schr\"odinger operator. See also \cite{BGT}, \cite{BGH}, \cite{BT1}, \cite{BT2}, \cite{MMT}, \cite{M3}, \cite{M4}, \cite{MT}, \cite{MY1} and \cite{Ta2} for the Strichartz estimates on manifolds or for variable coefficient operators. Note that the Strichartz estimates for the higher order or fractional operators are proved in \cite{FSWY}, \cite{MY1} and \cite{MY2}. Recently (\ref{2406221304}) for $-\Delta$ is extended to systems of orthonormal functions:
\begin{align}
\left \| \sum_{j=0}^ \infty{\nu_j |e^{it\Delta}f_j|^2} \right\|_{L^{q/2} _t L^{r/2} _x} \lesssim \| \nu\|_{l^{\beta}}.
\label{2406221341}
\end{align}
where $\{f_j\}$ is any orthonormal system in $L^2 (\R^d)$ and $\nu = \{\nu_j\}$ is any complex-valued sequence. This is called the orthonormal Strichartz estimate. (\ref{2406221341}) first appeared in \cite{FLLS} and extended in \cite{FS}, \cite{BHLNS}, \cite{BKS2} and \cite{BLN2}. See also \cite{FS}, \cite{BKS1} and \cite{BLN} for the wave or Klein-Gordon equations. The first purpose of this paper is to extend (\ref{2406221341}) for various operators. In the author's previous research, \cite{H1}, (\ref{2406221341}) is extended to the Schr\"odinger operator with scaling critical electric potentials and magnetic Schr\"odinger operator with very short range potentials by the perturbation method based on the smooth perturbation theory. This also works for the wave, Klein-Gordon and Dirac equations with potentials (see \cite{H2}). See \cite{Ms}, \cite{MS} and \cite{SBMM} for the orthonormal Strichartz estimates for the Dunkl Laplacian or $(k, a)$-generalized Laguerre operator. In the case of the ordinary Strichartz estimates, thanks to the Keel-Tao theorem (\cite{KT}), (\ref{2406221304}) is reduced to prove the (microlocal) dispersive estimates. This makes it possible to prove the Strichartz estimates for the Schr\"odinger operator with slowly decaying potentials (\cite{M1}), which are difficult to prove by the perturbation method invented in \cite{RS}. On the other hand, proofs for the orthonormal Strichartz estimates have been based on the duality principle invented in \cite{FS} since sufficient conditions like the Keel-Tao theorem have not been proved. The first result in this paper is the Keel-Tao type theorem for the orthonormal Strichartz estimates, which is a refinement of the Keel-Tao theorem for strongly continuous unitary groups.
\begin{thm}\label{2405021538}
Let $\mathcal{H} =L^2 (X)$ and $H$ be a self-adjoint operator on $\mathcal{H}$, where $X$ is a $\sigma$-finite measure space. Furthermore we assume there exist bounded Borel measurable function $\phi$ and $\sigma >0$ satisfying
\[\|e^{-itH} |\phi|^2 (H)\|_{1 \rightarrow \infty} \lesssim |t|^{-\sigma}\]
for all $t \in \R \setminus \{0\}$. If $\beta >0$ and $(q, r)$ is a $\sigma$-admissible pair (see Notations below) satisfying either of the following:
\begin{enumerate}
\item
If $\sigma \ge 1$, $(r, \beta)$ satisfies $r \in [2, \frac{2(2\sigma +1)}{2\sigma -1})$ and $\beta = \frac{2r}{r+2}$ or $r \in [\frac{2(2\sigma +1)}{2\sigma -1}, \frac{2\sigma}{\sigma -1})$ and $\beta < \frac{q}{2}$,
\item
If $\sigma \in (\frac{1}{2}, 1)$, $(r, \beta)$ satisfies $r \in [2, \frac{2(2\sigma +1)}{2\sigma -1})$ and $\beta = \frac{2r}{r+2}$ or $r \in [\frac{2(2\sigma +1)}{2\sigma -1}, \infty)$ and $\beta < \frac{r(2\sigma -1)}{r(2\sigma -1) -2}$,
\item
If $\sigma \in (0, \frac{1}{2}]$, $(r, \beta)$ satisfies $r \in [2, \infty)$ and $\beta = \frac{2r}{r+2}$,
\end{enumerate}
then we have
\begin{align}
\left\| \sum_{j=0}^{\infty} \nu_j |e^{-itH} \phi (H) f_j|^2 \right\|_{L^{\frac{q}{2}} _t  L^{\frac{r}{2}} _x} \lesssim \|\nu\|_{l^{\beta}} \label{2405021614}
\end{align}
for all orthonormal systems $\{f_j\}$ in $\mathcal{H}$ and complex-valued sequences $\nu = \{\nu _j\}$.
\end{thm}
\begin{remark}\label{2405050957}
Under the conditions in Theorem \ref{2405021538} we have
\begin{align*}
\|e^{-itH} \phi (H)\|_{\mathcal{B} (\mathcal{H})} \lesssim 1 \quad and \quad \|e^{-isH} \phi (H) (e^{-itH} \phi (H))^* \|_{1 \to \infty} \lesssim |t-s|^{-\sigma}
\end{align*}
for all $t, s \in \R$ satisfying $t \ne s$. Then by the Keel-Tao theorem (\cite{KT}) we have the ordinary Strichartz estimates:
\begin{align}
&\|e^{-itH} \phi (H) u\|_{L^{q} _t L^r _x} \lesssim \|u\|_{\mathcal{H}} \label{2405051013}, \\
&\left\|\int_{0}^{t} e^{-i(t-s)H} |\phi |^2 (H) F(s) ds \right\|_{L^q _t L^r _x} \lesssim \|F\|_{L^{\tilde{q}'} _t L^{\tilde{r}'} _x} \notag
\end{align}
for all $\sigma$-admissible pairs $(q, r)$ and $(\tilde{q}, \tilde{r})$. We note that (\ref{2405051013}) is equivalent to
\begin{align}
\left\| \sum_{j=0}^{\infty} \nu_j |e^{-itH} \phi (H) f_j|^2 \right\|_{L^{\frac{q}{2}} _t  L^{\frac{r}{2}} _x} \lesssim \|\nu\|_{l^{1}} \label{2405051024}
\end{align} 
which is weaker than (\ref{2405021614}). Though the endpoint is excluded in Theorem \ref{2405021538}, it is proved in \cite{FS2} that if $H=-\Delta$, (\ref{2405051024}) cannot be refined, i.e. the $l^1$ norm in the right hand side cannot be changed to $l^{\beta}$ norm for $\beta >1$.   
\end{remark}
The next corollary applies to many Schr\"odinger type operators, for example, with unbounded potentials, which cannot be regarded as perturbations. 
\begin{corollary}\label{2405021556}
Let $I \subset \R$, $\mathcal{H} =L^2 (X)$ and $H$ be a self-adjoint operator on $\mathcal{H}$, where $X$ is a $\sigma$-finite measure space. Furthermore we assume there exist bounded Borel measurable function $\phi$ and $\sigma >0$ satisfying
\[\|e^{-i(t-s)H} |\phi|^2 (H)\|_{1 \rightarrow \infty} \lesssim |t-s|^{-\sigma}\]
for all $t, s \in I$ with $t \ne s$. Then we have
\begin{align}
\left\| \sum_{j=0}^{\infty} \nu_j |e^{-itH} \phi (H) f_j|^2 \right\|_{L^{\frac{q}{2}} (I; L^{\frac{r}{2}} _x )} \lesssim \|\nu\|_{l^{\beta}} \label{2405021618}
\end{align}
for all orthonormal systems $\{f_j\}$ in $\mathcal{H}$, complex-valued sequences $\nu = \{\nu _j\}$ and $(q, r, \beta)$ as in Theorem \ref{2405021538}.
\end{corollary}
Now we give applications of Theorem \ref{2405021538}. In the next corollary we consider the Schr\"odinger operator with the Aharonov-Bohm potential defined by the Friedrichs extension.   
\begin{corollary}\label{2405080015}
Let $H_{AB} := \left(D_x + \frac{A(\hat{x})}{|x|} \right)^2$ with $A(\hat{x}) := \alpha \left( -\frac{x_2}{|x|}, \frac{x_1}{|x|} \right), \alpha \ne 0$ be the Schr\"odinger operator with the Aharonov-Bohm potential, where $\hat{x}= \frac{x}{|x|} \in \mathbb{S}^1$. Then
\begin{align*}
\left\| \sum_{j=0}^{\infty} \nu_j |e^{-itH_{AB}} f_j|^2 \right\|_{L^{\frac{q}{2}} _t  L^{\frac{r}{2}} _x} \lesssim \|\nu\|_{l^{\beta}}
\end{align*}
holds for $(q, r, \beta)$ satisfying the assumption in Theorem \ref{2405021538} with $\sigma =1$.
\end{corollary}
The above example is the orthonormal Strichartz estimates for the Schr\"odinger operator with scaling critical magnetic potentials. Due to the lack of the Kato-smoothness of $|D|^{\frac{1}{2}} |x|^{-\frac{1}{2}}$, it is difficult to prove the orthonormal Strichartz estimates by the perturbation method. However since the dispersive estimates are known, we can prove Corollary \ref{2405080015}. We give details in Example \ref{2405072241} for Corollary \ref{2405080015} and more general scaling critical magnetic potentials. As applications of Corollary \ref{2405021556}, we prove the orthonormal Strichartz estimates for the Schr\"odinger operator with unbounded electromagnetic potentials in Example \ref{2405051203} and for the $(k, a)$-generalized Laguerre operator in Example \ref{2405061101}. As far as the author knows, this would be the first result on the Schr\"odinger operator with general unbounded potentials and $(k, a)$-generalized Laguerre operator with general $a \in (0, 2]$, though the ordinary Strichartz estimates are known. The author believes that Theorem \ref{2405021538} and Corollary \ref{2405021556} would be useful when we consider the Schr\"odinger operator with slowly decaying potentials such as Coulomb potentials or on manifolds. We refer to \cite{M1}, \cite{BT1} and \cite{BT2} where the Keel-Tao theorem is indispensable to prove the ordinary Strichartz estimates for such Hamiltonians. See also Remark \ref{2405080047} for further comments and examples including the orthonormal Strichartz estimates on manifolds. Next we consider the refined Strichartz estimates.
\begin{corollary}\label{2405091036}
Let $H_{AB}$ be as in Corollary \ref{2405080015}. Then we have
\begin{align*}
\|e^{-itH_{AB}} u\|_{L^q _t L^r _x} \lesssim \|u\|_{\dot{B}^0 _{2, 2\beta} (\sqrt{H_{AB}})}
\end{align*}
for all $(q, r, \beta)$ as in Corollary \ref{2405080015}. See Corollary \ref{2405081750} for the definition of $\dot{B}^0 _{2, 2\beta} (\sqrt{H_{AB}})$.
\end{corollary}
There are some papers on the refined Strichartz estimates. \cite{B}, \cite{BV} and \cite{CaKe} have proved the refined Strichartz estimates for $-\Delta$. Moreover their refinements are based on the Bourgain-Morrey type norms and they are more precise than our types. However, for example in \cite{BV}, their arguments are based on precise observations on the restriction theorem and the explicit formula of $e^{it\Delta}$ is essential. For the Schr\"odinger operator $H$ with general potentials, it is difficult to obtain an explicit formula of $e^{-itH} P_{\ac} (H)$. Hence there are few results on the refined Strichartz estimates for general $H$. On the other hand, it is found in \cite{FS} that, for $e^{it\Delta}$, the Besov-type refined Strichartz estimates are deduced from the orthonormal Strichartz estimates. Their proofs do not need explicit formula and much simpler than others. Based on their arguments, in \cite{H1}, the refined Strichartz estimates are proved for the Schr\"odinger operator with inverse square type potentials and for the magnetic Schr\"odinger operator with very short range potentials. In Section 2, we give a proof of Corollary \ref{2405091036} and give a refinement of the small data scattering for mass critical NLS with scaling critical magnetic potentials. Moreover, in Section 4, we prove the small data scattering for infinitely many particle systems with scaling critical magnetic potentials.

The second purpose of this paper is to prove mapping properties of resolvents for the repulsive Hamiltonian $H = -\Delta -x^2 +V$ and apply to the orthonormal Strichartz estimates. Here $V :\R^d \to \R$ is a perturbation to $H_0 = -\Delta -x^2$. Concerning the classical trajectories corresponding to $H_0$;
 \[
\left\{
\begin{array}{l}
\dot{y} (t) = 2\eta (t) \\
\dot{\eta} (t) = 2y(t) \\
y(0) =x, \eta (0) = \xi  \\
\end{array}
\right.
\]    
$y(t) = \frac{x+ \xi}{2} e^{2t} + \frac{x-\xi}{2} e^{-2t}$ and $\eta (t) = \frac{x+ \xi}{2} e^{2t} + \frac{\xi -x}{2} e^{-2t}$ hold. This implies that classical particles scatter faster than those of the free Laplacian. Therefore we may expect for better mapping properties of resolvents and faster decay of the propagator in some sense. The first result is the Kato-Yajima type uniform resolvent estimates for $H$. We set $\tilde{S}^0 (\R^d) := \{V \in C^{\infty} (\R^d; \R) \mid \forall \alpha \in \N^d _0, |\partial^{\alpha} _x V(x)| \lesssim \langle x \rangle^{-|\alpha|}\}$. 
\begin{thm}\label{2405200111}
Assume $H=H_0 +V$ satisfies $V \in \tilde{S}^0 (\R^d)$ and $\langle [H_0, iA]u, u \rangle \lesssim \langle [H, iA]u, u \rangle$ for all $u \in \mathcal{S} (\R^d)$. Then we have
\begin{align}
\sup_{z \in \C \setminus \R} \|\langle \log \langle x \rangle \rangle^{-1} (H-z)^{-1} \langle \log \langle x \rangle \rangle^{-1}\|_{\mathcal{B} (L^2 (\R^d))} < \infty. \label{2405200125}
\end{align}
where $A = a^w (x, D_x)$ and $a(x, \xi) = \log \langle x+\xi \rangle - \log \langle x-\xi \rangle \in S(\langle \log \langle x \rangle \rangle, dx^2 +d\xi^2)$.
\end{thm}
Typical examples for $H$ are $H= -\Delta -x^2 + c \frac{1}{\langle \log \langle x \rangle \rangle ^{\alpha}}$ for $\alpha >0$ and sufficiently small $c>0$. See Lemma \ref{2405111618} for details. Since for the free Laplacian, for $d \ge 3$, 
\begin{align*}
\sup_{z \in \C \setminus \R} \|\langle x \rangle^{-\gamma} (-\Delta -z)^{-1} \langle x \rangle^{-\gamma}\|_{\mathcal{B} (L^2 (\R^d))} < \infty
\end{align*}
holds iff $\gamma \ge 1$, Theorem \ref{2405200111} implies that $(H-z)^{-1}$ has a better property than $(-\Delta -z)^{-1}$ in terms of maps on weighted $L^2$ spaces. We remark that Theorem \ref{2405200111} contains a refinement of \cite{KaYo}, where the uniform resolvent estimates are proved with polynomially decaying weight functions and potentials which are not necessarily $C^{\infty}$. Though there are some papers treating energy-localized resolvent estimates, for example, \cite{BCHM}, \cite{I} and references therein, there seems to be no other result on the uniform resolvent estimates. In particular, Theorem \ref{2405200111} is the first result treating logarithmic decaying weight functions and potentials. By the smooth perturbation theory (\cite{KY}, \cite{KatoYajima} and \cite{D1}), we obtain the following Kato smoothing estimates.
\begin{corollary}\label{2405231157}
Let $H$ be as in Theorem \ref{2405200111}. Then we have
\begin{align*}
&\left\|\frac{1}{\langle \log \langle x \rangle \rangle} e^{-itH} u \right\|_{L^2 _t L^2 (\R^d)} \lesssim \|u\|_2, \\
& \left\|\frac{1}{\langle \log \langle x \rangle \rangle} \int_{0}^{t} e^{-i(t-s)H} \frac{1}{\langle \log \langle x \rangle \rangle} F(s) ds \right\|_{L^2 _t L^2 (\R^d)} \lesssim \|F\|_{L^2 _t L^2 (\R^d)}.
\end{align*}
\end{corollary}
As mentioned above, for the free Laplacian with $d \ge 3$, the Kato smoothing estimates hold with weight function $\langle x \rangle^{-\gamma}$ iff $\gamma \ge 1$. Hence we know $e^{-itH}$ decays faster than $e^{it\Delta}$. This is also seen from the (orthonormal) Strichartz estimates. 
\begin{thm}\label{2405231314}
Assume $H=H_0 +V$ is as in Theorem \ref{2405200111} and $|V(x)| \lesssim \langle \log \langle x \rangle \rangle ^{-2}$ holds. Then we obtain
\begin{align*}
\left\| \sum_{j=0}^{\infty} \nu_j |e^{-itH} f_j|^2 \right\|_{L^{\frac{q}{2}} _t  L^{\frac{r}{2}} _x} \lesssim \|\nu\|_{l^{\beta}}
\end{align*}
for any $(q, r, k, \beta)$ satisfying $(q, r) \in [2, \infty]^2, \frac{2}{q} = 2k\left( \frac{1}{2} - \frac{1}{r} \right), (q, r, k) \ne (2, \infty, 1), k \ge \frac{d}{2}$ and assumptions in Theorem \ref{2405021538}.
\end{thm}
Though there are no other result on the orthonormal Strichartz estimates for the repulsive Hamiltonian, \cite{C} proved the ordinary Strichartz estimates for $H_0$ with $\frac{d}{2}$ admissible pairs and \cite{KaYo} proved the ordinary Strichartz estimates for $H$ with $(q, r)$ as in Theorem \ref{2405231314} but $V$ is polynomially decaying. Therefore even for the ordinary Strichartz estimates, logarithmic decaying perturbations are new. These would have potential applications to NLS with repulsive potentials such as \cite{C} and \cite{Z}. Though we have seen good mapping properties of $(H-z)^{-1}$, the next theorem indicates that it does not act on the Schwartz space well, which is different from $(-\Delta -z)^{-1}$. As in \cite{Ta3}, we set $\C^{+} = \{z \in \C \mid \im z >0\}$, $\Omega_{r, R} (s) = \{(x, \xi) \in T^* \R^d \mid |x| >R, |\xi| >r, \cos (x, \xi) =s \}$, $\Omega_{s_1, s_2, r, R, mid} = \cup _{s \in [s_1, s_2]} \Omega_{r, R} (s)$ and $\Omega_{\epsilon, r, R, in / out} = \{(x, \xi) \in T^* \R^d \mid |x| >R, |\xi| >r, \pm \cos (x, \xi) <-1 +\epsilon \}$. We say that $f \in \mathcal{S}'$ satisfies $f \in \mathcal{S}$ microlocally outside $\Omega_{r, R} (1)$ iff for any $\epsilon >0$, there exists $a \in S^{0, 0}$ such that $a =1$ in $\Omega_{\epsilon, r, R, out}$, supported in $\Omega_{2\epsilon, r/2, R/2, out}$ and $(1-a^w (x, D_x))f \in \mathcal{S}$. Similarly, we say that $f \in \mathcal{S}'$ satisfies $f \notin \mathcal{S}$ microlocally in $\Omega_{r, R} (1)$ iff for any $\epsilon >0$, there exists $a \in S^{0, 0}$ such that $a$ is elliptic in $\Omega_{\epsilon, r, R, out}$, supported in $\Omega_{2\epsilon, r/2, R/2, out}$ and $a^w (x, D_x)f \notin \mathcal{S}$. Here we have used the scattering symbol class $S^{0, 0} = S \left(1, \frac{dx^2}{\langle x \rangle^2} + \frac{d\xi^2}{\langle \xi \rangle^2} \right)$. 
\begin{thm}\label{2406161602}
$(i)$ Let $V \in \tilde{S}^0 (\R^d)$ and $H=H_0 + cV$ for $c \in (0, 1)$. Then for any $z \in \C^{+}$, there exists $c_0 >0$ such that if $c \in (0, c_0)$ and $f \in \mathcal{S}$, $(H-z)^{-1} f \in \mathcal{S}$ microlocally outside $\Omega _{r, R} (1)$ for any large $r, R >0$. In particular, for any $z \in \C^{+}$ and $f \in \mathcal{S}$, $(H_0 -z)^{-1} f \in \mathcal{S}$ microlocally outside $\Omega _{r, R} (1)$ for any large $r, R >0$.

\noindent $(ii)$ For any $z \in \C^{+}$, there exists $f \in \mathcal{S}$ such that $(H_0 -z)^{-1} f \notin \mathcal{S}$ microlocally in $\Omega _{r, R} (1)$ for some $r, R >0$.
\end{thm}
Actually we prove $(H-z)^{-1} f \in \mathcal{S}$ microlocally outside $\Omega _{r, R} (1) \cap \{ |x| = |\xi|\}$. Theorem \ref{2406161602} implies that there exists $f \in \mathcal{S}$ such that $(H-z)^{-1} f$ does not decay rapidly in the radial sink (outgoing region). On the other hand, for any $f \in \mathcal{S}$, $(H-z)^{-1} f$ decays rapidly away from the radial sink. This nature and difficulties in the proof come from bad behavior of $\frac{1}{\xi^2 -x^2 -z}$, which does not belong to good symbol classes. Such non-ellipticity also appears in the case of the Klein-Gordon operator (see \cite{Ta3}). However, in that case, resolvents are homeomorphisms on $\mathcal{S}$. This difference comes from the totally different behavior of classical particles $(y(t), \eta (t))$, where in the case of the Klein-Gordon operator, null-bicharacteristics satisfy $|\eta (t)| \lesssim _{x, \xi} 1$.

This paper is organized as follows. In section \ref{2311152141}, we prove Theorem \ref{2405021538}, Corollary \ref{2405021556} and give their applications. NLS with scaling critical magnetic potentials is also discussed. In section \ref{2405081322}, the Kato-Yajima estimates, orthonormal Strichartz estimates, uniform Sobolev estimates and mapping properties Theorem \ref{2406161602} are proved. In section \ref{2406230934} we prove the small data scattering for infinitely many particle systems with scaling critical magnetic potentials.
\subsubsection*{\textbf{Notations}}

\begin{itemize}

\item For $(q, r) \in [2, \infty]^2$ and $\sigma >0$, we say that $(q, r)$ is a $\sigma$-admissible pair if $\frac{2}{q} =2\sigma (\frac{1}{2} - \frac{1}{r})$ and $(\sigma, q, r) \ne (1, 2, \infty)$ hold.
\item
For a measure space $(X, d\mu)$, $ L^p(X)$ denotes the ordinary Lebesgue space and its norm is denoted by $\| \cdot \|_p$. Similarly, $L^{p, q} (X)$ denotes the ordinary Lorentz space and its norm is denoted by $\|\cdot\|_{p, q}$.

\item
 $\mathcal{F}$ denotes the Fourier transform on $\mathcal{S}'$. Here $\mathcal{S}'$ denotes the set of all the tempered distributions.

\item
 For a Banach space $X$, $L^p_t X$ denotes the set of all the measurable functions $f : \mathbb{R} \rightarrow X$ such that $\|f\|_{L^p_t X} :=(\int_{\mathbb{R}} \|f(t)\|^p_X dt)^{1/p} < \infty$.
 
\item For $p, q \in [1, \infty]$ and $s \in \R$, $B^s _{p, q}$ ($\dot{B}^s _{p, q}$) denotes the ordinary inhomogeneous (homogeneous) Besov space. $H^{s, p}$ ($\dot{H}^{s, p}$) denotes the ordinary $L^p$-inhomogeneous (homogeneous) Sobolev space. 

\item For a self-adjoint operator $H$ and a Borel measurable function $f$, $f(H)$ is defined as $f(H)= \int_{\mathbb{R}} f(\lambda) dE(\lambda)$, where $E(\lambda)$ is the spectral measure associated to $H$. $P_{\ac} (H)$ denotes the orthogonal projection onto the absolutely continuous subspace.

\item For a normed space $X$, $\mathcal{B}(X)$ denotes the set of all the bounded operators on $X$.
\end{itemize}

\section{A refinement of the Keel-Tao theorem for strongly continuous unitary groups}\label{2311152141}
\subsection{\textbf{Abstract theorem}}
In this subsection we prove a refinement of the Keel-Tao theorem, which states that the dispersive estimate implies the orthonormal Strichartz estimate. For $r \ge 2$, we define $\tilde{r}:=2(\frac{r}{2})'$, where $(\frac{r}{2})'$ is the H\"older conjugate exponent  of $\frac{r}{2}$. The next lemma is called the duality principle, which is originated in \cite{FS}. See \cite{BHLNS} and \cite{BLN} for our type.
\begin{lemma}\label{2405031343}
Let $p, q, r \in [2, \infty ]$ and $\beta \in [1, \infty ]$. Suppose $A: \R \rightarrow \mathcal{B} (L^2 (X))$ is strongly continuous and $\sup _{t \in \R} \|A(t)\| < \infty$ holds, where $X$ is a $\sigma$-finite measure space. Then the following are equivalent:
\begin{enumerate}
\item
We have
\begin{align*}
\left\| \sum_{j=0}^{\infty} \nu_j |A(t) f_j|^2 \right\|_{L^{\frac{q}{2}, \frac{p}{2}} _t  L^{\frac{r}{2}} _x} \lesssim \|\nu\|_{l^{\beta}}
\end{align*}
for all orthonormal systems $\{f_j\}$ in $L^2 (X)$ and all complex-valued sequences $\nu = \{\nu _j\}$.
\item
We have
\begin{align*}
\|WA(t) A(t)^* \overline{W}\|_{\mathfrak{S}^{\beta'}} \lesssim \|W\|_{L^{\tilde{q}, \tilde{p}} _t L^{\tilde{r}} _x}
\end{align*}
for all $W \in L^{\tilde{q}, \tilde{p}} _t L^{\tilde{r}} _x$, where $A(t)^* : L^1 _t L^2 (X) \rightarrow L^2 (X)$ is a formal adjoint of $A(t): L^2 (X) \rightarrow L^{\infty} _t L^2 (X)$.
\end{enumerate}
\end{lemma}
Now we prove the orthonormal Strichartz estimates for restricted admissible pairs. We remark that (\ref{2405021534}) is a more refined estimate than (\ref{2405021614}). Our proof is based on that of the orthonormal Strichartz estimates for $-\Delta$ in \cite{FS} and for $\psi (D_x)$ in \cite{BLN}.
\begin{proposition}\label{2405021420}
Let $\mathcal{H} =L^2 (X)$ and $H$ be a self-adjoint operator on $\mathcal{H}$, where $X$ is a $\sigma$-finite measure space. Furthermore we assume there exist bounded Borel measurable function $\phi$ and $\sigma >0$ satisfying
\[\|e^{-itH} |\phi|^2 (H)\|_{1 \rightarrow \infty} \lesssim |t|^{-\sigma}\]
for all $t \in \R \setminus \{0\}$. If $(q, r)$ is a $\sigma$-admissible pair (see Notations in Section 1) satisfying $\max \{1+2\sigma, 2\} < \tilde{r} <2+2\sigma$, we have
\begin{align}
\left\| \sum_{j=0}^{\infty} \nu_j |e^{-itH} \phi (H) f_j|^2 \right\|_{L^{\frac{q}{2}, \beta} _t  L^{\frac{r}{2}} _x} \lesssim \|\nu\|_{l^{\beta}} \label{2405021534}
\end{align}
for $\beta = \frac{2r}{r+2}$, all orthonormal systems $\{f_j\}$ in $\mathcal{H}$ and all complex-valued sequences $\nu = \{\nu _j\}$.
\end{proposition}
\begin{proof}
(Step 1) For a (space-time) simple function $F$ and $z \in \C$ with $\re z <0$, we define
\begin{align}
T_{z} F (s, x) = \int_{\R} i(e^{iz\pi /2} (t-s)^{-z-1} _{+} -e^{-iz\pi /2} (s-t)^{-z-1} _{+}) e^{-i(t-s)H} |\phi|^2 (H) F(t) dt \label{2405051128}
\end{align}
Here $s \in \R$, $x \in X$ and $a_+ = \max \{a, 0\}$ for $a \in \R$. Then it is easy to see that 
\begin{align*}
\langle T_{z} F, G \rangle = \int \int i(e^{iz\pi /2} (t-s)^{-z-1} _{+} -e^{-iz\pi /2} (s-t)^{-z-1} _{+}) \langle e^{-i(t-s)H} |\phi|^2 (H) F(t), G(s) \rangle dsdt
\end{align*}
is analytic in $\{z \in \C \mid \re z <0\}$ for all simple functions $F$ and $G$. Furthermore we have
\begin{align}
T_{-1} F (s, x) = 2\int e^{-i(t-s)H} |\phi|^2 (H) F(t) dt = 2 \{(e^{isH} \phi (H))(e^{isH} \phi (H))^* \} F. \label{2405042010}
\end{align}
We note that $(e^{isH} \phi (H))^* :L^1 _t \mathcal{H} \rightarrow \mathcal{H}$ is a formal adjoint operator of $e^{isH} \phi (H) : \mathcal{H} \rightarrow L^{\infty} _t \mathcal{H}$. By our assumption and the Dunford-Pettis theorem, $e^{-itH} |\phi|^2 (H)$ has a Schwartz kernel $K$ satisfying $K(t, \cdot ,  \cdot) \in L^{\infty} (X \times X)$ and
\[\|K(t)\|_{L^{\infty} (X \times X)} = \|e^{-itH} |\phi|^2 (H)\|_{1 \rightarrow \infty} \lesssim |t|^{-\sigma}.\]
Then for simple functions $W_1$, $W_2$, $G$ and $k \in \R$, we obtain
\begin{align*}
W_1 T_{-\frac{\tilde{r}}{2} +ik} W_2 G=W_1 (s, x) \int a_z (t, s) \int_{X} K(t-s, x, y) W_2 (t, y) G(t, y) dydt
\end{align*}
for $a_z (t, s) =  i(e^{iz\pi /2} (t-s)^{-z-1} _{+} -e^{-iz\pi /2} (s-t)^{-z-1} _{+})$. By computing the Hilbert-Schmidt norm we have
\begin{align*}
\|W_1 T_{-\frac{\tilde{r}}{2} +ik} W_2\|^2 _{\mathfrak{S}^2} &= \int |W_1 (s, x)|^2 |a_{-\frac{\tilde{r}}{2} +ik} (t, s)|^2 |K(t-s, x, y)|^2 |W_2 (t, y)|^2 dsdtdxdy \\
& \lesssim \int |W_1 (s, x)|^2 |W_2 (t, y)|^2  |t-s|^{-2\sigma -2+ \tilde{r}} dsdtdxdy \\
& = \int \|W^2 _1 (s)\|_1 \|W^2 _2 (t)\|_1 |t-s|^{-2\sigma -2+ \tilde{r}} dsdt \\
& \lesssim \|W^2 _1\|_{L^{u, 2} _t L^1 _x} \|\|W^2 _2 \|_1 * |\cdot|^{-2\sigma -2+ \tilde{r}}\|_{L^{u', 2} _t} \\
& \lesssim \|W_1\|^2 _{L^{2u, 4} _t \mathcal{H}} \|W_2\|^2 _{L^{2u, 4} _t \mathcal{H}}.
\end{align*}
Here $u \in \R$ is defined by $\frac{2}{u} = \tilde{r} -2\sigma$ and we have used the Hardy-Littlewood-Sobolev inequality in the last line. We also note that the implicit constant is at most exponential. This yields
\begin{align}
\|W_1 T_{-\frac{\tilde{r}}{2} +ik} W_2\|_{\mathfrak{S}^2} \lesssim  \|W_1\|_{L^{2u, 4} _t \mathcal{H}} \|W_2\|_{L^{2u, 4} _t \mathcal{H}} \label{2405021809}
\end{align}
for all simple functions $W_1$ and $W_2$. 

(Step 2) We restrict $z \in \{z \in \C \mid \re z <0\}$ to $\{z \in \C \mid -\frac{1}{10} <\re z <0\}$. We set
\begin{align}
& \tilde{T} _z G = \mathcal{F}^* _{\tau} \left[ \int_{\R} \frac{(\tau -\lambda)^{z} _+}{\Gamma (z+1)} |\phi|^2 (\lambda) dE(\lambda) \right] \mathcal{F} _t G  \label{2405021847} \\
& \tilde{T} _{z, R, \tilde{R}} G =  \mathcal{F}^* _{\tau} \chi _{R} (\tau) \left[ \int_{\R} \frac{(\tau -\lambda)^{z} _+}{\Gamma (z+1)} \tilde{\chi} _{\tilde{R}} (\tau - \lambda)|\phi|^2 (\lambda) dE(\lambda) \right] \mathcal{F} _t G \label{2405021848}
\end{align} 
for $R, \tilde{R} >1$ and $G$ as below:
\begin{align}
G= \sum_{i \in I} \phi_i (t) u_i \quad  where \quad |I| < \infty, \quad \phi_i \in C^{\infty} _0 (\R; \C) \quad and \quad  u_i \in \mathcal{H}. \label{2405021904}
\end{align}
In (\ref{2405021848}), $\chi$ and $\tilde{\chi}$ are cut off functions near the origin satisfying $\chi (t), \tilde{\chi} (t) =1$ if $|t| \le 1$ and $\chi (t), \tilde{\chi} (t) =0$ if $|t| \ge 2$. We have set $\chi _R (t) = \chi (\frac{t}{R})$ and $\tilde{\chi} _{\tilde{R}} (t)=\tilde{\chi} (\frac{t}{\tilde{R}})$. We recall that functions as in (\ref{2405021904}) are dense in $L^{p, q} _t \mathcal{H}$ for all $p, q \in [1, \infty )$. First we prove that (\ref{2405021847}) and (\ref{2405021848}) are well-defined. For $G$ as in (\ref{2405021904}), we have
\begin{align*}
\|(\tau -H)^{z} _+ \mathcal{F} _t G\|^2 _{L^2 _{\tau} \mathcal{H}} &\lesssim \sum_{i \in I} \|\hat{\phi_i} (\tau) (\tau -H)^{z} _{+} u_i\|^2 _{L^2 _{\tau} \mathcal{H}} \\
& \lesssim \sum_{i \in I} \int |\hat{\phi_i} (\tau)|^2 \int (\tau -\lambda)^{2\re z} _+ d\|E(\lambda) u_i\|^2 _{\mathcal{H}} d\tau 
\end{align*}
and
\begin{align*}
&\int |\hat{\phi_i} (\tau)|^2 \int (\tau -\lambda)^{2\re z} _+ d\|E(\lambda) u_i\|^2 _{\mathcal{H}} d\tau \\ &\lesssim \int \int_{|\tau -\lambda| \ge1} + \int \int_{|\tau -\lambda| \le1} |\hat{\phi_i} (\tau)|^2 (\tau -\lambda)^{2\re z} _+ d\|E(\lambda) u_i\|^2 _{\mathcal{H}} d\tau \\
& \lesssim \|\hat{\phi_i}\|^2 _2 \|u_i\|^2 _{\mathcal{H}} + \int |\hat{\phi_i}|^2 * (|\cdot|^{2\re z} \chi (\cdot)) (\lambda) d\|E(\lambda) u_i\|^2 _{\mathcal{H}} \\
& \lesssim \|\hat{\phi_i}\|^2 _2 \|u_i\|^2 _{\mathcal{H}} + \|\hat{\phi_i}\|^2 _{\infty} \||\cdot|^{2\re z} \chi (\cdot)\|_1 \|u_i\|^2 _{\mathcal{H}} < \infty.
\end{align*}
These estimates imply that $\left[ \int_{\R} \frac{(\tau -\lambda)^{z} _+}{\Gamma (z+1)} |\phi|^2 (\lambda) dE(\lambda) \right] \mathcal{F} _t G \in L^2 _{\tau} \mathcal{H}$ and hence (\ref{2405021847}) is well-defined. (\ref{2405021848}) is similarly proved to be well-defined. Then, for $G$ and $J$ as in (\ref{2405021904}), we have
\begin{align*}
\langle {2\pi} \tilde{T} _{z, R, \tilde{R}} G, J \rangle &= \int \chi _R (\tau) e^{-ik\tau} e^{is\tau} \left\langle \frac{(\tau -H)^{z} _+}{\Gamma (z+1)} \tilde{\chi} _{\tilde{R}} (\tau - H)|\phi|^2 (H)G(k), J(s) \right\rangle dsd\tau dk \\
& = \int \chi _R (\tau) e^{-ik\tau} e^{is\tau} \int_{\R} \frac{(\tau -\lambda)^{z} _+}{\Gamma (z+1)} \tilde{\chi} _{\tilde{R}} (\tau - \lambda)|\phi|^2 (\lambda) d \langle E(\lambda) G(k), J(s) \rangle dsd\tau dk \\
&= \int \frac{\tau ^{z} _+}{\Gamma (z+1)} \tilde{\chi} _{\tilde{R}} (\tau) e^{-i(k-s)\tau} \\
& \quad \quad \quad \quad \quad \quad \cdot \left( \int e^{-i\lambda (k-s)} |\phi|^2 (\lambda) \chi _R (\tau + \lambda) d \langle E(\lambda) G(k), J(s) \rangle \right) dsd\tau dk \\
& = \int \frac{\tau ^{z} _+}{\Gamma (z+1)} \tilde{\chi} _{\tilde{R}} (\tau) e^{-i(k-s)\tau} \langle e^{-i(k-s)H} |\phi|^2 (H) \chi _R (\tau +H) G(k), J(s) \rangle dsd\tau dk.
\end{align*}
The absolute value of the integrand is estimated by $\frac{\tau ^{\re z} _+}{|\Gamma (z+1)|} \tilde{\chi} _{\tilde{R}} (\tau) \|G(k)\|_{\mathcal{H}} \|J(s)\|_{\mathcal{H}}$ and by the dominated convergence theorem we obtain
\begin{align*}
\langle {2\pi} \tilde{T} _{z, R, \tilde{R}} G, J \rangle \rightarrow \int \left( \int \frac{\tau ^{z} _+}{\Gamma (z+1)} \tilde{\chi} _{\tilde{R}} (\tau) e^{-i(k-s)\tau} d\tau \right) \langle e^{-i(k-s)H} |\phi|^2 (H) G(k), J(s) \rangle dsdk
\end{align*}
as $R \rightarrow \infty$. Next we consider the limit as $\tilde{R} \rightarrow \infty$. Since $\tilde{\chi} _{\tilde{R}} \rightarrow 1$ in $\mathcal{S}'$ holds, we have $\mathcal{F} [\tilde{\chi} _{\tilde{R}}] \rightarrow \sqrt{2\pi} \delta$ in $\mathcal{S}'$. This yields $\mathcal{F} [\frac{\tau ^{z} _+}{\Gamma (z+1)}] * \mathcal{F} [\tilde{\chi} _{\tilde{R}}] \rightarrow \sqrt{2\pi} \mathcal{F} [\frac{\tau ^{z} _+}{\Gamma (z+1)}]$ in $\mathcal{S}'$. By using the convolution inequality and $\mathcal{F} [\frac{\tau ^{z} _+}{\Gamma (z+1)}] (t) = i(e^{iz\pi /2} t^{-z-1} _{+} -e^{-iz\pi /2} t^{-z-1} _{-}) / \sqrt{2\pi}$ (see \cite{GS}), we obtain
\begin{align*}
\left\|\mathcal{F} [\frac{\tau ^{z} _+}{\Gamma (z+1)}] * \mathcal{F} [\tilde{\chi} _{\tilde{R}}] \right\|_{L^{\frac{1}{1+ \re z}, \infty}} \lesssim \left\|\mathcal{F} [\frac{\tau ^{z} _+}{\Gamma (z+1)}] \right\|_{L^{\frac{1}{1+ \re z}, \infty}} \|\mathcal{F} [\tilde{\chi} _{\tilde{R}}]\|_1 \underset{\im z}{\lesssim} 1
\end{align*}
uniformly in $\tilde{R}$. Hence $\{\mathcal{F} [\frac{\tau ^{z} _+}{\Gamma (z+1)}] * \mathcal{F} [\tilde{\chi} _{\tilde{R}}]\}_{\tilde{R}}$ is bounded in $L^{\frac{1}{1+ \re z}, \infty}$, which is a dual of $L^{-\frac{1}{\re z}, 1}$. Then there exists $\{\tilde{R} _n\}$ satisfying $\tilde{R} _n \rightarrow \infty$ and $\mathcal{F} [\frac{\tau ^{z} _+}{\Gamma (z+1)}] * \mathcal{F} [\tilde{\chi} _{\tilde{R} _n}] \rightharpoonup \sqrt{2\pi} \mathcal{F} [\frac{\tau ^{z} _+}{\Gamma (z+1)}]$ in $L^{\frac{1}{1+ \re z}, \infty}$ as $n \rightarrow \infty$, where $\rightharpoonup$ implies the weak-$*$ convergence. This yields
\begin{align*}
&\left\langle \sqrt{2\pi} \int \left( \widehat{ \frac{\tau ^{z} _+}{\Gamma (z+1)} \tilde{\chi}_{\tilde{R}_n} } \right) (k-s) e^{-ikH} |\phi|^2 (H) G(k) dk, e^{-isH} J(s) \right\rangle \\
& = \sqrt{2\pi} \int  \left( \widehat{ \frac{\tau ^{z} _+}{\Gamma (z+1)} \tilde{\chi}_{\tilde{R}_n} } \right) (k-s) \langle e^{-ikH} |\phi|^2 (H) G(k), e^{-isH} J(s) \rangle dk \\
& \underset{n \rightarrow \infty}{\longrightarrow} \sqrt{2\pi} \int \mathcal{F} [\frac{\tau ^{z} _+}{\Gamma (z+1)}] (k-s) \langle e^{-ikH} |\phi|^2 (H) G(k), e^{-isH} J(s) \rangle dk
\end{align*}
for all $s \in \R$ since
\begin{align*}
|\langle e^{-ikH} |\phi|^2 (H) G(k), e^{-isH} J(s) \rangle| \lesssim \|G(k)\|_{\mathcal{H}} \|J(s)\|_{\mathcal{H}} \in L^{-\frac{1}{\re z}, 1}
\end{align*}
for all $s \in \R$. Furthermore we have
\begin{align*}
&\left| \left\langle \sqrt{2\pi} \int \left( \widehat{ \frac{\tau ^{z} _+}{\Gamma (z+1)} \tilde{\chi}_{\tilde{R}_n} } \right) (k-s) e^{-ikH} |\phi|^2 (H) G(k) dk, e^{-isH} J(s) \right\rangle \right| \\
& \lesssim \|J(s)\|_{\mathcal{H}} \left\| \left( \widehat{ \frac{\tau ^{z} _+}{\Gamma (z+1)} \tilde{\chi}_{\tilde{R}_n} } \right) (- \cdot) * (e^{-ikH} |\phi|^2 (H) G(k)) \right\|_{L^{\infty} _s \mathcal{H}} \\
& \lesssim \|J(s)\|_{\mathcal{H}} \|G\|_{L^{-\frac{1}{\re z}, 1}} \left\|\mathcal{F} [\frac{\tau ^{z} _+}{\Gamma (z+1)}] * \mathcal{F} [\tilde{\chi} _{\tilde{R}_n}] \right\|_{L^{\frac{1}{1+ \re z}, \infty}} \\
& \underset{\im z}{\lesssim} \|J(s)\|_{\mathcal{H}} \|G\|_{L^{-\frac{1}{\re z}, 1}} \in L^1 _s \mathcal{H}
\end{align*}
uniformly in $n \in \N$ by the above calculations. Then the dominated convergence theorem yields
\begin{align*}
\lim_{n \to \infty} (\lim_{R \to \infty} \langle {2\pi} \tilde{T} _{z, R, \tilde{R}_n} G, J \rangle) &= \sqrt{2\pi} \int \int \mathcal{F} [\frac{\tau ^{z} _+}{\Gamma (z+1)}] \langle e^{-ikH} |\phi|^2 (H) G(k), e^{-isH} J(s) \rangle dkds \\
& = \langle T_z G, J \rangle .
\end{align*}
On the other hand we have
\begin{align*}
\langle \tilde{T} _{z, R, \tilde{R}_n} G, J \rangle &= \left\langle \tilde{\chi} _{\tilde{R}_n} (\tau - H) \frac{(\tau -H) ^{z} _+}{\Gamma (z+1)} |\phi|^2 (H) \mathcal{F}_t G, \chi _R (\tau) \mathcal{F}_t J \right\rangle \\
& \underset{R \rightarrow \infty}{\longrightarrow} \left\langle \tilde{\chi} _{\tilde{R}_n} (\tau - H) \frac{(\tau -H) ^{z} _+}{\Gamma (z+1)} |\phi|^2 (H) \mathcal{F}_t G, \mathcal{F}_t J \right\rangle \\
& \underset{n \rightarrow \infty}{\longrightarrow} \left\langle \frac{(\tau -H) ^{z} _+}{\Gamma (z+1)} |\phi|^2 (H) \mathcal{F}_t G, \mathcal{F}_t J \right\rangle = \langle \tilde{T}_z G, J \rangle 
\end{align*}
since $\mathcal{F}_t J$, $\frac{(\tau -H) ^{z} _+}{\Gamma (z+1)} |\phi|^2 (H) \mathcal{F}_t G \in L^2 _{\tau} \mathcal{H}$. As a consequence $\langle 2\pi \tilde{T}_z G, J \rangle =  \langle T_z G, J \rangle$ holds for all $G, J$ as in (\ref{2405021904}). Moreover we have
\begin{align}
| \lim_{R \to \infty} \langle {2\pi} \tilde{T} _{z, R, \tilde{R}} G, J \rangle |  \underset{\im z}{\lesssim} \|J\|_{L^1 _s \mathcal{H}} \|G\|_{L^{-\frac{1}{\re z}, 1}}, \quad |\langle T_z G, J \rangle | \underset{\im z}{\lesssim} \|J\|_{L^1 _s \mathcal{H}} \|G\|_{L^{-\frac{1}{\re z}, 1}} \label{2405031333}
\end{align}
by the above argument, where the latter estimate can be proved similarly as the former one.

(Step 3) In this step we prove that $\langle T_z F, G \rangle$ extends to a continuous function in $\{ z \in \C \mid \re z \le 0\}$ for all simple functions $F$ and $G$. We take arbitrary space-time simple functions $F$ and $G$. Let $\{F_n\}$ satisfy (\ref{2405021904}) and $\|F-F_n\|_{L^1 _t \mathcal{H} \cap L^2 _t \mathcal{H} \cap L^{-\frac{1}{\re z}, 1} _t \mathcal{H} } \rightarrow 0$ as $n \rightarrow \infty$ (We may take $F_n$ in the same way as $G_n$ below). We define $\{G_n\}$ by 
\begin{align*}
G_n := G* \rho_{\frac{1}{n}} = \sum_{i \in I} (\rho_{\frac{1}{n}} * \chi_i) u_i \quad for \quad G= \sum_{i \in I} \chi_i u_i
\end{align*}
where $u_i$ is a simple function on $X$ and $\chi_i$ are characteristic functions of some disjoint intervals of $\R$. In the above definition, $\{\rho_{\frac{1}{n}} \}$ denotes a mollifier of $C^{\infty} _0$ functions. Hence for sufficiently large $n \in \N$, $G_n$ satisfies conditions in (\ref{2405021904}) by a support property of $\rho$. Furthermore we have $\|G-G_n\|_{L^1 _t \mathcal{H} \cap L^2 _t \mathcal{H} } \rightarrow 0$ as $n \rightarrow \infty$. By using these properties and the fact that $\tilde{T}_{ik}$ is bounded on $L^2 _t \mathcal{H}$, for $z \in \{z \in \C \mid -\frac{1}{10} <\re z <0\}$ and $k \in \R$, we have
\begin{align*}
&|\langle 2\pi \tilde{T}_{ik} F, G \rangle - \langle T_z F, G \rangle | \\
& \lesssim \varlimsup_{n \to \infty} |\langle 2\pi \tilde{T}_{ik} F, G-G_n \rangle| + \varlimsup_{n \to \infty} | \langle 2\pi \tilde{T}_{ik} F, G_n \rangle - \langle 2\pi \tilde{T}_{ik} F_n, G_n \rangle | \\
&  \quad + \varlimsup_{n \to \infty} | \langle 2\pi \tilde{T}_{ik} F_n, G_n \rangle - \langle 2\pi \tilde{T}_{z} F_n, G_n \rangle | + \varlimsup_{n \to \infty} | \langle 2\pi \tilde{T}_{z} F_n, G_n \rangle - \langle T_z F_n, G_n \rangle | \\
& \quad + \varlimsup_{n \to \infty} | \langle T_z F_n, G_n \rangle - \langle T_z F_n, G \rangle | + \varlimsup_{n \to \infty} | \langle T_z F_n, G \rangle - \langle T_z F, G \rangle | \\
& \lesssim \|F\|_{L^2 _t \mathcal{H}} \varlimsup_{n \to \infty} \|G-G_n\|_{L^2 _t \mathcal{H}} + \varlimsup_{n \to \infty} \|G_n\|_{L^2 _t \mathcal{H}} \|F-F_n\|_{L^2 _t \mathcal{H}} \\
& \quad + \varlimsup_{n \to \infty} \|F_n\|_{L^2 _t \mathcal{H}} \|(\tilde{T}_{ik} - \tilde{T}_z)^* G_n\|_{L^2 _t \mathcal{H}} + \varlimsup_{n \to \infty} \|F_n\|_{L^{-\frac{1}{\re z}, 1} _t \mathcal{H} } \|G- G_n\|_{L^1 _t \mathcal{H}} \\
& \quad + \varlimsup_{n \to \infty} \|G\|_{L^1 _t \mathcal{H}} \|F_n -F\|_{L^{-\frac{1}{\re z}, 1} _t \mathcal{H} } \\
& = \varlimsup_{n \to \infty} \|F_n\|_{L^2 _t \mathcal{H}} \|(\tilde{T}_{ik} - \tilde{T}_z)^* G_n\|_{L^2 _t \mathcal{H}} \lesssim \sup_{n \in \N} \|(\tilde{T}_{ik} - \tilde{T}_z)^* G_n\|_{L^2 _t \mathcal{H}}.
\end{align*}
Now we estimate the right hand side assuming $z \to ik$. By (\ref{2405021848}) we obtain
\begin{align*}
\|(\tilde{T}_{ik} - \tilde{T}_z)^* G_n\|^2 _{L^2 _t \mathcal{H}} &\lesssim \left\| \left( \frac{(\tau -H)^{z} _+}{\Gamma (z+1)} - \frac{(\tau -H)^{ik} _+}{\Gamma (ik+1)} \right)^* \mathcal{F} G_n \right\|^2 _{L^2 _t \mathcal{H}} \\
& \lesssim \sum_{i \in I} \left\| \widehat{\rho_{\frac{1}{n}} * \chi_i} \left( \frac{(\tau -H)^{z} _+}{\Gamma (z+1)} - \frac{(\tau -H)^{ik} _+}{\Gamma (ik+1)} \right)^* u_i \right\|^2 _{L^2 _t \mathcal{H}} \\
& \lesssim \sum_{i \in I} \int |\widehat{\rho_{\frac{1}{n}} * \chi_i} (\tau)|^2 \int \left| \frac{(\tau -\lambda)^{z} _+}{\Gamma (z+1)} - \frac{(\tau -\lambda)^{ik} _+}{\Gamma (ik+1)} \right|^2 d \|E(\lambda) u_i\|^2 _{\mathcal{H}} d \tau .
\end{align*}
For each term we sprit $\R ^2$ into $\{|\tau - \lambda| < 1\}$ and $\{|\tau - \lambda| \ge 1\}$. In the first region we have
\begin{align*}
& \int \int_{|\tau - \lambda| < 1} |\widehat{\rho_{\frac{1}{n}} * \chi_i} (\tau)|^2 \left| \frac{(\tau -\lambda)^{z} _+}{\Gamma (z+1)} - \frac{(\tau -\lambda)^{ik} _+}{\Gamma (ik+1)} \right|^2 d \|E(\lambda) u_i\|^2 _{\mathcal{H}} d \tau \\
& = \int |\widehat{\rho_{\frac{1}{n}} * \chi_i} ( \cdot )|^2 * \left( \left| \frac{( \cdot )^{z} _-}{\Gamma (z+1)} - \frac{( \cdot )^{ik} _-}{\Gamma (ik+1)} \right|^2 \chi ( \cdot ) \right) (\lambda) d \|E(\lambda) u_i\|^2 _{\mathcal{H}} \\
& \lesssim \|u_i\|^2 _{\mathcal{H}} \|\widehat{\rho_{\frac{1}{n}} * \chi_i }\|^2 _{\infty} \left\| \left|\frac{ \tau ^{z} _-}{\Gamma (z+1)} - \frac{ \tau ^{ik} _-}{\Gamma (ik+1)} \right|^2 \chi ( \tau ) \right\|_{1} \\
& \lesssim \|u_i\|^2 _{\mathcal{H}} \|\chi_i\|^2 _{1} \left\| \left|\frac{ \tau ^{z} _-}{\Gamma (z+1)} - \frac{ \tau ^{ik} _-}{\Gamma (ik+1)} \right|^2 \chi ( \tau ) \right\|_{1} \underset{z \to ik}{\longrightarrow} 0
\end{align*}
uniformly in $n \in \N$, where in the last limit we have used the dominated convergence theorem. In the second region we transform into
\begin{align*}
&\underbrace{ \int \int_{|\tau - \lambda| \ge 1} (|\widehat{\rho_{\frac{1}{n}} * \chi_i} (\tau) |^2 - |\hat{\chi_i} (\tau)|^2) \left| \frac{(\tau -\lambda)^{z} _+}{\Gamma (z+1)} - \frac{(\tau -\lambda)^{ik} _+}{\Gamma (ik+1)} \right|^2 d \|E(\lambda) u_i\|^2 _{\mathcal{H}} d \tau }_{=: I_1} \\
& + \underbrace{ \int \int_{|\tau - \lambda| \ge 1}  |\hat{\chi_i} (\tau)|^2 \left| \frac{(\tau -\lambda)^{z} _+}{\Gamma (z+1)} - \frac{(\tau -\lambda)^{ik} _+}{\Gamma (ik+1)} \right|^2 d \|E(\lambda) u_i\|^2 _{\mathcal{H}} d \tau }_{=: I_2}.
\end{align*}
For any $\epsilon >0$ we take $n_0 \in \N$ satisfying $n \ge n_0 \Rightarrow \| \widehat{\rho_{\frac{1}{n}} * \chi_i} - \hat{\chi_i} \|_2 = \| \rho_{\frac{1}{n}} * \chi_i - \chi_i\|_2 < \epsilon$. Then we have
\begin{align*}
|I_1| \lesssim \|u_i\|^2 _{\mathcal{H}} \||\widehat{\rho_{\frac{1}{n}} * \chi_i}|^2 - |\hat{\chi_i}|^2\|_1 \lesssim \|u_i\|^2 _{\mathcal{H}} (\| \widehat{\rho_{\frac{1}{n}} * \chi_i} \|_2 + \| \hat{\chi_i} \|_2)\| \widehat{\rho_{\frac{1}{n}} * \chi_i} - \hat{\chi_i} \|_2 \lesssim \epsilon
\end{align*}
uniformly in $n \ge n_0$ and $z$ near $ik$. For $I_2$, by the dominated convergence theorem, there exists $r_0 >0$ such that if $|z - ik| <r_0$ then $|I_2| < \epsilon$. This implies that if $|z - ik| <r_0$ and $n \ge n_0$ we have $|I_1| + |I_2| \lesssim \epsilon$. For $n \le n_0$, again by the dominated convergence theorem, there exists $r_1 < r_0$ such that if $|z-ik| < r_1$ then $|I_1| + |I_2| \lesssim \epsilon$. Therefore, for any $\epsilon >0$, we have
\begin{align*}
\sup_{n \in \N} \|(\tilde{T}_{ik} - \tilde{T}_z)^* G_n\|^2 _{L^2 _t \mathcal{H}} < \epsilon
\end{align*}
if $|z-ik|$ is sufficiently small. By the above argument we have $\lim_{z \to ik} |\langle 2\pi \tilde{T}_{ik} F, G \rangle - \langle T_z F, G \rangle | =0$, where $z$ moves in $\{ z \in \C \mid \re z <0\}$. Moreover it is easier to prove that $\langle \tilde{T} _{is} F, G \rangle \to \langle \tilde{T} _{ik} F, G \rangle$ if $ \R \ni s \to k$. As a result, by setting $T_{ik} = 2\pi \tilde{T}_{ik}$ for $k \in \R$, $\{T_z \}$ is an analytic family of operators in the sense of Stein (i.e.   for any simple functions $F$ and $G$, $\langle T_z F, G \rangle$ is analytic in $\{ z \in \C \mid \re z <0\}$ and continuous in $\{ z \in \C \mid \re z \le 0\}$). We note that 
\begin{align}
\|W_1 T_{ik} W_2 \|_{\mathfrak{S}^{\infty}} \lesssim \|W_1\|_{L^{\infty} _t L^{\infty} _x} \|W_2\|_{L^{\infty} _t L^{\infty} _x} \label{2405041543}
\end{align}
holds for any simple functions $W_1$ and $W_2$ from (\ref{2405021848}).

(Step 4) Now we define $S_z := {W_1}^{-z} T_z {W_2}^{-z}$ for non-negative simple functions $W_1, W_2$ and $z \in \{ z \in \C \mid \re z \le 0\}$. Then $\langle S_z F, G \rangle$ is analytic in $\{ z \in \C \mid \re z <0\}$ for any simple functions $F$ and $G$, which can be seen as the analyticity of $\langle T_z F, G \rangle$. To see the continuity near $\re z =0$, we calculate as
\begin{align*}
| \langle S_z F, G \rangle - \langle S_{ik} F, G \rangle| &\le | \langle T_{z} {W_2}^{-z} F, {W_1}^{-\bar{z}} G \rangle - \langle T_z {W_2}^{-ik} F, {W_1}^{-\bar{z}} G \rangle | \\
& \quad + |\langle T_z {W_2}^{-ik} F, {W_1}^{-\bar{z}} G \rangle - \langle T_z {W_2}^{-ik} F, {W_1}^{ik} G \rangle |  \\
& \quad + | \langle T_z {W_2}^{-ik} F, {W_1}^{ik} G \rangle - \langle T_{ik} {W_2}^{-ik} F, {W_1}^{ik} G \rangle|.
\end{align*}
The first term is estimated from above by $\|{W_1}^{-\bar{z}} G\|_{L^2 _t \mathcal{H}} \|T_{z} \{({W_2}^{-z} - {W_2}^{-ik}) F\}\|_{L^2 _t \mathcal{H}}$. By setting $W_2 = \sum_{i \in I} \chi_{i} (t) u_i$, $F = \sum_{j \in J} \tilde{\chi}_j (t) \tilde{u}_j$ and repeating a similar calculation as the proof of the well-definedness of (\ref{2405021847}), we obtain
\begin{align*}
\|T_{z} \{({W_2}^{-z} - {W_2}^{-ik}) F\}\|^2 _{L^2 _t \mathcal{H}} \lesssim \sum_{i, j} \|(u^{-z} _{i} - u^{-ik} _i)\tilde{u}_j\|^2 _{\mathcal{H}}
\end{align*}
where the implicit constant is dependent on $\chi_i$ and $\tilde{\chi}_j$ but independent of $z$. Hence the first term goes to $0$ as $z \to ik$ since $u_i$ and $\tilde{u}_j$ are simple functions on $X$. The second term also goes to $0$ by a similar calculation and $\| {W_1}^{-\bar{z}} G - {W_1}^{ik} G \|_{L^2 _t \mathcal{H}} \to 0$. Finally by using the continuity of $T_z$, the third term disappears as $z \to ik$. Therefore $\{S_z\}$ is an analytic family of operators in the sense of Stein. Furthermore by using (\ref{2405021809}) and (\ref{2405041543}) we obtain
\begin{align*}
\|S_{ik}\|_{\mathfrak{S}^{\infty}} \lesssim 1, \quad \|S_{- \frac{\tilde{r}}{2} +ik}\|_{\mathfrak{S}^{2}} \lesssim \|W^{\frac{\tilde{r}}{2}} _1\|_{L^{2u, 4} _t \mathcal{H}} \|W^{\frac {\tilde{r}}{2}} _2\|_{L^{2u, 4} _t \mathcal{H}} = \|W_1\|^{\frac{\tilde{r}}{2}} _{L^{u \tilde{r}, 2 \tilde{r}} _t L^{\tilde{r}} _x} \|W_2\|^{\frac{\tilde{r}}{2}} _{L^{u \tilde{r}, 2 \tilde{r}} _t L^{\tilde{r}} _x}
\end{align*}
where the implicit constant is at most exponential.
By the complex interpolation we have
\begin{align*}
\|S_{-1}\|_{\mathfrak{S}^{\tilde{r}}} \lesssim \|W_1\|_{L^{u \tilde{r}, 2 \tilde{r}} _t L^{\tilde{r}} _x} \|W_2\|_{L^{u \tilde{r}, 2 \tilde{r}} _t L^{\tilde{r}} _x}.
\end{align*}
This is equivalent to 
\begin{align}
\|W_1 T_{-1} W_2\|_{\mathfrak{S}^{\tilde{r}}} \lesssim \|W_1\|_{L^{\tilde{q}, 2 \tilde{r}} _t L^{\tilde{r}} _x} \|W_2\|_{L^{\tilde{q}, 2 \tilde{r}} _t L^{\tilde{r}} _x} \label{2405042005}
\end{align}
since $\tilde{q}= u \tilde{r}$ holds. Actually (\ref{2405042005}) is true for all simple functions $W_1$ and $W_2$ by considering $W_j = \sgn W_j |W_j|$. Then by Lemma \ref{2405031343}, (\ref{2405042010}) and (\ref{2405042005}) we obtain 
\begin{align*}
\left\| \sum_{j=0}^{\infty} \nu_j |e^{-itH} \phi (H) f_j|^2 \right\|_{L^{\frac{q}{2}, \beta} _t  L^{\frac{r}{2}} _x} \lesssim \|\nu\|_{l^{\beta}}
\end{align*}
since $\beta = \frac{2r}{r+2}$. Thus we are done.
\end{proof}
Now we are ready to prove Theorem \ref{2405021538} and Corollary \ref{2405021556}. The proof is based on simple interpolation arguments.
\begin{proof}[Proof of Theorem \ref{2405021538}]
First we assume $\sigma \in (\frac{1}{2}, \infty)$. If $r \in (\frac{2(\sigma +1)}{\sigma}, \frac{2(2\sigma +1)}{2\sigma -1})$, we have $\max \{1+2\sigma, 2\} < \tilde{r} <2+2\sigma$ since $\tilde{r} =2(\frac{r}{2})'$ holds. Furthermore, for $\beta = \frac{2r}{r+2}$, we have $\beta \le \frac{q}{2}$ since $(q, r)$ is $\sigma$-admissible. Then by Proposition \ref{2405021420} and the inclusion of the Lorentz spaces, we obtain
\begin{align}
\left\| \sum_{j=0}^{\infty} \nu_j |e^{-itH} \phi (H) f_j|^2 \right\|_{L^{\frac{q}{2}} _t  L^{\frac{r}{2}} _x} \lesssim \left\| \sum_{j=0}^{\infty} \nu_j |e^{-itH} \phi (H) f_j|^2 \right\|_{L^{\frac{q}{2}, \beta} _t  L^{\frac{r}{2}} _x} \lesssim \|\nu\|_{l^{\beta}}. \label{2405050900}
\end{align}
By the complex interpolation between (\ref{2405050900}) and the trivial one (i.e. $(q, r, \beta) = (\infty, 2, 1)$) we obtain (\ref{2405021614}) for all $r \in [2, \frac{2(2\sigma +1)}{2\sigma -1})$. If $\sigma >1$, by the endpoint Strichartz estimate (\cite{KT}) and the triangle inequality, we have (\ref{2405021614}) for $(q, r, \beta) = (2, \frac{2\sigma}{\sigma -1}, 1)$ (see (\ref{2405051024})). By interpolating (\ref{2405021614}) for $(2, \frac{2\sigma}{\sigma -1}, 1)$ and $(q, r, \beta)$, where $r < \frac{2(2\sigma +1)}{2\sigma -1}$ is sufficiently close to $\frac{2(2\sigma +1)}{2\sigma -1}$, we obtain (\ref{2405021614}) for all $r \in [2, \frac{2\sigma}{\sigma -1})$. If $\sigma =1$, instead of the endpoint estimate, we employ $(q, r)$ sufficiently close to $(2, \infty)$ and gain the same result. Finally in the case of $\sigma \in (\frac{1}{2}, 1)$, we use (\ref{2405021614}) for $(q, r, \beta) = (\frac{2}{\sigma}, \infty, 1)$ and repeating a similar argument, we obtain the desired result. Next we assume $\sigma \in (0, \frac{1}{2}]$. If $r \in (\frac{2(\sigma +1)}{\sigma}, \infty)$ and $\beta = \frac{2r}{r+2}$, then $\beta \le \frac{q}{2}$ and assumptions in Proposition \ref{2405021420} hold. Hence we have (\ref{2405050900}) for $r \in (\frac{2(\sigma +1)}{\sigma}, \infty)$. By the complex interpolation with the trivial estimate, we finish the proof.      
\end{proof}
We finish this subsection by proving Corollary \ref{2405021556}. However the arguments are almost the same.
\begin{proof}[Proof of Corollary \ref{2405021556}]
Let $I = [a, b]$ and $-\infty \le a < b \le \infty$. Let $\chi_{I} \in C^{\infty} _{0} (\R)$ denote a function satisfying $\chi_{I} (t) =1$ if $t \in [a+\epsilon, b-\epsilon]$ and $\chi_{I} (t) = 0$ if $t \in (-\infty, a + \frac{\epsilon}{2} ) \cup (b-\frac{\epsilon}{2}, \infty)$ for sufficiently small $\epsilon >0$. As in the proof of Proposition \ref{2405021420}, we consider $\chi_{I} (t) T_z \chi_{I} (t)$ and $\chi_{I} (t) \tilde{T} _z \chi_{I} (t)$, where $T_z$ and $\tilde{T} _z$ are as in (\ref{2405051128}) or (\ref{2405021847}) respectively. Then by repeating the proof of Proposition \ref{2405021420},
\begin{align*}
\left\| \sum_{j=0}^{\infty} \nu_j |e^{-itH} \phi (H) f_j|^2 |\chi_{I} (t)|^2 \right\|_{L^{\frac{q}{2}, \beta} _t  L^{\frac{r}{2}} _x} \lesssim \|\nu\|_{l^{\beta}}
\end{align*}
holds, where the implicit constant does not depend on $\chi_{I}$. By taking $\epsilon \to 0$ and the monotone convergence theorem we obtain
\begin{align*}
\left\| \sum_{j=0}^{\infty} \nu_j |e^{-itH} \phi (H) f_j|^2 \right\|_{L^{\frac{q}{2}, \beta} (I; L^{\frac{r}{2}} _x)} \lesssim \|\nu\|_{l^{\beta}}.
\end{align*}
Finally we obtain the desired estimates by repeating the proof of Theorem \ref{2405021538}.
\end{proof}
\subsection{\textbf{Applications to some operators}}
In this subsection we give some applications of Theorem \ref{2405021538} and Corollary \ref{2405021556}. These examples do not follow from the perturbation theory in \cite{H1} and \cite{H2} due to the lack of the Kato smoothing estimate. First we consider the magnetic Schr\"odinger operator with unbounded electromagnetic potentials and partially improve the known results in \cite{Y2}.
\begin{example}\label{2405051203}
Let $d \ge 1$ and $H = (D_{x} -A)^2 +V$ be the magnetic Schr\"odinger operator on $L^2 (\R^d)$ satisfying the following conditions:
\begin{enumerate}
\item
Electric potential $V: \R^d \to \R$ is a smooth function satisfying $|\partial^{\alpha} _x V(x)| \lesssim 1$ if $|\alpha| \ge 2$.
\item
Magnetic potential $A: \R^d \to \R^d$ is smooth and satisfies $|\partial^{\alpha} _x A(x)| \lesssim 1$ if $|\alpha| \ge 1$. 
\item
Magnetic field $B:= (\frac{\partial A_k}{\partial x_j} - \frac{\partial A_j}{\partial x_k})_{j, k}$ satisfies $|\partial^{\alpha} _x B(x)| \lesssim \langle x \rangle^{-1-\epsilon}$ if $|\alpha| \ge 1$ for some $\epsilon >0$. 
\end{enumerate}
Then it is well-known that $H$ is essentially self-adjoint on $C^{\infty} _{0} (\R^d)$. Furthermore Yajima (\cite{Y2} and \cite{Y3}) proved the following expression of the propagator:
\begin{align}
e^{-i(t-s)H} u= \int_{\R^d} E(t, s, x, y) u(y) dy, \quad E(t, s, x, y) = \frac{e^{\mp id\pi /4}}{(2\pi |t-s|)^{\frac{d}{2}}} e^{iS(t, s, x, y)} e(t, s, x, y) \label{2405060941}
\end{align}
for all $u \in \mathcal{S}$ and $0 < \pm (t-s) < T \ll 1$. Here $S$ is a real-valued function solving the Hamilton-Jacobi equation and $e$ satisfies $|e(t, s, x, y) -1| \lesssim |t-s|$. Therefore by taking sufficiently small $T >0$,
\begin{align*}
\|e^{-i(t-s)H} u\|_{\infty} \lesssim |t-s|^{-\frac{d}{2}} \|u\|_{1}
\end{align*} 
holds for $t \ne s$ satisfying $|t|, |s| < T$. By applying Corollary \ref{2405021556} we obtain
\begin{align*}
\left\| \sum_{j=0}^{\infty} \nu_j |e^{-itH} f_j|^2 \right\|_{L^{\frac{q}{2}} _T  L^{\frac{r}{2}} _x} \lesssim \|\nu\|_{l^{\beta}}
\end{align*}
for any $\frac{d}{2}$-admissible pair $(q, r)$ and $\beta$ satisfying the assumption in Theorem \ref{2405021538} with $\sigma = \frac{d}{2}$. Here $L^{\frac{q}{2}} _T  L^{\frac{r}{2}} _x$ denotes $L^{\frac{q}{2}} ( [-T, T]; L^{\frac{r}{2}} _x)$. For any $N \in \N$ we calculate
\begin{align*}
\left\| \sum_{j=0}^{\infty} \nu_j |e^{-itH} f_j|^2 \right\|_{L^{\frac{q}{2}} _{NT}  L^{\frac{r}{2}} _x} \lesssim \sum_{k=0}^{4N-1} \left\| \sum_{j=0}^{\infty} \nu_j |e^{-itH} f_j|^2 \right\|_{L^{\frac{q}{2}} ( [-NT+ \frac{k}{2} T, -NT + \frac{k+1}{2} T]; L^{\frac{r}{2}} _x)}
\end{align*}
by the triangle inequality. Since $\{ e^{-i(-NT+ \frac{k}{2} T)H} f_j \}_j$ is an orthonormal system in $L^2 (\R^d)$ for all $N$ and $k$, we have
\begin{align*}
\left\| \sum_{j=0}^{\infty} \nu_j |e^{-itH} f_j|^2 \right\|_{L^{\frac{q}{2}} ( [-NT+ \frac{k}{2} T, -NT + \frac{k+1}{2} T]; L^{\frac{r}{2}} _x)} &= \left\| \sum_{j=0}^{\infty} \nu_j |e^{-itH} (e^{-i(-NT+ \frac{k}{2} T)H} f_j) |^2 \right\|_{L^{\frac{q}{2}} ( [0, \frac{T}{2}]; L^{\frac{r}{2}} _x)} \\
& \lesssim \|\nu\|_{l^{\beta}}.
\end{align*} 
Hence we have proved the following corollary.
\end{example}
\begin{corollary}\label{2405061029}
Let $H$ be the magnetic Schr\"odinger operator introduced in Example \ref{2405051203}. For any bounded interval $I \subset \R$, we have
\begin{align*}
\left\| \sum_{j=0}^{\infty} \nu_j |e^{-itH} f_j|^2 \right\|_{L^{\frac{q}{2}} (I; L^{\frac{r}{2}} _x )} \lesssim \|\nu\|_{l^{\beta}}
\end{align*}
for any $(q, r, \beta)$ satisfying the assumption in Theorem \ref{2405021538} with $\sigma = \frac{d}{2}$.
\end{corollary}
We notice that, in Corollary \ref{2405061029}, the implicit constant is dependent on the volume of $I$. Actually it is impossible to take $I = \R$ in general. As a counterexample we have $H= -\Delta + x^2$, which has infinitely many eigenfunctions. Next we consider the $(k, a)$-generalized Laguerre operator. Recently the ordinary and orthonormal Strichartz estimates are considered in \cite{Ms}, \cite{MS} and \cite{TT} but the latter are restricted to $a=1$ or $2$ (see \cite{MS}).
\begin{example}\label{2405061101}
We recall the definition of  the $(k, a)$-generalized Laguerre operator, which was introduced in \cite{BKO}. Let $\mathcal{R} \subset \R^d \setminus \{0\}$ be a reduced root system. The finite Coxeter group associated with $\mathcal{R}$ is a subgroup $C \subset O(d)$ generated by $\{ r_{\alpha}\}_{\alpha \in \mathcal{R}}$, where $r_{\alpha} (x):= x- 2 \frac{x \cdot \alpha}{|\alpha|^2} \alpha$ denotes the reflection. If $k: \mathcal{R} \to \C$ is invariant under the finite Coxeter group associated with $\mathcal{R}$, $k$ is called a multiplicity function. For $\xi \in \C^d$, a non-negative multiplicity function $k$ and $f \in C^1 (\R^d)$, we define
\begin{align*}
T_{\xi} (k) f(x) := \partial_{\xi} f(x) + \sum_{\alpha \in \mathcal{R}^{+}} k(\alpha) \alpha \cdot \xi \frac{f(x) - f(r_{\alpha} x)}{\alpha \cdot x}
\end{align*} 
where $\mathcal{R}^{+}$ is a positive subsystem of $\mathcal{R}$ and $\partial_{\xi}$ denotes the directional derivative for $\xi$. It is known that this definition is independent of the choice of $\mathcal{R}^{+}$ (see \cite{BKO} pp.1275). The Dunkl Laplacian $\Delta_{k}$ is defined by $\Delta_{k} := \sum_{j=1}^{d} \{ T_{\xi_{j}} (k)\}^2$, where $\xi_{1}, \ldots, \xi_{d}$ is an orthonormal basis of $\R^d$. This is also independent of the choice of an orthonormal basis (see \cite{BKO} pp.1278). Then the $(k, a)$-generalized Laguerre operator $H_{k, a}$ is defined by
\begin{align*}
H_{k, a} := \frac{-|x|^{2-a} \Delta_{k} + |x|^{a}}{a}
\end{align*} 
for $a >0$. We set $\langle k \rangle := \frac{1}{2} \sum_{\alpha \in \mathcal{R}^{+}} k(\alpha)$, $\sigma_{k, a} := \frac{a+2 \langle k \rangle +d-2}{a}$ and $\vartheta_{k, a} (x):= |x|^{a-2} \Pi_{\alpha \in \mathcal{R}^{+}} |\alpha \cdot x|^{2k(\alpha)}$. It is proved in \cite{BKO} Corollary 3.22 that if $\sigma_{k, a} >0$ then $H_{k, a}$ is essentially self-adjoint on $L^2 (\R^d, \vartheta_{k, a} (x)dx)$ and its spectrum is discrete. Recently Taira and Tamori (\cite{TT}) proved that if $\sigma_{k, a} >0$ and either of the following holds:
\begin{enumerate}
\item
$d=1$ and $a \ge 2- 4 \langle k \rangle$
\item
$d \ge 2$ and $0<a \le 1$ or $d \ge 2$ and $a=2$
\end{enumerate}
then the dispersive estimate holds: $\|e^{-itH_{k, a}} \|_{1 \to \infty} \lesssim |t|^{-\sigma_{k, a}}$ for $|t| \le \frac{\pi}{2}$. Then by using Corollary \ref{2405021556} and repeating a similar argument as in Example \ref{2405051203}, we obtain 
\begin{align}
\left\| \sum_{j=0}^{\infty} \nu_j |e^{-itH_{k, a}} f_j|^2 \right\|_{L^{\frac{q}{2}} (I; L^{\frac{r}{2}} _x )} \lesssim \|\nu\|_{l^{\beta}}\label{2405072143}
\end{align}
where $I \subset \R$ is an arbitrary bounded interval and $(q, r, \beta )$ satisfies the condition in Theorem \ref{2405021538} with $\sigma = \sigma_{k, a}$. We notice that $L^p _x = L^p (\R^d, \vartheta_{k, a} (x)dx)$. Next we assume $d \ge 2$, $k \equiv 0$ and $a \in (1, 2)$. We take $\phi_0 \in (0, 2\pi)$ and a finite partition of unity $\{ \chi_{i} \}_{1\le i \le N}$ for $\mathbb{S}^{d-1}$ satisfying $\omega, \eta \in \supp \chi_i \Rightarrow  \omega \cdot \eta \ge \cos \phi_0$. It is proved in \cite{TT} that 
\begin{align*}
\|\chi_i (\hat{x}) e^{-i(t-s)H_{k, a}} \chi_i (\hat{x})\|_{1 \to \infty} \lesssim |t-s|^{-\sigma_{k, a}}
\end{align*}
holds for $t, s \in (-\frac{\pi}{4}, \frac{\pi}{4})$, where $\hat{x} = \frac{x}{|x|}$. Then by repeating the proof of Proposition \ref{2405021420} and Corollary \ref{2405021556} with $\phi \equiv 1$ and $\{T_z \}$ replaced by $\{ \chi_i (\hat{x}) \chi_{I} (t)T_z \chi_{I} (t) \chi_i (\hat{x}) \}$, we obtain
\begin{align*}
\left\| \sum_{j=0}^{\infty} \nu_j | \chi_i (\hat{x}) e^{-itH_{k, a}} f_j|^2 \right\|_{L^{\frac{q}{2}} (I; L^{\frac{r}{2}} _x )} \lesssim \|\nu\|_{l^{\beta}}
\end{align*}
where $I, (q, r, \beta)$ are as above. By taking a sum with respect to $i$, we obtain (\ref{2405072143}). We notice that these are refinements of Theorem 1.1 in \cite{TT} and an extension of \cite{MS}. It would be also possible to prove the global-in-time orthonormal Strichartz estimate for $-|x|^a \Delta_{k}$ by a similar argument. The ordinary Strichartz estimate for $-|x|^a \Delta_{k}$ is proved in \cite{TT} Theorem 1.3. 
\end{example}
The next example is concerned with the Schr\"odinger operator with scaling critical electromagnetic potentials. The orthonormal Strichartz estimate for the Schr\"odinger operator with scaling critical electric potentials (including inverse square potentials) are proved in \cite{H1}. However critical magnetic potentials are not considered since it is difficult to treat the first order term by the smooth perturbation theory (\cite{KY}, \cite{KatoYajima}).
\begin{example}\label{2405072241}
Let $H_{A, a}$ be the magnetic Schr\"odinger operator defined by
\begin{align*}
H_{A, a} := \left(D_x + \frac{A(\hat{x})}{|x|} \right)^2 + \frac{a(\hat{x})}{|x|^2}
\end{align*}
on $\R^2$. Here $\hat{x} = \frac{x}{|x|}$, $A \in W^{1, \infty} (\mathbb{S}^1; \R^2)$ and $a \in W^{1, \infty} (\mathbb{S}^1; \R)$. Furthermore we assume the transversality condition: $A(\hat{x}) \cdot \hat{x} =0$ for all $x \in \R^2$. We define the total flux along $\mathbb{S}^1$ by $\Phi_{A} := \frac{1}{2\pi} \int_{0}^{2\pi} A(\cos \theta, \sin \theta) \cdot (-\sin \theta, \cos \theta) d\theta$. If $\Phi_{A} \notin \Z$, we assume $\|a_{-}\|_{L^{\infty} (\mathbb{S}^1)} < \min_{k \in \Z} |k - \Phi_{A}|^2$. If $\Phi_{A} \in \Z$, we may assume $A \equiv 0$ by the unitary equivalence. In the latter case we also assume $\min a(\hat{x})>0$. A typical example is the Schr\"odinger operator with the Aharonov-Bohm potential (see $H_{AB}$ in Corollary \ref{2405080015}). Under these assumptions $H_{A, a}$ is a self-adjoint operator on $L^2 (\R^2)$ defined by the Friedrichs extension. In \cite{FFFP}, the dispersive estimate $\| e^{-itH_{A, a}}\|_{1 \to \infty} \lesssim |t|^{-1}$ is proved. Therefore by using Theorem \ref{2405021538}, we obtain
\begin{align*}
\left\| \sum_{j=0}^{\infty} \nu_j |e^{-itH_{A, a}} f_j|^2 \right\|_{L^{\frac{q}{2}} _t  L^{\frac{r}{2}} _x} \lesssim \|\nu\|_{l^{\beta}}
\end{align*}
for all $(q, r, \beta)$ satisfying the assumption in Theorem \ref{2405021538} with $\sigma =1$.
\end{example}
\begin{remark}\label{2405080047}
We here give more comments and examples that are not considered before.
\begin{itemize}
\item
There seems to be no result on the orthonormal Strichartz estimate for variable coefficient operators or the Schr\"odinger operator on manifolds except for the Laplacian on $\mathbb{T}^d$ (\cite{N}). Concerning the Laplacian on compact Riemannian manifolds, the ordinary global-in-time Strichartz estimate never holds but the local-in-time Strichartz estimate holds with some derivative loss (\cite{BGT}). This is also true for non-elliptic operators (\cite{MT}). On non-compact manifolds, the ordinary (global-in-time or local-in-time) Strichartz estimate is proved by many authors under non-trapping conditions (\cite{BT1}, \cite{BT2}, \cite{MMT}, \cite{M3}, \cite{M4}, \cite{MY1} and \cite{Ta2}). However derivative losses may appear without non-trapping conditions. We give examples of non-compact Riemannian manifolds with (infinitely many) trapped classical trajectories on which, however, we have the orthonormal Strichartz estimate without derivative losses. Let $\Gamma \subset SO(3, 1)$ be a convex co-compact subgroup and $X := \Gamma \backslash \mathbb{H}^{3}$. Then $X$ is a hyperbolic manifold with infinitely many trapped geodesics (see \cite{BGH}). However, if its limit set has Hausdorff dimension $\delta < \frac{3}{2}$, the global dispersive estimate: $\|e^{it \Delta_{X}}\|_{1 \to \infty} \lesssim |t|^{-\frac{3}{2}}$ holds for all $t \in \R \setminus 0$ (\cite{BGH} Theorem 0.1). Therefore by using Theorem \ref{2405021538} we obtain the global-in-time orthonormal Strichartz estimate without derivative losses for $(q, r, \beta)$ with $\sigma = \frac{3}{2}$.
\item
Though we do not give details, Theorem \ref{2405021538} also applies to the Schr\"odinger operator with the Delta potential: $H_{\delta} = -\partial^2 _x + q\delta$, $q \in \R \setminus 0$. This is because the dispersive estimate $\|e^{-itH_{\delta}} P_{\ac} (H_{\delta})\|_{1 \to \infty} \lesssim |t|^{-\frac{1}{2}}$ for $q <0$ and $\|e^{-itH_{\delta}} \|_{1 \to \infty} \lesssim |t|^{-\frac{1}{2}}$ for $q >0$ is known (\cite{MMS} and \cite{Se}). 
\item
We consider the Stark Hamiltonian with perturbations: $H=H_0 +V$, $H_0 = -\Delta -x_1$ on $\R^d$, $d \ge 3$. We do not give details here (see Section \ref{2405081322} for a similar argument). By the Avron-Herbst formula (see \cite{BD}) we have $\|e^{-itH_0}\|_{1 \to \infty} \lesssim |t|^{- \frac{d}{2}}$, $t \in \R \setminus 0$, where the decay rate is optimal. By Theorem \ref{2405021538} we obtain the orthonormal Strichartz estimate for $H_0$. Furthermore by Theorem 10.1 in \cite{KT} we have the ordinary (homogeneous and inhomogeneous) Strichartz estimate with $L^r _x$ replaced by $L^{r, 2} _x$. By using the endpoint Strichartz estimate and the argument by Duyckaerts (see \cite{BM} Proposition 5.1 for $-\Delta$), we have the uniform Sobolev estimate $\|(H_0 -z)^{-1} f\|_{2^*, 2} \lesssim \|f\|_{2_{*}, 2}$ uniformly in $z \in \C \setminus \R$, where $2^* = \frac{2d}{d-2}$ and $2_{*} = (2^*)'$. This implies that $|x|^{-1}$ is $H_0$-smooth (see \cite{KY}, \cite{KatoYajima}). On the other hand Ben-Artzi and Devinatz (\cite{BD}) proved the $H_0$-smoothness of $\langle x \rangle^{- \frac{3}{4}}$. Therefore for $V=V_1 +V_2$, where $|x|^2 V_1 \in L^{\infty} (\R^d)$ and $|V_2 (x)| \lesssim \langle x \rangle^{-\frac{3}{2}}$, we have the Kato-Yajima estimate $\| |V|^{\frac{1}{2}} (H_0 -z)^{-1} |V|^{\frac{1}{2}} \|_{2 \to 2} \lesssim 1$ uniformly in $z \in \C \setminus \R$. If we set $H_{\epsilon} = H_0 + \epsilon V$ for sufficiently small $\epsilon >0$, by using $|\epsilon V|^{\frac{1}{2}} \sgn V (H_{\epsilon} -z)^{-1} |\epsilon V|^{\frac{1}{2}} = |\epsilon V|^{\frac{1}{2}} \sgn V (H_0 -z)^{-1} |\epsilon V|^{\frac{1}{2}} + |\epsilon V|^{\frac{1}{2}} \sgn V (H_0 -z)^{-1} |\epsilon V|^{\frac{1}{2}} (I + |\epsilon V|^{\frac{1}{2}} \sgn V (H_0 -z)^{-1} |\epsilon V|^{\frac{1}{2}})^{-1} |\epsilon V|^{\frac{1}{2}} \sgn V (H_0 -z)^{-1} |\epsilon V|^{\frac{1}{2}}$, we have the $H_{\epsilon}$-smoothness of $|\epsilon V|^{\frac{1}{2}}$. Hence by using Theorem 2.3 in \cite{H1} we obtain the orthonormal Strichartz estimate for $H_{\epsilon}$ with $\frac{d}{2}$-admissible pair $(q, r)$. For $H_{\epsilon}$ we can also prove the uniform Sobolev estimate as in Section \ref{2405081322}.  
\end{itemize}
\end{remark}
\subsection{\textbf{Refined Strichartz estimate}}\label{2405082127}
In this subsection we prove the refined Strichartz estimates. We refer to \cite{FS} for $-\Delta$ and to \cite{H1} and \cite{H2} for perturbed Hamiltonians. First we consider $H$ as in Example \ref{2405051203}. 
\begin{corollary}\label{2405081633}
For $H$ as in Example \ref{2405051203}, we further assume $V \ge 0$ and $0 \notin \sigma_{p} (H)$. Then we have
\begin{align*}
\|e^{-itH} u\|_{L^q (I; L^r _x)} \lesssim \|u\|_{\dot{B}^0 _{2, 2\beta} (\sqrt{H})}
\end{align*}
for all bounded $I \subset \R$, $(q, r)$ and $\beta$ as in Example \ref{2405051203}. Here $\|u\|_{\dot{B}^0 _{2, 2\beta} (\sqrt{H})} := \| \|\phi_{j} (\sqrt{H}) u\|_2 \|_{l^{2\beta}}$ denotes the homogeneous Besov space associated with $H$, where $\{\phi_j \}$ denotes a homogeneous Littlewood-Paley decomposition.
\end{corollary}
\begin{proof}
Since we assume $V \ge 0$ and $0 \notin \sigma_{p} (H)$, by the diamagnetic inequality, we have the following Gaussian upper bound for the Schr\"odinger semigroup:
\begin{align*}
|e^{-tH} (x, y)| \lesssim t^{-\frac{d}{2}} e^{-\frac{|x-y|^2}{At}}
\end{align*}
for some $A>0$ and all $t >0$. This yields $\|u\|_{p} \approx \|u\|_{\dot{F}^0 _{p, 2} (\sqrt{H})}$, where $\|u\|_{\dot{F}^0 _{p, 2} (\sqrt{H})} := \| \|\{ \phi_{j} (\sqrt{H}) u\}\|_{l^2} \|_p$ is the homogeneous Triebel-Lizorkin space associated with $H$. We refer to Proposition 4.12 in \cite{H1} for a similar proof. Then we have
\begin{align*}
\|e^{-itH} u\|_{L^{q} (I; L^{r}_x )} ^2
&\lesssim \left\|\sum_{j \in \Z} |e^{-itH} \phi_j (\sqrt{H})u|^2\right\|_{L^{\frac{q}{2}} (I; L^{\frac{r}{2}}_x )}\\
&\lesssim \left\|\sum_{j \in \Z} |e^{-itH} \phi_{2j} (\sqrt{H})u|^2\right\|_{L^{\frac{q}{2}} (I; L^{\frac{r}{2}}_x )} + \left\|\sum_{j \in \Z} |e^{-itH} \phi_{2j+1} (\sqrt{H})u|^2\right\|_{L^{\frac{q}{2}} (I; L^{\frac{r}{2}}_x )} \\
& = \left\|\sum_{j \in \Z} \|\phi_{2j} (\sqrt{H})u\|^2 _2 \left|e^{-itH} \frac{\phi_{2j} (\sqrt{H})u}{\|\phi_{2j} (\sqrt{H})u\|_2} \right|^2\right\|_{L^{\frac{q}{2}} (I; L^{\frac{r}{2}}_x )} \\
& \quad \quad + \left\|\sum_{j \in \Z} \|\phi_{2j+1} (\sqrt{H})u\|^2 _2 \left|e^{-itH} \frac{\phi_{2j+1} (\sqrt{H})u}{\|\phi_{2j+1} (\sqrt{H})u\|_2} \right|^2\right\|_{L^{\frac{q}{2}} (I; L^{\frac{r}{2}}_x )}
\end{align*}
where in the first line we have used $\|u\|_{p} \approx \|u\|_{\dot{F}^0 _{p, 2} (\sqrt{H})}$. By a support property of $\{\phi_j \}$, $\left\{ \frac{\phi_{2j} (\sqrt{H})u}{\|\phi_{2j} (\sqrt{H})u\|_2} \right\}$ and $\left\{ \frac{\phi_{2j+1} (\sqrt{H})u}{\|\phi_{2j+1} (\sqrt{H})u\|_2} \right\}$ are orthonormal systems in $L^2 (\R^d)$. Therefore by using Corollary \ref{2405061029} we obtain
\begin{align*}
\|e^{-itH} u\|_{L^{q} (I; L^{r}_x )} ^2 &\lesssim \|\{\|\phi_{2j} (\sqrt{H})u\|^2 _2 \} \|_{l^{\beta}} + \|\{\|\phi_{2j+1} (\sqrt{H})u\|^2 _2 \} \|_{l^{\beta}} \\
& \lesssim \|u\|^2 _{\dot{B}^0 _{2, 2\beta} (\sqrt{H})}
\end{align*}
and this gives the desired estimate.
\end{proof}
We notice that, since $2\beta >2$, this includes a refinement of \cite{Y2}, where the ordinary Strichartz estimate $\|e^{-itH} u\|_{L^{q} (I; L^{r}_x )} \lesssim \|u\|_2 = \|u\|_{\dot{B}^0 _{2, 2} (\sqrt{H})}$ is proved. Next we consider the Schr\"odinger operator with scaling critical electromagnetic potentials discussed in Example \ref{2405072241}. The following result is a refinement of \cite{FFFP}.
\begin{corollary}\label{2405081750}
Let $H_{A, a}$ be as in Example \ref{2405072241}. We further assume either $\Phi_{A} \notin \frac{1}{2} \Z$ or $a(\pi - \theta) = a (\pi + \theta)$ for $\theta \in [0, 2\pi]$. Then we have
\begin{align*}
\|e^{-itH_{A, a}} u\|_{L^q _t L^r _x} \lesssim \|u\|_{\dot{B}^0 _{2, 2\beta} (\sqrt{H_{A, 0}})}
\end{align*}
for all $(q, r, \beta)$ as in Example \ref{2405072241}. Notice that $\dot{B}^0 _{2, 2\beta} (\sqrt{H_{A, 0}})$ is similarly defined as in Corollary \ref{2405081633}.
\end{corollary}
\begin{proof}
By our assumption and Proposition 1.2 in \cite{GYZZ}, the Gaussian upper bound:
\begin{align*}
|e^{-tH_{A, a}} (x, y)| \lesssim t^{-1} e^{-\frac{|x-y|^2}{At}}
\end{align*}
holds for some $A>0$ and all $t >0$. Hence by repeating the argument in the proof of Corollary \ref{2405081633}, we obtain
\begin{align*}
\|e^{-itH_{A, a}} u\|_{L^q _t L^r _x} \lesssim \|u\|_{\dot{B}^0 _{2, 2\beta} (\sqrt{H_{A, a}})}
\end{align*}
for all $(q, r, \beta)$ as in Example \ref{2405072241}. Now we employ the following Kato-Yajima estimate:
\begin{align}
\sup_{z \in \C \setminus \R} \| |x|^{-1} (H_{A, a} -z)^{-1} |x|^{-1} \|_{2 \to 2} \lesssim 1 \label{2405081902}
\end{align}
which is proved in \cite{GWZZ}. Then
\begin{align*}
|\langle e^{itH_{A, a}} e^{-itH_{A, 0}} u, v \rangle - \langle e^{isH_{A, a}} e^{-isH_{A, 0}} u, v \rangle| &= \left| i \int_{s}^{t} \langle \frac{a(\hat{x})}{|x|} e^{-irH_{A, 0}} u, \frac{1}{|x|} e^{-irH_{A, a}} v\rangle dr \right| \\
& \lesssim \left\|\frac{1}{|x|} e^{-irH_{A, 0}} u \right\|_{L^2 ([s, t]; L^2 _x)} \left\|\frac{1}{|x|} e^{-irH_{A, a}} v \right\|_{L^2 _r L^2 _x} \\
& \lesssim \|v\|_2 \left\|\frac{1}{|x|} e^{-irH_{A, 0}} u \right\|_{L^2 ([s, t]; L^2 _x)}
\end{align*}
holds. This yields
\begin{align*}
\| e^{itH_{A, a}} e^{-itH_{A, 0}} u- e^{isH_{A, a}} e^{-isH_{A, 0}} u \|_2 \lesssim \left\|\frac{1}{|x|} e^{-irH_{A, 0}} u \right\|_{L^2 ([s, t]; L^2 _x)} \to 0
\end{align*}
as $s, t \to \infty$, where we have used (\ref{2405081902}) for $a =0$. Hence we have proved the existence of wave operators $W_{\pm} := \slim_{t \to \pm \infty} e^{itH_{A, a}} e^{-itH_{A, 0}} $. The completeness of $W_{\pm}$ can be similarly proved. Then by the intertwining property 
\begin{align*}
\|\phi_j (\sqrt{H_{A, a}}) u\|_2 = \|W_{\pm} \phi_j (\sqrt{H_{A, 0}}) W^* _{\pm} u\|_2 \le \|  \phi_j (\sqrt{H_{A, 0}}) W^* _{\pm} u\|_2
\end{align*}
holds. By taking the $l^{2\beta}$ norm with respect to $j$, we obtain $\|u\|_{\dot{B}^0 _{2, 2\beta} (\sqrt{H_{A, a}})} \lesssim \|W^* _{\pm}u\|_{\dot{B}^0 _{2, 2\beta} (\sqrt{H_{A, 0}})}$. Though the ordinary Hardy inequality fails if $d=2$, we have the following Hardy inequality with magnetic potentials:
\begin{align*}
\min_{k \in \Z} |k - \Phi_{A}|^2 \int_{\R^2} \frac{|u|^2}{|x|^2} dx \le \int_{\R^2} \left| \left(D_x + \frac{A(\hat{x})}{|x|} \right) u\right|^2 dx
\end{align*}
for all $u \in C^{\infty} _0 (\R^2 \setminus \{0\})$ (see \cite{GYZZ}). By our assumption and this inequality we have $\langle H_{A, a} u, u \rangle \approx \langle H_{A, 0} u, u \rangle $ and hence $H^{\pm \frac{1}{2}} _{A, a} H^{\mp \frac{1}{2}} _{A, 0}$ is bounded on $L^2 (\R^2)$. Again by the intertwining property
\begin{align*}
\|H^{\pm \frac{1}{2}} _{A, 0} W^* _{\pm} u\|_2 = \|W^* _{\pm} H^{\pm \frac{1}{2}} _{A, a} u\|_2 \le \|H^{\pm \frac{1}{2}} _{A, a} u\|_2 \lesssim \|H^{\pm \frac{1}{2}} _{A, 0} u\|_2 
\end{align*}
holds. This implies that $W^* _{\pm}$ are bounded on $\dot{H}^{\pm1} (\sqrt{H_{A, 0}})$, where $\dot{H}^{\pm1} (\sqrt{H_{A, 0}})$ are defined similarly as the Besov spaces associated with $H_{A, 0}$. Now we use the real interpolation: $(\dot{H}^{s_0} (\sqrt{H_{A, 0}}), \dot{H}^{s_1} (\sqrt{H_{A, 0}}))_{\theta, q} = \dot{B}^{s} _{2, q} (\sqrt{H_{A, 0}})$ for $s_0 < s < s_1$, $s= (1-\theta)s_0 + \theta s_1$, $\theta \in (0, 1)$ and $q \in (2, \infty)$, which is proved in Appendix \ref{2405082014}. Then $W^* _{\pm}$ are bounded on $\dot{B}^0 _{2, 2\beta} (\sqrt{H_{A, 0}})$. As a result we obtain the desired estimate.
\end{proof}
Note that this argument does not hold if we replace $\dot{B}^0 _{2, 2\beta} (\sqrt{H_{A, 0}})$ with $\dot{B}^0 _{2, 2\beta}$ since we do not have the equivalence: $\langle -\Delta u, u\rangle \approx \langle H_{A, a} u, u \rangle$.
\subsection{\textbf{Mass critical NLS with scaling critical electromagnetic potentials}}\label{2405082014}
In this subsection we discuss about the $L^2$-critical NLS with scaling critical electromagnetic potentials:
\begin{align}
\left\{
\begin{array}{l}
i\partial_{t} u=H_{A, a} u + \lambda |u|^{2} u \\
u(0)=u_0
\end{array}
\right.
\tag{NLS}\label{2405082113}
\end{align}
where $\lambda \in \R \setminus \{0\}$. It is known (at least for $A=a=0$) that if the nonlinearity is $\lambda |u|^{p} u$ with $p <2$, the corresponding NLS is globally well-posed in $L^2 (\R^2)$. Concerning (\ref{2405082113}), if $\|u_0\|_2$ is sufficiently small, there exists a unique global solution to (\ref{2405082113}). By using the refined Strichartz estimate proved in subsection \ref{2405082127}, we refine this result and prove that a unique solution scatters.
\begin{thm}\label{2405082129}
Let $H_{A, a}$ be as in Corollary \ref{2405081750}. Then for any $M>0$, there exists $\eta >0$ such that if $\|u_0\|_2 <M$ and $\|u_0\|_{\dot{B}^{0} _{2, \infty} (\sqrt{H_{A, 0}})} < \eta$, then we have a unique global solution to (\ref{2405082113}) satisfying $u \in C(\R; L^2 _x) \cap L^4 (\R_t \times \R^2 _x)$. Furthermore $u$ scatters as $t \to \pm \infty$, i.e. there exists $u_{\pm} \in L^2 (\R^2)$ satisfying
\begin{align*}
\lim_{t \to \pm \infty} \|u(t) - e^{-itH_{A, 0}} u_{\pm}\|_2 = 0.
\end{align*}
\end{thm}
\begin{proof}
The proof is based on a standard contraction mapping argument. We set
\begin{align*}
X := \left\{ u \in L^{\infty} (\R; L^2 _x) \cap L^4 (\R_t \times \R^2 _x) \mid \|u\|_{L^4 (\R_t \times \R^2 _x)} \le K \eta^{\frac{1}{4}} \right\}.
\end{align*}
where $K, \eta >0$ are specified later.
Then $X$ is a complete metric space with respect to $d (u_1, u_2) = \|u_1 - u_2\|_{L^{\infty} (\R; L^2 _x) \cap L^4 (\R_t \times \R^2 _x)}$. We prove that
\begin{align*}
\Phi : u \mapsto  e^{-itH_{A, a}} u_0 + i\int_{0}^{t} e^{-i(t-s)H_{A, a}} (|u|^{2} u) (s) ds
\end{align*}
is a contraction on $X$ if $\eta >0$ is small enough. We take $u \in X$. Then
\begin{align*}
\|\Phi (u)\|_{L^4 (\R_t \times \R^2 _x)} &\le \|e^{-itH_{A, a}} u_0\|_{L^4 (\R_t \times \R^2 _x)} + \left\| \int_{0}^{t} e^{-i(t-s)H_{A, a}} (|u|^{2} u) (s) ds \right\|_{L^4 (\R_t \times \R^2 _x)} \\
& \le \|e^{-itH_{A, a}} u_0\|_{L^4 (\R_t \times \R^2 _x)} + C_0 \||u|^2 u\|_{L^{\frac{4}{3}} (\R_t \times \R^2 _x)} \\
& =  \|e^{-itH_{A, a}} u_0\|_{L^4 (\R_t \times \R^2 _x)} + C_0 \|u\|^3 _{L^4 (\R_t \times \R^2 _x)} \\
& \le C_1 \|u_0\|_{\dot{B}^0 _{2, \frac{8}{3}} (\sqrt{H_{A, 0}})} + C_0 (K \eta^{\frac{1}{4}})^3 \\
& \le C_1 \|u_0\|^{\frac{1}{4}} _{\dot{B}^{0} _{2, \infty} (\sqrt{H_{A, 0}})} \|u_0\|^{\frac{3}{4}} _2 + C_0 (K \eta^{\frac{1}{4}})^3 \\
& \le (C_1 M^{\frac{3}{4}} + C_0 K^3 \eta^{\frac{1}{2}}) \eta^{\frac{1}{4}}
\end{align*}
holds, where in the second line we have used the inhomogeneous Strichartz estimate, in the fourth line we have used Corollary \ref{2405081750} and in the last line $\|u_0\|_{\dot{B}^{0} _{2, \infty} (\sqrt{H_{A, 0}})} < \eta$ and $\|u_0\|_2 <M$ have been used. By a similar calculation we obtain
\begin{align*}
\|\Phi (u)\|_{L^{\infty} (\R; L^2 _x)} \le \|u_0\|_2 + C_2 \|u\|^3 _{L^4 (\R_t \times \R^2 _x)} < \infty
\end{align*}
Then by setting $K= 2C_1 M^{\frac{3}{4}}$ and $\eta >0$ sufficiently small, we obtain $\Phi : X \to X$. Concerning a difference we have 
\begin{align*}
\|\Phi (u) - \Phi (v)\|_{L^{\infty} (\R; L^2 _x) \cap L^4 (\R_t \times \R^2 _x)} &= \left\| \int_{0}^{t} e^{-i(t-s)H_{A, a}} (|u|^2 u - |v|^2 v) (s) ds \right\|_{L^{\infty} (\R; L^2 _x) \cap L^4 (\R_t \times \R^2 _x)} \\
& \le C_3 \||u|^2 u - |v|^2 v\|_{L^{\frac{4}{3}} (\R_t \times \R^2 _x)} \\
& \le C_4 (\|u\|^2 _{L^4 (\R_t \times \R^2 _x)} + \|v\|^2 _{L^4 (\R_t \times \R^2 _x)}) \|u-v\|_{L^4 (\R_t \times \R^2 _x)} \\
& \le 2C_4 K^2 \eta^{\frac{1}{2}} \|u-v\|_{L^{\infty} (\R; L^2 _x) \cap L^4 (\R_t \times \R^2 _x)}.
\end{align*}
Therefore by taking $\eta < \min \{ (\frac{C_1 M^{\frac{3}{4}}}{C_0 K^3})^2, \frac{1}{(4C_4 K^2)^2}\}$, $\Phi :X \to X$ is a contraction. This implies the unique existence of a global solution to (\ref{2405082113}). Next we prove that the soluion $u$ scatters. By the inhomogeneous Strichartz estimate 
\begin{align*}
\|e^{itH_{A, a}} u(t) - e^{isH_{A, a}} u(s)\|_2 &\le \|u(r) - e^{-i(r-s)H_{A, a}} u(s)\|_{L^{\infty} ([s, t]; L^2 _x)} \\
& = \left\| \int_{s}^{r} e^{-i(r-k)H_{A, a}} (|u|^2 u) (k) dk\right\|_{L^{\infty} ([s, t]; L^2 _x)} \\
& \lesssim \||u|^2 u\|_{L^{\frac{4}{3}} ([s, t]; L^{\frac{4}{3}} _x)} = \|u\|^3 _{L^4 ([s, t]; L^4 _x)}
\end{align*}
holds. Since $u \in L^4 (\R_t \times \R^2 _x)$, this implies the existence of $\widetilde{u_{\pm}} = \lim_{t \to \pm \infty} e^{itH_{A, a}} u(t)$. Furthermore, by the completeness of wave operators $W_{\pm} = \slim_{t \to \pm \infty} e^{itH_{A, a}} e^{-itH_{A, 0}}$ (see the proof of Corollary \ref{2405081750}), there exist $u_{\pm}$ such that $u_{\pm} = \lim_{t \to \pm \infty} e^{itH_{A, 0}} e^{-itH_{A, a}} \widetilde{u_{\pm}}$. Then
\begin{align*}
\|u(t) - e^{-itH_{A, 0}} u_{\pm}\|_2 &\le \|u(t) - e^{-itH_{A, a}} \widetilde{u_{\pm}}\|_2 + \|e^{-itH_{A, a}} \widetilde{u_{\pm}} - e^{-itH_{A, 0}} u_{\pm}\|_2 \\
& \longrightarrow 0
\end{align*}
holds as $t \to \pm \infty$.
\end{proof}
\section{Uniform resolvent estimate for repulsive Hamiltonian}\label{2405081322}
In this section we consider $H=H_0 +V$, $H_0 = -\Delta -x^2$ and $V \in C^{\infty} (\R^d; \R)$ on $L^2 (\R^d)$, where $d \ge 1$. First we introduce Fourier integral operators (FIOs) that make our analysis simpler. See \cite{Dy}, \cite{Ma} and \cite{Hor} for Lagrangian distributions and FIOs. For a canonical transform $\kappa$ on $T^* \R^d$, we define $I(\kappa)$ as the set of all the FIOs associated with $\kappa$. Since the symbol of $H_0$ is $\xi^2 -x^2$, we consider the canonical transform 
\begin{align}
\kappa: (x, \xi) \mapsto (y, \eta) = \left(\frac{x + \xi}{\sqrt{2}}, \frac{\xi -x}{\sqrt{2}} \right).\label{2405110145}
\end{align}
By this transform, $\xi^2 -x^2$ changes into $2y\eta$. To look at $\kappa$ more closely, we define
\begin{align*}
&\mathcal{F}_h :u \longmapsto \frac{1}{(2\pi h)^{\frac{d}{2}}} \int_{\R^d} e^{-ix\xi /h} u(x) dx \\
&J_B : u \longmapsto |\det B|^{\frac{1}{2}} u \circ B \\
&K_C : u \longmapsto e^{-i\langle Cx, x \rangle /2h} u 
\end{align*}
where $B$ is an invertible matrix and $C$ is a non-degenerate symmetric matrix. Then it is well-known that $\mathcal{F}_h \in I(\ell)$, $J_B \in I(j_B)$ and $K_C \in I(k_C)$ holds, where $\ell : (x, \xi) \mapsto (-\xi, x)$, $j_B : (x, \xi) \mapsto (Bx, {}^t B^{-1} \xi)$ and $k_C : (x, \xi) \mapsto (x, \xi + Cx)$ are corresponding canonical transforms. Since we can decompose $\kappa : (x, \xi) \mapsto (x, x+\xi) \mapsto (-(x+\xi), x) \mapsto \left(\frac{x+\xi}{\sqrt{2}}, -\sqrt{2} x \right) \mapsto \left(\frac{x+\xi}{\sqrt{2}}, \frac{\xi -x}{\sqrt{2}} \right)$, we define
\begin{align*}
U_h :u \longmapsto K_{id} J_{-\sqrt{2} id}  \mathcal{F}_h K_{id} u = \frac{1}{(\sqrt{2} \pi h)^{\frac{d}{2}}} e^{-ix^2 /2h} \int_{\R^d} e^{i\sqrt{2} xy/h} e^{-iy^2 /2h} u(y) dy
\end{align*}
where $id$ denotes the $d \times d$ identity matrix. Then $U_h$ is a FIO associated with $\kappa$ since
\begin{align*}
\Lambda_{\Phi} &:= \{(x, \partial _x \Phi, y, \partial _y \Phi ) \mid (x, y) \in \R^{2d}\} \\
&= \{(x, -x +\sqrt{2} y, y, -y+\sqrt{2} x) \mid (x, y) \in \R^{2d}\} \\
&= \left\{\left(x, \xi, \frac{x+\xi}{\sqrt{2}}, \frac{x-\xi}{\sqrt{2}}\right) \mid (x, \xi) \in T^* \R^d \right\} \\
& = \{(x, \xi, y, -\eta) \mid (y, \eta) = \kappa (x, \xi)\}
\end{align*}
holds, where $\Phi (x, y) = -\frac{x^2}{2} -\frac{y^2}{2} + \sqrt{2} xy$ is a phase function of the Schwartz kernel of $U_h$. By the unitarity of $\mathcal{F}_h, J_B$ and $K_C$, $U_h$ is also a unitary operator on $L^2 (\R^d)$. Furthermore $U_h : \mathcal{S} \to \mathcal{S}$ and $\mathcal{S}' \to \mathcal{S}'$ is a homeomorphism. We set $U := U_1$. Then for all $u \in \mathcal{S} (\R^d)$ and $V \in C^{\infty} _0 (\R^d)$, we have
\begin{align*}
U^* VUu (z) &= \frac{1}{(\sqrt{2} \pi)^d} \int e^{iz^2 /2} e^{-i\sqrt{2} xz} e^{ix^2 /2} V(x) \left( e^{-ix^2 /2} \int e^{i\sqrt{2} xy} e^{-iy^2 /2} u(y) dy \right) dx \\
& = \frac{1}{(\sqrt{2} \pi)^d} \int \int V(x) e^{-i\sqrt{2} x(z-y)} e^{i(z^2 -y^2)/2} u(y)dydx \\
& = \frac{1}{(2\pi)^d} \int \int e^{i(z-y) \zeta} V\left( \frac{z+y}{2\sqrt{2}} -\frac{\zeta}{\sqrt{2}} \right) u(y)dyd \zeta = V^w \left( \frac{x -D_x}{\sqrt{2}}\right) u(z)
\end{align*}
where, for a symbol $a \in S(m, g)$ with slowly varying metrics $g$ and $g$-continuous weight functions $m$, $a^w (x, D_x)$ denotes the Weyl quantization:
\begin{align*}
a^w (x, D_x)u(x) = \frac{1}{(2\pi)^d} \int \int e^{i(x-y) \xi} a\left(\frac{x+y}{2}, \xi \right) u(y)dyd\xi
\end{align*} 
See \cite{Hor} and \cite{Ma} for pseudodifferential operators. We take $V \in S^0 (\R^d) = \{V \in C^{\infty} (\R^d; \C) \mid \forall \alpha \in \N^d _0, |\partial^{\alpha} _x V(x)| \lesssim 1\}$ and a smooth cut off function $\chi$ near the origin. By the above argument, we have
\begin{align*}
U^* V_n U = V_n ^w \left( \frac{x -D_x}{\sqrt{2}}\right) 
\end{align*}
where $V_n = V \chi (\frac{\cdot}{n})$. Concerning the left hand side we have $\langle U^* V_n U u, v \rangle \to \langle U^* V Uu, v \rangle$ as $n \to \infty$ for all $u, v \in \mathcal{S} (\R^d)$. On the other hand we have
\begin{align*}
V_n ^w \left( \frac{x -D_x}{\sqrt{2}}\right) u (x) &= \frac{1}{(2\pi)^d} \int \int e^{i(x-y) \xi} V_n \left( \frac{x+y}{2\sqrt{2}} -\frac{\xi}{\sqrt{2}} \right) u(y)dyd \xi \\
& = \frac{1}{(2\pi)^d} \int \int e^{ix\xi} \frac{(1+D^2 _{\xi})^N}{\langle x \rangle^{2N}} \left\{ e^{-iy\xi} \frac{(1+D^2 _y)^N}{\langle \xi \rangle^{2N}} V_n \left( \frac{x+y}{2\sqrt{2}} -\frac{\xi}{\sqrt{2}} \right) u(y) \right\} dyd \xi
\end{align*}
for all $u \in \mathcal{S} (\R^d)$. Since the integrand is uniformly estimated from above by a integrable function, $V_n ^w \left( \frac{x -D_x}{\sqrt{2}}\right) u (x) \to V ^w \left( \frac{x -D_x}{\sqrt{2}}\right) u (x)$ holds for all $x \in \R^d$. Furthermore, the uniform bound $\left|V_n ^w \left( \frac{x -D_x}{\sqrt{2}}\right) u (x) \right| \lesssim \langle x \rangle^{-2N}$ implies that $\left\langle V_n ^w \left( \frac{x -D_x}{\sqrt{2}}\right) u, v \right\rangle \to \left\langle V ^w \left( \frac{x -D_x}{\sqrt{2}}\right) u, v \right\rangle$ holds for all $u, v \in \mathcal{S} (\R^d)$. Therefore, for $V \in S^0 (\R^d)$, we have
\begin{align*}
U^* V U = V ^w \left( \frac{x -D_x}{\sqrt{2}}\right). 
\end{align*}
Notice that this is an analogue of the well-known formula $\mathcal{F} a^w (x, D_x) \mathcal{F}^* = a^w (-D_{\xi}, \xi)$.
\subsection{\textbf{Kato-Yajima estimate for the repulsive Hamiltonian}}\label{2405091758}
In this subsection we prove the Kato-Yajima type uniform resolvent estimate for $H$ under weaker assumptions on the decay order of $V$ than \cite{KaYo}. Their result is based on the LAP bound by \cite{BCHM} in the middle energy regime. On the other hand, in the high or low energy regime, they use the dispersive estimate for $e^{-itH_0}$ and H\"older's inequality to obtain a uniform bound. Due to the use of H\"older's inequality, their weight functions and perturbations are assumed to be decaying with polinomial order. In order to gain a uniform resolvent estimate with logarithmic decaying weight functions and perturbations, we employ the weakly conjugate operator method, which is a version of the Mourre theory. It has been used for the Schr\"odinger operator with inverse square type potentials (\cite{BM}), for the fractional Schr\"odinger operator with the Hardy type potentials (\cite{MY1}) and for the massless Klein-Gordon operator (\cite{M5}). See also \cite{BoMa} for more abstract result. However there seems to be no result on the Schr\"odinger operator with growing potentials. In this subsection, as a conjugate operator, we use
\begin{align*}
A=a^w (x, D_x), \quad a(x, \xi)= \log \langle x+\xi \rangle - \log \langle x-\xi \rangle \in S(\langle \log \langle x \rangle \rangle, g_0)
\end{align*}   
where $g_0 = dx^2 +d\xi^2$. This conjugate operator is used in \cite{BCHM} to gain an energy-localized Mourre inequality and a LAP bound in the middle energy regime but no energy-global resolvent estimate is proved. We notice that by a direct calculation
\begin{align}
A=U (\log \langle \sqrt{2} x\rangle - \log \langle \sqrt{2} D_x\rangle)U^* \label{2405121442}
\end{align}
holds. This formula and Nelson's commutator theorem with the harmonic oscillator ensures the essential self-adjointness of $A$ on $\mathcal{S} (\R^d)$. The following lemma is important in order to apply the Mourre theory.
\begin{lemma}\label{2405111618}
For any $V \in \tilde{S}^0 (\R^d):= \{V \in C^{\infty} (\R^d; \R) \mid \forall \alpha \in \N^d _0, |\partial^{\alpha} _x V(x)| \lesssim \langle x \rangle^{-|\alpha|}\}$ there exists $\epsilon_0 >0$ such that 
\begin{align}
&\langle [H_0, iA]u, u \rangle \lesssim \langle [H, iA]u, u \rangle \label{2405111853} \\ 
&|\langle [[H, iA], iA]u, u\rangle | \lesssim \langle [H_0, iA]u, u\rangle \label{2405111854}
\end{align}
holds for $H=H_0 + \epsilon V$, $\epsilon \in [0, \epsilon_0)$ and all $u \in C^{\infty} _0 (\R^d)$.
\end{lemma}
Note that the restriction $\epsilon \in [0, \epsilon_0)$ is used only in the proof of (\ref{2405111853}).
\begin{proof}
(i) We prove (\ref{2405111853}). By (\ref{2405121442}) and $H_0 = U(xD_x +D_x x -1)U^*$, it is easy to see that
\begin{align*}
[H_0, iA] = U \left( \frac{4x^2}{\langle \sqrt{2} x \rangle^2} + \frac{4D^2 _x}{\langle \sqrt{2} D_x \rangle^2} \right)U^* \ge 0
\end{align*}
holds. Since we have
\begin{align*}
[V, iA] &= U[U^* VU, i(\log \langle \sqrt{2} x\rangle - \log \langle \sqrt{2} D_x\rangle)]U^* \\
&= U\left[V^w \left( \frac{x -D_x}{\sqrt{2}}\right), i(\log \langle \sqrt{2} x\rangle - \log \langle \sqrt{2} D_x\rangle)\right]U^*
\end{align*}
it suffices to show that
\begin{align*}
&\left| \left\langle \left[V^w \left( \frac{x -D_x}{\sqrt{2}}\right), i\log \langle \sqrt{2} x \rangle \right]u, u \right\rangle \right| \lesssim \left\langle \left(\frac{4x^2}{\langle \sqrt{2} x \rangle^2} + \frac{4D^2 _x}{\langle \sqrt{2} D_x \rangle^2}\right)u, u \right\rangle, \\
&\left| \left\langle \left[V^w \left( \frac{x -D_x}{\sqrt{2}}\right), i\log \langle \sqrt{2} D_x \rangle \right]u, u \right\rangle \right| \lesssim \left\langle \left(\frac{4x^2}{\langle \sqrt{2} x \rangle^2} + \frac{4D^2 _x}{\langle \sqrt{2} D_x \rangle^2}\right)u, u \right\rangle
\end{align*}
holds for all $u \in \mathcal{S}$. We only prove the former estimate since the latter can be proved similarly. First we prove that the commutator in the left hand side is bounded. Notice that this is not a consequence of general theories of pseudodifferential operators since $V \left( \frac{x -\xi}{\sqrt{2}}\right)$ does not have a symbol type decay estimate. However we can use $ \log \langle \sqrt{2} x \rangle \in S(\langle \log \langle x \rangle \rangle, g_1 )$ in our estimate, where $g_1 = \frac{dx^2}{\langle x \rangle^2} + d\xi^2$. By the composition formula for the Weyl calculus (see \cite{Hor} and \cite{Ma}), we have
\begin{align*}
V^w \left( \frac{x -D_x}{\sqrt{2}}\right) \log \langle \sqrt{2} x \rangle = b^w (x, D_x), \quad  \log \langle \sqrt{2} x \rangle V^w \left( \frac{x -D_x}{\sqrt{2}}\right) = c^w (x, D_x)
\end{align*}
where $b, c \in S(\langle \log \langle x \rangle \rangle, g_0)$ have the following formulas:
\begin{align*}
&b (x, \xi) = (2\pi)^{-2d} \int_{\R^{4d}} e^{-i(y\eta -z\zeta)} V\left( \frac{x+\frac{z}{2} -(\xi + \eta)}{\sqrt{2}}\right) \log \left\langle \sqrt{2} \left(\frac{y}{2} +x \right)\right\rangle dydzd\zeta d\eta, \\
&c(x, \xi) = (2\pi)^{-2d} \int_{\R^{4d}} e^{-i(y\eta -z\zeta)}  \log \left\langle \sqrt{2} \left(x+\frac{z}{2}\right)\right\rangle V \left(\frac{\frac{y}{2} +x-(\xi + \zeta)}{\sqrt{2}}\right)  dydzd\zeta d\eta
\end{align*}
By the Fourier inversion formula, we obtain
\begin{align*}
&b(x, \xi) = (2\pi)^{-d} \int_{\R^{2d}} e^{-iy\eta} V\left( \frac{x -(\xi + \eta)}{\sqrt{2}}\right) \log \left\langle \sqrt{2} \left(\frac{y}{2} +x \right)\right\rangle dyd\eta, \\
&c(x, \xi) = (2\pi)^{-d} \int_{\R^{2d}} e^{iz\zeta} \log \left\langle \sqrt{2} \left(x+\frac{z}{2}\right)\right\rangle V\left( \frac{x -(\xi + \zeta)}{\sqrt{2}}\right) dzd\zeta.
\end{align*}
Now we use the stationary phase theorem:
\begin{align}
&\int_{\R^d} e^{ixQx/2h} a(x)dx = \sum_{k=0}^{N-1} \frac{(2\pi)^{d/2} h^{k+d/2} e^{i\pi \sgn Q/4} }{|\det Q|^{1/2} (2i)^k k!} ((D_x Q^{-1} D_x)^k a)(0) + R_N, \notag \\
& |R_N| \lesssim \frac{h^{N+d/2}}{|\det Q|^{1/2} 2^N N!} \sum_{|\alpha| \le d+1} \|\partial^{\alpha} _x (D_x Q^{-1} D_x)^N a\|_1. \label{2405121915}
\end{align}
for either of the following non-degenerate $2d \times 2d$ symmetric matrix $Q$: 
\begin{align*}
Q = -
\begin{pmatrix}
0 & I \\
I & 0
\end{pmatrix}
\quad or \quad Q=
\begin{pmatrix}
0 & I \\
I & 0
\end{pmatrix}
\end{align*}
Since $\sgn Q =0$, we have
\begin{align*}
b(x, \xi) &= (2\pi)^{-d/2} \sum_{k=0}^{N-1} \frac{1}{(2i)^k k!} (-2D_y D_{\eta})^k \left\{ V\left( \frac{x -(\xi + \eta)}{\sqrt{2}}\right) \log \left\langle \sqrt{2} \left(\frac{y}{2} +x \right)\right\rangle \right\} \Biggm\vert_{y=\eta=0} \\
&  \quad \quad \quad \quad \quad \quad \quad \quad \quad \quad \quad \quad \quad \quad \quad + R_{N, 1} (x, \xi)
\end{align*}
\begin{align*}
c(x, \xi) &= (2\pi)^{-d/2} \sum_{k=0}^{N-1} \frac{1}{(2i)^k k!} (2D_z D_{\zeta})^k \left\{ \log \left\langle \sqrt{2} \left(x+\frac{z}{2}\right)\right\rangle V\left( \frac{x -(\xi + \zeta)}{\sqrt{2}}\right) \right\} \Biggm\vert_{z=\zeta =0} \\
& \quad \quad \quad \quad \quad \quad \quad \quad \quad \quad \quad \quad \quad \quad \quad + R_{N, 2} (x, \xi).
\end{align*}
To prove the $L^2$-boundedness of $R^w _{N, 1} (x, D_x)$ and $R^w _{N, 2} (x, D_x)$ for large $N \in \N$, it suffices to show that $|\partial^{\alpha} _x \partial^{\beta} _{\xi} R_{N, 1} (x, \xi)| \lesssim 1$ and $|\partial^{\alpha} _x \partial^{\beta} _{\xi} R_{N, 2} (x, \xi)| \lesssim 1$ hold. However, by changing the order of differentiation and integration, we only need to prove these estimate for $\alpha = \beta =0$. By (\ref{2405121915}), we have
\begin{align*}
|R_{N, 1} (x, \xi)| &\lesssim \frac{1}{2^N N!} \sum_{|\alpha| + |\beta| \le 2d+1} \left\|\partial^{\alpha} _y \partial^{\beta} _{\eta} (-2D_y D_{\eta})^N \left\{ V\left( \frac{x -(\xi + \eta)}{\sqrt{2}}\right) \log \left\langle \sqrt{2} \left(\frac{y}{2} +x \right)\right\rangle \right\} \right\|_{L^1 (\R^d _y \times \R^d _{\eta})} \\
& \lesssim 1
\end{align*}
for large $N \in \N$ since we assume $|\partial^{\alpha} _x V(x)| \lesssim \langle x \rangle^{-|\alpha|}$. Notice that $R_{N, 2} (x, \xi)$ can be estimated similarly. We also have the $L^2$-boundedness for $\Psi$DOs defined by the $k$-th terms $(k \ge 1)$ in the above summation. Since the 0th terms in the above summations cancel each other, we have
\begin{align*}
\left[V^w \left( \frac{x -D_x}{\sqrt{2}}\right), i\log \langle \sqrt{2} x \rangle \right] =i(b^w (x, D_x) -c^w (x, D_x)) \in OPS(1, g_0)
\end{align*}
and hence its $L^2$-boundedness. We take $\chi \in C^{\infty} _0 (\R^d; [0, 1])$ such that $\chi =1$ on $\{|x| \le 1\}$ and $\chi =0$ on $\{|x| \ge 2\}$. We set $\bar{\chi} = 1-\chi$. Then
\begin{align*}
\left |\left\langle \bar{\chi} \left[V^w \left( \frac{x -D_x}{\sqrt{2}}\right), i\log \langle \sqrt{2} x \rangle \right] \bar{\chi} u, u \right\rangle \right | \lesssim \|\bar{\chi} u\|^2 _2 &\lesssim \left\|\bar{\chi} \frac{2x}{\langle \sqrt{2} x \rangle} u \right\|^2 _2 \\
& \lesssim \left\langle \frac{4x^2}{\langle \sqrt{2} x \rangle^2} u, u \right\rangle
\end{align*}
holds. Furthermore, by a support property of $\chi$, we have
\begin{align*}
\left |\left\langle \chi \left[V^w \left( \frac{x -D_x}{\sqrt{2}}\right), i\log \langle \sqrt{2} x \rangle \right] \bar{\chi} u, u \right\rangle \right | \lesssim \|\bar{\chi} u\|_2 \|\chi u\|_2 \lesssim \|\bar{\chi} u\|^2 _2 + \|\chi u\|^2 _2
\end{align*}
and
\begin{align*}
\|\chi u\|_2 \lesssim \|\chi (x)\chi (D_x)u\|_2 + \|\bar{\chi} (D_x) u\|_2 &\lesssim \|\chi (x) \langle \sqrt{2} D_x \rangle^{-1} \langle \sqrt{2} D_x \rangle \chi(D)u \|_2 + \left\|\frac{2D_x}{\langle \sqrt{2} D_x \rangle} u \right\|_2 \\
& \lesssim \left\| \frac{2D_x}{\langle \sqrt{2} D_x \rangle}  \langle \sqrt{2} D_x \rangle \chi(D)u \right\|_2 + \left\|\frac{2D_x}{\langle \sqrt{2} D_x \rangle} u \right\|_2 \\
& \lesssim  \left(\left\langle \frac{4D_x ^2}{\langle \sqrt{2} D_x \rangle^2} u, u \right\rangle \right)^{1/2}
\end{align*}
where in the latter estimate we have used the Poincar\'e inequality. Therefore
\begin{align*}
\left |\left\langle \chi \left[V^w \left( \frac{x -D_x}{\sqrt{2}}\right), i\log \langle \sqrt{2} x \rangle \right] \bar{\chi} u, u \right\rangle \right | \lesssim \left\langle \left(\frac{4x^2}{\langle \sqrt{2} x \rangle^2} + \frac{4D^2 _x}{\langle \sqrt{2} D_x \rangle^2}\right)u, u \right\rangle
\end{align*}
holds. Finally we have
\begin{align*}
\left |\left\langle \chi \left[V^w \left( \frac{x -D_x}{\sqrt{2}}\right), i\log \langle \sqrt{2} x \rangle \right] \chi u, u \right\rangle \right | \lesssim \|\chi u\|^2 _2 \lesssim \left\langle \frac{4D_x ^2}{\langle \sqrt{2} D_x \rangle^2} u, u \right\rangle.
\end{align*}
These estimates imply
\begin{align*}
\left| \left\langle \left[V^w \left( \frac{x -D_x}{\sqrt{2}}\right), i\log \langle \sqrt{2} x \rangle \right]u, u \right\rangle \right| \lesssim \left\langle \left(\frac{4x^2}{\langle \sqrt{2} x \rangle^2} + \frac{4D^2 _x}{\langle \sqrt{2} D_x \rangle^2}\right)u, u \right\rangle
\end{align*}
and we are done. (ii) Next we prove (\ref{2405111854}). By a direct calculation,
\begin{align*}
[[H_0, iA], iA] &= U\left[\frac{4x^2}{\langle \sqrt{2} x \rangle^2} + \frac{4D^2 _x}{\langle \sqrt{2} D_x \rangle^2}, i(\log \langle \sqrt{2} x\rangle - \log \langle \sqrt{2} D_x\rangle )\right]U^* \\
& = -U\left[\frac{4x^2}{\langle \sqrt{2} x \rangle^2}, i \log \langle \sqrt{2} D_x\rangle \right]U^* + U\left[\frac{4D_x ^2}{\langle \sqrt{2} D_x \rangle^2}, i \log \langle \sqrt{2} x\rangle \right]U^*
\end{align*}
holds. For both terms, by the composition formula, we have
\begin{align*}
\left[\frac{4x^2}{\langle \sqrt{2} x \rangle^2}, i \log \langle \sqrt{2} D_x\rangle \right] \in OPS(1, g_2), \quad \left[\frac{4D_x ^2}{\langle \sqrt{2} D_x \rangle^2}, i \log \langle \sqrt{2} x\rangle \right] \in OPS(1, g_2)
\end{align*}
where $g_2 = \frac{dx^2}{\langle x \rangle^2} + \frac{d\xi^2}{\langle \xi \rangle^2}$. This yields the $L^2$-boundedness of $[[H_0, iA], iA]$. Since we have
\begin{align*}
\|u\|^2 _2 \lesssim \left\langle \left(\frac{4x^2}{\langle \sqrt{2} x \rangle^2} + \frac{4D^2 _x}{\langle \sqrt{2} D_x \rangle^2}\right)u, u \right\rangle
\end{align*}
in the process of (i), we obtain $| \langle [[H_0, iA], iA]u, u \rangle | \lesssim \langle [H_0, iA]u, u\rangle$. Therefore it suffices to show that $[[V, iA], iA]$ is bounded on $L^2 (\R^d)$. We prove this as a consequence of
\begin{align*}
U^* [[V, iA], iA] U &= \left[\left[V^w \left( \frac{x -D_x}{\sqrt{2}}\right), i(\log \langle \sqrt{2} x\rangle - \log \langle \sqrt{2} D_x\rangle)\right], i(\log \langle \sqrt{2} x\rangle - \log \langle \sqrt{2} D_x\rangle) \right] \\
& \in OPS(1, g_0).
\end{align*}
By the Jacobi identity, this is reduced to $\left[\left[V^w \left( \frac{x -D_x}{\sqrt{2}}\right), \log \langle \sqrt{2} \mathcal{A} \rangle\right], \log \langle \sqrt{2} \mathcal{B} \rangle\right] \in OPS(1, g_0)$, where $(\mathcal{A}, \mathcal{B}) = (x, x), (x, D_x)$ and $(D_x, D_x)$. Since the case $(\mathcal{A}, \mathcal{B}) =(D_x, D_x)$ can be proved almost similarly to the case $(\mathcal{A}, \mathcal{B}) =(x, x)$, we only prove the other two cases. First we assume $(\mathcal{A}, \mathcal{B}) = (x, x)$. We need to estimate $b(x, \xi)$ and $c(x, \xi)$ more precisely. Let $\rho \in (0, 1)$ and $0 < \epsilon_0, \epsilon_1 \ll 1$. We set
\begin{align*}
&\chi_{1} (x, y, \eta) = \chi \left(\frac{y}{\epsilon_0 \langle x \rangle}\right) \chi \left( \frac{\eta}{\epsilon_1 \langle x \rangle^{\rho}}\right), \\
&\chi_{2} (x, y, \eta) = \left\{1-\chi \left(\frac{y}{\epsilon_0 \langle x \rangle}\right) \right\} \chi \left( \frac{\eta}{\epsilon_1 \langle x \rangle^{\rho}}\right), \\
&\chi_{3} (x, y, \eta) = 1- \chi \left(\frac{\eta}{\epsilon_1 \langle x \rangle^{\rho}}\right)
\end{align*} 
and
\begin{align*}
I_j = (2\pi)^{-d} \int_{\R^{2d}} e^{-iy\eta} \chi_j (x, y, \eta) V\left( \frac{x -(\xi + \eta)}{\sqrt{2}}\right) \log \left\langle \sqrt{2} \left(\frac{y}{2} +x \right)\right\rangle dyd\eta
\end{align*}
for $j=1,2,3$. For $I_1$, by changing variables, we have
\begin{align*}
I_1 &= (2\pi)^{-d} \langle x \rangle^{(1+ \rho)d} \int_{\R^{2d}} e^{-i\langle x \rangle^{(1+ \rho)} y\eta} \chi \left(\frac{y}{\epsilon_0}\right) \chi \left( \frac{\eta}{\epsilon_1}\right) \\
& \quad \quad \quad \quad \quad \quad \quad \quad \quad \quad \quad \cdot V\left( \frac{x -(\xi + \langle x \rangle^{\rho} \eta)}{\sqrt{2}}\right) \log \left\langle \sqrt{2} \left(\frac{\langle x \rangle y}{2} +x \right)\right\rangle dyd\eta.
\end{align*}
Then by the stationary phase theorem we obtain
\begin{align*}
I_1 &= \sum_{k=0}^{N-1} (2\pi)^{-d/2} \langle x \rangle^{(1+ \rho)d} \frac{\left(\frac{1}{\langle x \rangle^{1+ \rho}}\right)^{k+d}}{(2i)^k k!} \\
& \quad \cdot (-2D_y D_{\eta})^k \left\{ \chi \left(\frac{y}{\epsilon_0}\right) \chi \left( \frac{\eta}{\epsilon_1}\right) V\left( \frac{x -(\xi + \langle x \rangle^{\rho} \eta)}{\sqrt{2}}\right) \log \left\langle \sqrt{2} \left(\frac{\langle x \rangle y}{2} +x \right)\right\rangle \right\} \Biggm\vert_{y=\eta =0} \\
& \quad \quad \quad \quad + \tilde{R}_{N} \\
& = (2\pi)^{-d/2} \sum_{k=0}^{N-1} \frac{1}{(2i)^k k!} (-2D_y D_{\eta})^k \left\{ V\left( \frac{x -(\xi + \eta)}{\sqrt{2}}\right) \log \left\langle \sqrt{2} \left(\frac{y}{2} +x \right)\right\rangle \right\} \Biggm\vert_{y=\eta=0} + \tilde{R}_{N}
\end{align*}
where $\tilde{R}_{N}$ satisfies
\begin{align*}
|\tilde{R}_{N} (x, \xi)| \lesssim \langle x \rangle^{-N + (2d+1)\rho}.
\end{align*}
Concerning differentials of $\tilde{R}_{N}$, we also have $|\partial^{\alpha} _x \partial^{\beta} _{\xi} \tilde{R}_{N} (x, \xi)| \lesssim \langle x \rangle^{-N + (2d+1)\rho}$ for all $\alpha, \beta \in \N_0$. This is proved by changing the order of differentiation and integration as follows. We compute as
\begin{align*}
\partial^{\alpha} _x \partial^{\beta} _{\xi} \tilde{R}_{N} (x, \xi) &= \partial^{\alpha} _x \partial^{\beta} _{\xi} I_1 \\
&\quad- \underbrace{\partial^{\alpha} _x \partial^{\beta} _{\xi}  \sum_{k=0}^{N-1} \frac{(2\pi)^{-\frac{d}{2}}}{(2i)^k k!} (-2D_y D_{\eta})^k \left\{ V\left( \frac{x -(\xi + \eta)}{\sqrt{2}}\right) \log \left\langle \sqrt{2} \left(\frac{y}{2} +x \right)\right\rangle \right\} \Biggm\vert_{y=\eta=0}}_{=:J(x, \xi, N)} \\
&= (2\pi)^{-d} \int_{\R^{2d}} e^{-iy\eta} \chi_1 (x, y, \eta) \partial^{\alpha} _x \partial^{\beta} _{\xi} \left\{V\left( \frac{x -(\xi + \eta)}{\sqrt{2}}\right) \log \left\langle \sqrt{2} \left(\frac{y}{2} +x \right)\right\rangle \right\}dyd\eta \\
&  \quad \quad \quad \quad \quad \quad \quad-J(x, \xi, N) \\
& \quad \quad \quad \quad + \sum_{\substack{\gamma + a =\alpha, \delta +b=\beta \\ |\gamma| + |\delta| \ne 0}} \frac{1}{(2\pi)^{d}} \int_{\R^{2d}} e^{-iy\eta} \partial^{\gamma} _x \partial^{\delta} _{\xi} \{\chi_1 (x, y, \eta) \} \\
& \quad \quad \quad \quad \quad \quad \quad \quad \quad \cdot \partial^{a} _x \partial^{b} _{\xi} \left\{V\left( \frac{x -(\xi + \eta)}{\sqrt{2}}\right) \log \left\langle \sqrt{2} \left(\frac{y}{2} +x \right)\right\rangle \right\}dyd\eta.
\end{align*}
Then we apply the stationary phase theorem to the first term as above. However the $0 \sim (N-1)$-th terms appearing from the expansion cancel with $J(x, \xi, N)$ and only the remainder term $R_{N, 3}$ is left. Concerning $R_{N, 3}$, we have $|R_{N, 3}| \lesssim \langle x \rangle^{-N + (2d+1)\rho}$ as before. Therefore it suffices to show that
\begin{align*}
\left|\int_{\R^{2d}} e^{-iy\eta} \partial^{\gamma} _x \partial^{\delta} _{\xi} \{\chi_1 (x, y, \eta) \}
\partial^{a} _x \partial^{b} _{\xi} \left\{V\left( \frac{x -(\xi + \eta)}{\sqrt{2}}\right) \log \left\langle \sqrt{2} \left(\frac{y}{2} +x \right)\right\rangle \right\}dyd\eta \right| \lesssim \langle x \rangle^{-N + (2d+1)\rho}
\end{align*}  
holds for $\gamma, \delta, a, b$ as in the above sum. We again apply the stationary phase theorem to the integral. By using $|\gamma| + |\delta| \ne 0$ and a support property of $\chi_1$, each term in the expansion vanishes except for the remainder term $R_{N, 4}$. Since $|R_{N, 4}| \lesssim \langle x \rangle^{-N + (2d+1)\rho}$ holds, we obtain the desired estimate.
For $I_2$, by taking care of the fact that
\begin{align*}
&L_1 =\langle \langle x \rangle^{-\rho} \eta \rangle^{-2} (1+ \langle x \rangle^{-2\rho} \eta D_y) \quad satisfies \quad L_1 e^{-iy \eta} = e^{-iy \eta}, \\
&L_2 = -|y|^{-2} yD_{\eta} \quad satisfies \quad L_2 e^{-iy \eta} = e^{-iy \eta}
\end{align*}
and a support property of $\chi_j$, if we repeat integration by parts sufficiently many times, we obtain $|I_2 (x, \xi)| \lesssim \langle x \rangle^{-N}$ for any $N \in \N$. Again by changing the order of differentiation and integration, for any $\alpha, \beta \in \N^d _0$, we obtain $|\partial^{\alpha} _x \partial^{\beta} _{\xi}I_2 (x, \xi)| \lesssim \langle x \rangle^{-N}$ for all $N \in \N$. For $I_3$, we notice that $\epsilon_1 \langle x \rangle^{\rho} \le |\eta| $ holds in the support of $\chi_3$. Then we use the fact
\begin{align*}
&L_3 = -|\eta|^{-2} \eta D_{y} \quad satisfies \quad L_3 e^{-iy \eta} = e^{-iy \eta}, \\
&L_4 = \langle y \rangle^{-2} (1-yD_{\eta}) \quad satisfies \quad L_4 e^{-iy \eta} = e^{-iy \eta}
\end{align*}
and repeat integration by parts. As a result we get $|I_3 (x, \xi)| \lesssim \langle x \rangle^{-N}$ for any $N \in \N$. As above, for any $\alpha, \beta \in \N^d _0$, $|\partial^{\alpha} _x \partial^{\beta} _{\xi}I_3 (x, \xi)| \lesssim \langle x \rangle^{-N}$ holds for all $N \in \N$. We also calculate $c(x, \xi)$ in the same manner. Then for any $\alpha, \beta \in \N^d _0$, we obtain
\begin{align*}
|\partial^{\alpha} _x \partial^{\beta} _{\xi} (b(x, \xi) -c(x, \xi))| \lesssim \langle x \rangle^{-1}
\end{align*}  
if we take $N \in \N$ in the remainder term large enough. This is because the $0$th term in the sum in $I_1$ cancel with that appearing from $c(x, \xi)$. This implies
\begin{align*}
\left[V^w \left( \frac{x -D_x}{\sqrt{2}}\right), \log \langle \sqrt{2} x \rangle\right] \in OPS(\langle x \rangle^{-1}, g_0).
\end{align*} 
Then by a symbolic calculus in $S(m, g_0)$, we obtain
\begin{align*}
\left[\left[V^w \left( \frac{x -D_x}{\sqrt{2}}\right), \log \langle \sqrt{2} x \rangle\right], \log \langle \sqrt{2} x \rangle \right] \in OPS(1, g_0).
\end{align*}
Next we consider the case where $(\mathcal{A}, \mathcal{B}) = (x, D_x)$. It suffices to show that 
\begin{align*}
[f^w (x, D_x), \log \langle \sqrt{2} D_x \rangle] \in OPS(1, g_0)
\end{align*}
holds for any $f \in S(\langle x \rangle^{-1}, g_0)$. However calculations are similar to the above one and we only give an outline. The symbol of $f^w (x, D_x) \log \langle \sqrt{2} D_x \rangle$ is given by
\begin{align*}
&(2\pi)^{-2d} \int_{\R^{4d}} e^{-i(y\eta -z\zeta)} c(x+z/2, \xi + \eta) \log \left\langle \sqrt{2} (\xi + \zeta)\right\rangle dydzd\zeta d\eta \\
&= (2\pi)^{-d} \int_{\R^{2d}} e^{iz\zeta} c(x+z/2, \xi) \log \left\langle \sqrt{2} (\xi + \zeta)\right\rangle dzd\zeta.
\end{align*}
Instead of the cut-off functions used above, we set
\begin{align*}
&\chi_1 (\xi, z, \zeta) = \chi \left( \frac{z}{\epsilon_0 \langle \xi \rangle^{\rho}}\right) \chi \left( \frac{\zeta}{\epsilon_1 \langle \xi \rangle}\right), \\
&\chi_2 (\xi, z, \zeta ) = \left\{1- \chi \left( \frac{z}{\epsilon_0 \langle \xi \rangle^{\rho}}\right) \right\}\chi \left( \frac{\zeta}{\epsilon_1 \langle \xi \rangle}\right), \\
&\chi_3 (\xi, z, \zeta ) = 1- \chi \left( \frac{\zeta}{\epsilon_1 \langle \xi \rangle}\right)
\end{align*}
and
\begin{align*}
\tilde{I}_j = (2\pi)^{-d} \int_{\R^{2d}} e^{iz\zeta} \chi_j (\xi, z, \zeta) c(x+z/2, \xi) \log \left\langle \sqrt{2} (\xi + \zeta)\right\rangle dzd\zeta.
\end{align*}
Then we can estimate $\tilde{I}_1$ by the stationary phase theorem and both $\tilde{I}_2$ and $\tilde{I}_3$ are estimated by using integration by parts as above. Since the symbol of $\log \langle \sqrt{2} D_x \rangle f^w (x, D_x)$ are similarly estimated, we have
\begin{align*}
|\partial^{\alpha} _x \partial^{\beta} _{\xi} \sigma (f^w (x, D_x) \log \langle \sqrt{2} D_x \rangle - \log \langle \sqrt{2} D_x \rangle f^w (x, D_x) )| \lesssim \langle \xi \rangle^{-1}
\end{align*}
for any $\alpha$ and $\beta \in \N^d _0$, where $\sigma (A)$ denotes the symbol of $A$. This implies  that 
\begin{align*}
[f^w (x, D_x), \log \langle \sqrt{2} D_x \rangle] \in OPS(\langle \xi \rangle^{-1}, g_0)
\end{align*}
holds and hence we are done.
\end{proof}
Now we are ready to prove Theorem \ref{2405200111}. Our argument is based on \cite{BM}, \cite{MY1} and \cite{M5} (see also \cite{ABG}).
\begin{proof}[Proof of Theorem \ref{2405200111}]
We set $M_1 = [H, iA]$, $M_2 = [[H, iA], iA]$ and $S=[H_0, iA]^{\frac{1}{2}}$. Note that $M_1, M_2$ and $S$ are bounded operators on $L^2 (\R^d)$ by Lemma \ref{2405111618}.

(Step1) For $(\epsilon, z) \in (0, 1) \times \C^{+} \cup (-1, 0) \times \C^{-}$, by using the Mourre inequality $M_1 \gtrsim S^2$, we have
\begin{align*}
\mp \im (H-z-i\epsilon M_1) = \pm \im z \pm \epsilon M_1 \ge \pm \im z
\end{align*}
where $\mp$ corresponds to $\epsilon \in (0, 1) \cup (-1, 0)$. This yields 
\begin{align*}
|\im z| \|u\|^2 _2 \lesssim \langle \mp \im (H-z-i\epsilon M_1)u, u \rangle \lesssim \|(H-z-i\epsilon M_1)u\|_2 \|u\|_2
\end{align*}
and we obtain $|\im z| \|u\|_2 \lesssim \|(H-z-i\epsilon M_1)u\|_2$ for all $u \in D(H)$. Similarly we have $\|(H-z-i\epsilon M_1)^* u\|_2 \gtrsim |\im z| \|u\|_2$. Since $(H-z-i\epsilon M_1)^*$ is densely defined and $L^2 (\R^d) = \Ker (H-z-i\epsilon M_1)^* \oplus \overline{\im (H-z-i\epsilon M_1)}$ holds, $\im (H-z-i\epsilon M_1)$ is dense in $L^2 (\R^d)$. Therefore $R_{\epsilon} (z):= (H-z-i\epsilon M_1)^{-1} : L^2 (\R^d) \to D(H)$ exists and $\|R_{\epsilon} (z)\|_{\mathcal{B} (L^2 (\R^d))} \lesssim |\im z|^{-1}$ holds.

(Step 2) We set $W=S(A+ik)^{-1}$ and $F_{\epsilon} (z) = WR_{\epsilon} (z) W^*$ for sufficiently large $k >0$. Then we have
\begin{align*}
\|Se^{i\lambda A} u\|^2 _2 &= \|Su\|^2 _2 + \int_{0}^{\lambda} \frac{d}{d\mu} \langle S^2 e^{i\mu A}u, e^{i\mu A}u \rangle d\mu \\
& = \|Su\|^2 _2 + \int_{0}^{\lambda} \langle [S^2, iA] e^{i\mu A}u, e^{i\mu A}u \rangle d \mu \\
& = \|Su\|^2 _2 + \int_{0}^{\lambda} \langle [[H_0, iA], iA] e^{i\mu A}u, e^{i\mu A}u \rangle d \mu \\
& \le \|Su\|^2 _2 + c \int_{0}^{\lambda} \|Se^{i\mu A} u\|^2 _2 d\mu
\end{align*}
for some $c>0$ and all $u \in \mathcal{S} (\R^d)$, where we have used $\langle [[H_0, iA], iA]u, u\rangle \le c \langle [H_0, iA]u, u \rangle$, which is found in the proof of Lemma \ref{2405111618}. By the Gronwall inequality and density argument, $\|Se^{i\lambda A} u\|_2 \le e^{\frac{\lambda c}{2}} \|Su\|_2$ holds for all $\lambda >0$ and $u \in L^2 (\R^d)$. This yields
\begin{align*}
\|S(A+ik)^{-1} u\|_2 \le \int_{0}^{\infty} \|Se^{i\lambda (A+ik)} u\|_2 d\lambda \le \int_{0}^{\infty} e^{(-k+\frac{c}{2})\lambda} \|Su\|_2 d\lambda \lesssim \|Su\|_2
\end{align*}
if $k> \frac{c}{2}$. By the definition of $F_{\epsilon} (z)$, we obtain $\|F_{\epsilon} (z)\|_{\mathcal{B} (L^2 (\R^d))} \lesssim |\im z|^{-1}$. Now we show
\begin{align}
\partial_{\epsilon} F_{\epsilon} (z) = 2kiF_{\epsilon} (z) -SR_{\epsilon} (z) W^* -WR_{\epsilon} (z) S- \epsilon WR_{\epsilon} (z) M_2 R_{\epsilon} (z)W^*. \label{2405211724}
\end{align}
Note that $R_{\epsilon} (z) (D(A)) \subset D(A)$ holds. This inclusion is important in the Mourre theory and proofs may be well-known. However we cannot find a reference completely including our setting and we give a proof in Appendix \ref{202405222307}.
Hence we obtain
\begin{align*}
\frac{d}{d\lambda} \langle We^{-i\lambda A} R_{\epsilon} (z) e^{i\lambda A} W^* u, v \rangle \Biggm\vert_{\lambda =0} &= \langle [R_{\epsilon} (z), iA]W^* u, W^* v \rangle \\
& = \langle R_{\epsilon} (z)[iA, H-z-i\epsilon M_1] R_{\epsilon} (z) W^* u, W^* v\rangle \\
& = \langle WR_{\epsilon} (z) (-M_1 +i\epsilon M_2)R_{\epsilon} (z) W^* u, v \rangle
\end{align*}
for all $u, v \in \mathcal{S} (\R^d)$. On the other hand $\frac{d}{d\lambda} e^{i\lambda A} W^* u \vert_{\lambda =0}= iAW^* u = (iS-kW^*)u$ yields
\begin{align*}
\frac{d}{d\lambda} \langle R_{\epsilon} (z) e^{i\lambda A} W^* u, e^{i\lambda A} W^* v \rangle \Biggm\vert_{\lambda =0} &= \langle R_{\epsilon} (z) (iS-kW^*)u, W^* v \rangle + \langle R_{\epsilon} (z) W^* u, (iS-kW^*)v \rangle \\
& = -2k \langle F_{\epsilon} (z)u, v \rangle +i \langle WR_{\epsilon} (z)Su, v \rangle -i \langle SR_{\epsilon} (z)W^* u, v \rangle
\end{align*} 
since $S$ is self-adjoint. Therefore
\begin{align*}
WR_{\epsilon} (z) (-M_1 +i\epsilon M_2)R_{\epsilon} (z) W^* = -2kF_{\epsilon} (z) + iWR_{\epsilon} (z)S-iSR_{\epsilon} (z)W^*
\end{align*}
holds. By using $\partial_{\epsilon} F_{\epsilon} (z) =W\partial_{\epsilon} R_{\epsilon} (z)W^* = iWR_{\epsilon} (z) M_1 R_{\epsilon} (z) W^*$ we obtain (\ref{2405211724}).

(Step 3) By Step 2, for $(\epsilon, z) \in (0, 1) \times \C^{+}$, we obtain
\begin{align*}
\|F_{\epsilon} (z) u\|^2 _2 \lesssim \|SR_{\epsilon} (z)W^* u\|^2 _2 &= \langle S^2 R_{\epsilon} (z) W^* u, R_{\epsilon} (z) W^* u\rangle \\
& \lesssim \langle M_1 R_{\epsilon} (z) W^* u, R_{\epsilon} (z) W^* u \rangle \\
& \le \frac{1}{\epsilon} \{ \epsilon \langle M_1 R_{\epsilon} (z) W^* u, R_{\epsilon} (z) W^* u \rangle + \im z \|R_{\epsilon} (z)W^* u\|^2 _2 \} \\
& = - \frac{1}{\epsilon} \im \langle (H-z-i\epsilon M_1)R_{\epsilon} (z)W^* u, R_{\epsilon} (z)W^* u \rangle \\
& = - \frac{1}{\epsilon} \im \langle u, F_{\epsilon} (z)u \rangle \lesssim \frac{1}{|\epsilon|} \|u\|_2 \|F_{\epsilon} (z)u\|_2
\end{align*}
This yields $\|F_{\epsilon} (z)\|_{\mathcal{B} (L^2 (\R^d))} \lesssim \frac{1}{|\epsilon|}$ and $\|SR_{\epsilon} (z) W^* \|_{\mathcal{B} (L^2 (\R^d))} \lesssim \frac{1}{|\epsilon|^{1/2}} \|F_{\epsilon} (z)\|^{1/2} _{\mathcal{B} (L^2 (\R^d))} $. By a similar calculation, this is true for $(\epsilon, z) \in (-1, 0) \times \C^{-}$. Then
\begin{align*}
\|WR_{\epsilon} (z) S\|_{\mathcal{B} (L^2 (\R^d))} = \|SR_{-\epsilon} (\bar{z}) W^*\|_{\mathcal{B} (L^2 (\R^d))} \lesssim \frac{1}{|\epsilon|^{1/2}} \|F_{-\epsilon} (\bar{z})\|^{1/2} _{\mathcal{B} (L^2 (\R^d))} &= \frac{1}{|\epsilon|^{1/2}} \|F_{\epsilon} (z)^*\|^{1/2} _{\mathcal{B} (L^2 (\R^d))} \\
& = \frac{1}{|\epsilon|^{1/2}} \|F_{\epsilon} (z)\|^{1/2} _{\mathcal{B} (L^2 (\R^d))} 
\end{align*}
holds for $(\epsilon, z) \in (0, 1) \times \C^{+}$. By these estimates, we have
\begin{align*}
|\langle WR_{\epsilon} (z) M_2 R_{\epsilon} (z) W^* u, u \rangle| &= |\langle M_2 R_{\epsilon} (z) W^* u, R_{\epsilon} (z)^* W^* u \rangle| \\
& \lesssim \|SR_{\epsilon} (z) W^* u\|_2 \|SR_{\epsilon} (z)^* W^* u\|_2 \lesssim \frac{1}{|\epsilon|} \|F_{\epsilon} (z)\|_{\mathcal{B} (L^2 (\R^d))}  \|u\|_2 \|v\|_2
\end{align*}
where we have used $| \langle M_2 u, u \rangle | \lesssim \langle [H_0, iA]u, u\rangle$. Hence by (\ref{2405211724})
\begin{align*}
\|\partial_{\epsilon} F_{\epsilon} (z)\|_{\mathcal{B} (L^2 (\R^d))} \lesssim \|F_{\epsilon} (z)\|_{\mathcal{B} (L^2 (\R^d))} + \frac{1}{\epsilon ^{1/2}} \|F_{\epsilon} (z)\|^{1/2} _{\mathcal{B} (L^2 (\R^d))} 
\end{align*}
holds for $(\epsilon, z) \in (0, 1) \times \C^{+}$. By integrating in $(\epsilon, 1)$ we obtain
\begin{align*}
\|F_{\epsilon} (z)\|_{\mathcal{B} (L^2 (\R^d))} &\lesssim \|F_{1} (z)\|_{\mathcal{B} (L^2 (\R^d))} + \int_{\epsilon}^{1} \left( \|F_{t} (z)\|_{\mathcal{B} (L^2 (\R^d))} + t^{-1/2} \|F_{t} (z)\|^{1/2} _{\mathcal{B} (L^2 (\R^d))} \right)dt \\
& \lesssim 1 + \int_{\epsilon}^{1} \frac{1}{t} dt \lesssim 1- \log \epsilon.
\end{align*}
By using this inequality and again integrating in $(\epsilon, 1)$ we obtain $\|F_{\epsilon} (z)\|_{\mathcal{B} (L^2 (\R^d))} \lesssim 1$.

(Step 4) By Step 3, we have $\|S(A+ik)^{-1} R_{\epsilon} (z) (A-ik)^{-1} S\|_{\mathcal{B} (L^2 (\R^d))} \lesssim 1$. By taking $\epsilon \to 0$ we obtain $\|S(A+ik)^{-1} (H-z)^{-1} (A-ik)^{-1} S\|_{\mathcal{B} (L^2 (\R^d))} \lesssim 1$. We transform
\begin{align*}
&\langle \log \langle x \rangle \rangle^{-1} (H-z)^{-1} \langle \log \langle x \rangle \rangle^{-1} \\
&= \langle \log \langle x \rangle \rangle^{-1} (A+ik) S^{-1} \cdot S(A+ik)^{-1} (H-z)^{-1} (A-ik)^{-1} S \cdot S^{-1} (A-ik) \langle \log \langle x \rangle \rangle^{-1}.
\end{align*}
We recall that $\|u\|^2 _2 \lesssim \left\langle \left(\frac{4x^2}{\langle \sqrt{2} x \rangle^2} + \frac{4D^2 _x}{\langle \sqrt{2} D_x \rangle^2}\right)u, u \right\rangle$ and $S^2 = [H_0, iA] = U\left(\frac{4x^2}{\langle \sqrt{2} x \rangle^2} + \frac{4D^2 _x}{\langle \sqrt{2} D_x \rangle^2}\right)U^*$ yields $\|u\|_2 \lesssim \|Su\|_2$ and hence $S^{-1}$ is bounded on $L^2 (\R^d)$. Since $A \in OPS(\langle \log \langle x \rangle \rangle, g_0)$,
\begin{align*}
&\|\langle \log \langle x \rangle \rangle^{-1} (A+ik) S^{-1} \|_{\mathcal{B} (L^2 (\R^d))} \lesssim \|\langle \log \langle x \rangle \rangle^{-1} (A+ik) \|_{\mathcal{B} (L^2 (\R^d))} \lesssim 1, \\
& \|S^{-1} (A-ik) \langle \log \langle x \rangle \rangle^{-1}\|_{\mathcal{B} (L^2 (\R^d))} \lesssim \|(A-ik) \langle \log \langle x \rangle \rangle^{-1}\|_{\mathcal{B} (L^2 (\R^d))} \lesssim 1
\end{align*}
holds. Therefore we obtain the desired estimate.
\end{proof}
\begin{remark}\label{2405230157}
By the ordinary way it is possible to prove that 
\begin{align*}
\C^{\pm} \ni z \mapsto \langle \log \langle x \rangle \rangle^{-\rho} (H-z)^{-1} \langle \log \langle x \rangle \rangle^{-\rho} \in \mathcal{B} (L^2 (\R^d))
\end{align*}
is globally H\"older continuous for any $\rho >1$. However we omit its proof since we do not use this result in the remaining part of this paper.
\end{remark}
\subsection{\textbf{Orthonormal Strichartz estimate for the repulsive Hamiltonian}}\label{2405211228}
In this subsection we observe the orthonormal Strichartz estimate for $H$ as in Subsection \ref{2405091758}. For $\phi \in \mathcal{S} (\R^d)$, by the Mehler formula, 
\begin{align*}
e^{-itH_0} \phi (x) = \frac{1}{(2\pi i \sinh 2t)^{d/2}} \int_{\R^d} e^{i\{(x^2 +y^2) \cosh 2t -2xy\} /2\sinh 2t} \phi (y)dy
\end{align*}
holds. By this formula, it is proved in \cite{KaYo} that $\|e^{-itH_0} u\|_{\infty} \lesssim |t|^{-k} \|u\|_1$ holds for any $k \ge \frac{d}{2}$. Then by Theorem \ref{2405021538}, for any $(q, r, k, \beta)$ satisfying $\frac{2}{q} = 2k\left( \frac{1}{2} - \frac{1}{r} \right), (q, r, k) \ne (2, \infty, 1), k \ge \frac{d}{2}$ and assumptions in Theorem \ref{2405021538}, we obtain
\begin{align}
\left\| \sum_{j=0}^{\infty} \nu_j |e^{-itH_0} f_j|^2 \right\|_{L^{\frac{q}{2}} _t  L^{\frac{r}{2}} _x} \lesssim \|\nu\|_{l^{\beta}}. \label{2405231303}
\end{align}
Furthermore, by the Keel-Tao theorem (\cite{KT}), we have
\begin{align*}
&\|e^{-itH_0} u\|_{L^q _t L^{r, 2} _x} \lesssim \|u_0\|_2, \\
&\left\| \int_{0}^{t} e^{-i(t-s)H_0} F(s) ds \right\|_{L^q _t L^{r, 2} _x} \lesssim \|F\|_{L^{\tilde{q}'} _t L^{\tilde{r}', 2} _x}
\end{align*}
for any $k$-admissible pair $(q, r)$ and $(\tilde{q}, \tilde{r})$. First we give a proof of Theorem \ref{2405231314} based on the perturbation argument proved in \cite{H1}.
\begin{proof}[Proof of Theorem \ref{2405231314}]
We split $H=H_0 +V = H_0 +v_1 ^* v_2$, where $v_1 = |V|^{1/2}$ and $v_2 = |V|^{1/2} \sgn V$. Since $|v_1 (x)| \lesssim \langle \log \langle x \rangle \rangle^{-1}$ and $|v_2 (x)| \lesssim \langle \log \langle x \rangle \rangle^{-1}$ hold, by Theorem \ref{2405200111}, $v_1$ is $H_0$-smooth and $v_2$ is $H$-smooth in the sense of Kato (see \cite{KY}, \cite{KatoYajima}). Then we can adopt Theorem 2.3 in \cite{H1} with (\ref{2405231303}) and obtain the desired estimates.
\end{proof}
As a corollary of the Strichartz estimates, we consider the $L^p$-mapping property of resolvents called the uniform Sobolev estimates. 
\begin{lemma}\label{2406200852}
Suppose $r \in (2, \infty)$ if $d=1$ or $2$ and $r \in (2, \frac{2d}{d-2})$ if $d \ge 3$. Then 
\begin{align*}
\|(H_0 -z)^{-1} f\|_{r, 2} \lesssim \|f\|_{r', 2}
\end{align*}
holds for any $f \in \mathcal{S}$ and $z \in \C^{\pm}$.
\end{lemma}
\begin{proof}
We mimick the argument due to Duyckaerts (see \cite{BM}). We set $u(t)=e^{izt} (H_0 -z)^{-1} f$ for $f \in \mathcal{S}$. Since $u$ satisfies
\begin{align*}
\left\{
\begin{array}{l}
i\partial_{t} u= -H_0 u + e^{izt} f \\
u(0)=(H_0 -z)^{-1} f,
\end{array}
\right.
\end{align*} 
by the Strichartz estimate, we obtain
\begin{align*}
\|u\|_{L^2 ([-T, T]; L^{r, 2} _x)} \lesssim \|(H_0 -z)^{-1} f\|_2 + \|e^{izt} f\|_{L^2 ([-T, T]; L^{r', 2} _x)}
\end{align*}
where the implicit constant is independent of $T >0$. By the definition of $u$, 
\begin{align*}
\|(H_0 -z)^{-1} f\|_{r, 2} \lesssim \|f\|_{r', 2} + \frac{1}{\|e^{izt}\|_{L^2 ([-T, T])}} \|(H_0 -z)^{-1} f\|_2
\end{align*}
holds. By using $\|e^{izt}\|_{L^2 ([-T, T])} \ge \sqrt{T}$ and taking $T \to \infty$, we have the desired result.
\end{proof}
We extend the above lemma to $H$. The proof is based on the perturbation method in \cite{BM} and our Kato-Yajima type estimates.
\begin{thm}\label{2406201017}
Let $r$ be as in Lemma \ref{2406200852} and $V$ be as in Theorem \ref{2405231314}. If $V \in L^{\frac{r}{r-2}, \infty}$ holds, we obtain
\begin{align}
\|(H -z)^{-1} f\|_{r, 2} \lesssim \|f\|_{r', 2} \label{2406201112}
\end{align}
for any $f \in \mathcal{S}$ and $z \in \C^{\pm}$. 
\end{thm}
\begin{proof}
By the resolvent identity and Cauchy-Schwarz inequality,
\begin{align*}
|\langle (H-z)^{-1} f, g \rangle| \lesssim |\langle (H_0 -z)^{-1} f, g \rangle| + \||V|^{\frac{1}{2}} (H-z)^{-1} f\|_2 \||V|^{\frac{1}{2}} (H_0 -z)^{-1} g\|_2 
\end{align*}
holds for all $f, g \in \mathcal{S}$. By Theorem \ref{2405200111}, $|V(x)| \lesssim \langle \log \langle x \rangle \rangle^{-2}$ and resolvent identity,
\begin{align*}
|\langle |V|^{\frac{1}{2}} (H-z)^{-1} f, g \rangle | &\le |\langle |V|^{\frac{1}{2}} (H_0 -z)^{-1} f, g \rangle | + |\langle |V|^{\frac{1}{2}} (H_0 -z)^{-1} f, |V|^{\frac{1}{2}} (H-z)^{-1} |V|^{\frac{1}{2}} g \rangle | \\
& \lesssim \|g\|_2 \||V|^{\frac{1}{2}} (H_0 -z)^{-1} f\|_2
\end{align*}
holds. This yields $ \||V|^{\frac{1}{2}} (H-z)^{-1} f\|_2 \lesssim  \||V|^{\frac{1}{2}} (H_0 -z)^{-1} f\|_2$. Then H\"older's inequality and Lemma \ref{2406200852} implies
\begin{align*}
\||V|^{\frac{1}{2}} (H_0 -z)^{-1} f\|_2 \lesssim \|(H_0 -z)^{-1} f\|_{r, 2} \lesssim \|f\|_{r', 2}.
\end{align*}
By duality and Lemma \ref{2406200852}, we obtain (\ref{2406201112}).
\end{proof}
Next we prove Theorem \ref{2406161602}. The following lemma implies the local regularity of $(H-z)^{-1} f$ for $f \in \mathcal{S}$. We set $p_0 (x, \xi) = \xi^{2} -x^2$, $p(x, \xi) =p_0 (x, \xi) +V(x)$ and $P = p (x, D) = p^{w} (x, D)$.
\begin{lemma}[Local regularity]\label{2406161703}
Assume $V \in C^{\infty} (\R^d; \R)$ satisfies $|\partial^{\alpha} _x V(x)| \lesssim \langle x \rangle^{2-\epsilon - |\alpha|}$ for some $\epsilon >0$ and any $\alpha \in \N^{d} _{0}$. If $(-\Delta -x^2 +V -z)u =f \in \mathcal{S}$ holds for some $z \in \C^{+}$ and $u \in L^2 (\R^d)$, then we have $u \in C^{\infty} (\R^d)$.
\end{lemma}
\begin{proof}
The proof is based on a standard parametrix construction. Let $\chi \in C^{\infty} _0 (\R^d; [0, 1])$ satisfy $\chi (x) =1$ if $|x| \le M$ and $\chi (x) =0$ if $|x| \ge 2M$ for some fixed $M>0$. We set $\bar{\chi} _{R} (x) = 1- \chi_{R} (x) = 1- \chi (x/R)$. Since $\chi (x)u = \chi (x) \chi_{R} (D)u + \chi (x) \bar{\chi}_{R} (D)u$ and $\chi (x) \chi_{R} (D)u \in \mathcal{S}$ hold, it suffices to show $\chi (x) \bar{\chi}_{R} (D)u \in \mathcal{S}$ for sufficiently large $R>0$. Now we define 
\begin{align*}
&a_0 (x, \xi) = \frac{\chi (x) \bar{\chi}_{R} (\xi)}{p(x, \xi) -z}, \quad a_1 (x, \xi) =i \frac{\nabla _{\xi} a_0 (x, \xi) \cdot (-2x + \nabla _x V(x))}{p(x, \xi) -z} \\
&a_j (x, \xi) =\frac{ \left\{ \sum_{|\alpha| =j} \frac{1}{i^{|\alpha|} \alpha !} \partial^{\alpha} _{\xi} a_0 (x, \xi) \partial^{\alpha} _x (-x^2 +V(x)) + \cdots + \sum_{|\alpha| =1} \frac{1}{i} \partial^{\alpha} _{\xi} a_{j-1} (x, \xi) \partial^{\alpha} _x (-x^2 +V(x))  \right\}}{-(p(x, \xi) -z)}
\end{align*}
for $j \ge 2$. Then by a support property of $\chi$, $a_0 \in S(\langle \xi \rangle^{-2} \langle x \rangle^{-\infty}, g)$ and $a_1 \in S(\langle \xi \rangle^{-5} \langle x \rangle^{-\infty}, g)$ holds, where $g = \frac{dx^2}{\langle x \rangle^2} + \frac{d\xi^2}{\langle \xi \rangle^2}$. Furthermore by an induction argument, we obtain $a_j \in S(\langle \xi \rangle^{-4-j} \langle x \rangle^{-\infty}, g)$ for $j \ge 2$. By our construction, 
\begin{align*}
\chi (x) \bar{\chi}_{R} (D) = \left( \sum_{j=0}^{N} a_j (x, D) \right) (P-z) + r_N (x, D)
\end{align*}
holds for some $r_N \in S(\langle \xi \rangle^{-4-N} \langle x \rangle^{-\infty}, g)$. Therefore
\begin{align*}
\chi (x) \bar{\chi}_{R} (D) u = \left( \sum_{j=0}^{N} a_j (x, D) \right)f + r_N (x, D) (P-z)^{-1} f \in \bigcap_{N \in \N} H^{4+N, \infty} = \mathcal{S}
\end{align*}
holds and we have the desired result.
\end{proof}
The next lemma implies that $(P-z)^{-1} f \in \mathcal{S}$ holds outside the characteristic set of $p_0$.
\begin{lemma}\label{2406161842}
Let $P$ and $z$ be as in Lemma \ref{2406161703} and we assume $(P-z)u =f \in \mathcal{S}$ for some $u \in L^2 (\R^d)$. Then $\chi (x)u, \chi (D)u \in \mathcal{S}$, where $\chi$ is as in the proof of Lemma \ref{2406161703}. Furthermore, for $a \in S(1, g)$ supported in $\Omega (R, \gamma, 1) \cup \Omega (R, \gamma, 2) := \{(x, \xi) \in \R^{2d} \mid |x| >R, |\xi| >R, |\xi|>\gamma |x|\} \cup \{(x, \xi) \in \R^{2d} \mid |x| >R, |\xi| >R, |x|>\gamma |\xi|\}$ for some $\gamma >1$ and sufficiently large $R>0$, we have $a(x, D)u \in \mathcal{S}$.
\end{lemma}
\begin{proof}
$\chi (x)u \in \mathcal{S}$ is a consequence of Lemma \ref{2406161703}. $\chi (D)u \in \mathcal{S}$ is proved similarly. For $(x, \xi) \in \Omega (R, \gamma, 1) \cup \Omega (R, \gamma, 2)$, we have
\begin{align*}
\left(|\xi^2 -x^2 +V(x)| \gtrsim |\xi|^2 \quad and \quad |x| \lesssim |\xi| \right) \quad or \quad \left(|\xi^2 -x^2 +V(x)| \gtrsim |x|^2 \quad and \quad |\xi| \lesssim |x| \right)
\end{align*}
respectively if $R>0$ is large enough. Then by a similar parametrix construction as in Lemma \ref{2406161703}, we have $a_j \in S(\langle x \rangle^{-j/2} \langle \xi \rangle^{-j/2}, g)$ and $r_N \in S(\langle x \rangle^{-N/2} \langle \xi \rangle^{-N/2}, g)$ such that
\begin{align*}
a(x, D)u = \left( \sum_{j=0}^{N} a_j (x, D) \right)f + r_N (x, D) (P-z)^{-1} f
\end{align*}
holds. This implies $a(x, D)u \in \bigcap_{N \in \N} H^{N, N} = \mathcal{S}$.
\end{proof}
Note that Lemma \ref{2406161842} holds if we replace $a(x, D)$ with $a^w (x, D)$ by the change of quantization formula. The next lemma is used to deduce a slight decay of $(H-z)^{-1} f$, which is important in the proof of  Theorem \ref{2406181741}. In the case of super quadratic repulsive potentials, we may use the result in \cite{Ta4}. However its method does not apply to the present case. We also remark that weighted commutator method in \cite{I}, \cite{I2} is difficult to apply since $\im z \ne 0$. We define $N= -\Delta + x^2$.
\begin{lemma}\label{2406162127}
For any $z \in \C^{+}$, there exists $\alpha >0$ such that $(H_0 -z)^{-1} : D(N^{\alpha}) \to D(N^{\alpha})$. Furthermore, for any $z \in \C^{+}$ and $V \in \tilde{S}^0 (\R^d)$, there exist $c_0 >0$ and $\alpha >0$ such that $(H_0 +cV -z)^{-1} : D(N^{\alpha}) \to D(N^{\alpha})$ if $c \in (0, c_0)$.
\end{lemma}
\begin{proof}
By the MDFM decomposition of $e^{-itH_0} = M(t)D(t) \mathcal{F} M(t)$, where $M(t) f(x) =e^{i \frac{x^2}{2\tanh 2t}} f(x)$ and $D(t)f(x) = \frac{1}{(i \sinh 2t )^{d/2}} f(\frac{x}{\sinh 2t})$, we have
\begin{align*}
&\|x^2 e^{-itH_0} u\|_2 = |\sinh 2t|^2 \|D^2 _x M(t)u\|_2 \lesssim e^{4|t|} \|Nu\|_2, \\
&\|Ne^{-itH_0} u\|_2 \lesssim \|H_0 e^{-itH_0} u\|_2 + \|x^2 e^{-itH_0} u\|_2 \lesssim e^{4|t|} \|Nu\|_2
\end{align*}
for any $f \in D(N)$. By the complex interpolation, we obtain
\begin{align*}
\|N^{\alpha} e^{-itH_0} u\|_2 \lesssim e^{4\alpha |t|} \|N^{\alpha} u\|_2
\end{align*}
for any $\alpha \in [0, 1]$ and $u \in D(N^{\alpha})$, where the implicit constant is independent of $t$ and $\alpha$. Then by taking $\alpha \in \left(0, \min \{ \frac{\im z}{4}, 1 \} \right)$, 
\begin{align*}
\|N^{\alpha} (H_0 -z)^{-1} u\|_2 \lesssim \int_{0}^{\infty} e^{(4\alpha - \im z ) |t|} dt \cdot  \|N^{\alpha} u\|_2 \lesssim \|N^{\alpha} u\|_2
\end{align*}
holds and we obtain $(H_0 -z)^{-1} : D(N^{\alpha}) \to D(N^{\alpha})$. Since we have
\begin{align*}
\|N^{\alpha} (H_0 -z)^{-1} Vu\|_2 \lesssim \|N^{\alpha} Vu\|_2 \lesssim \|N^{\alpha} u\|_2
\end{align*}
where the last inequality follows from $\|NVu\|_2 \lesssim \|Nu\|_2$ and interpolation, there exists $c_0 >0$ such that $I + c (H_0 -z)^{-1} V$ is invertible on $D(N^{\alpha})$ for any $c \in (0, c_0)$. Then $(H_0 +cV -z)^{-1} = (I + c (H_0 -z)^{-1} V)^{-1} (H_0 -z)^{-1}$ yields
\begin{align*}
\|N^{\alpha} (H_0 +cV -z)^{-1} u\|_2 \lesssim \|N^{\alpha} u\|_2
\end{align*}
and $(H_0 +cV -z)^{-1} : D(N^{\alpha}) \to D(N^{\alpha})$ follows.
\end{proof}
We set $\beta (x, \xi) = \cos (x, \xi) = \frac{x \cdot \xi}{|x| |\xi|}$ and $b(x, \xi) = \tilde{\chi} \left( \tilde{R} \cdot \frac{|\xi|^2 - |x|^2}{|\xi|^2 + |x|^2} \right) \in S^{0, 0}$ for large $\tilde{R} >0$, where $\tilde{\chi} \in C^{\infty} _0 (\R)$ satisfies $\tilde{\chi} (t) =1$ if $|t| \le 1/4$ and $\tilde{\chi} (t) =0$ if $|t| \ge 1/2$. Note that multiplying $b(x, \xi)$ (or $b^w (x, D)$) implies the microlocalization onto $\{(x, \xi) \in T^* \R^d \mid |x| \sim |\xi|\}$ in the phase space. By Lemma \ref{2406161842}, we focus on the behavior on $\supp b$. Our argument below is based on the propagation of singularities and radial source estimates. We employ the argument in \cite{Ta4}, where the Klein-Gordon operator is considered. See also Appendix E in \cite{DZ}. For $k, l, M, N \ge 0$ and $\delta \in [0, 1)$, we define
\begin{align*}
&\lambda (x, \xi) = \lambda_{k, l, M, N, \delta} (x, \xi) = \langle x \rangle^{k} \langle \xi \rangle^{l} \langle \delta x \rangle^{-N} \langle \delta \xi \rangle^{-M}, \\
& \tilde{\lambda} (x, \xi) = \lambda_{k, l} (x, \xi) = \langle x \rangle^{k} \langle \xi \rangle^{l}.
\end{align*}
For $q \in C^{\infty} (T^* \R^d)$, $H_q = \partial_{\xi} q \frac{\partial}{\partial x} - \partial_{x} q \frac{\partial}{\partial \xi}$ denotes the Hamilton vector field generated by $q$.
\begin{lemma}\label{2406172116}
$(i)$ On the support of $b$, we have $H_{p_0} \lambda^2 \lesssim \lambda^2$.

\noindent $(ii)$ On $\supp b \cap \Omega _{\epsilon, r, R, in}$, $H_{p_0} \lambda^2 \lesssim -\lambda^2$ holds if $k>N$ and $l>M$ hold.
\end{lemma}
\begin{proof}
By direct computations, we obtain
\begin{align*}
\{p_0, \lambda^2 \} &= 2\lambda (2\xi \cdot \partial _x \lambda + 2x \cdot \partial _{\xi} \lambda) \\
& = 4 \lambda^2 \beta (x, \xi) |x||\xi| \left\{ \frac{k(1+ \delta^2 x^2) -N(\delta^2 + \delta^2 x^2)}{\langle x \rangle^2 \langle \delta x \rangle^2} +  \frac{l(1+ \delta^2 \xi ^2) -M(\delta^2 + \delta^2 \xi ^2)}{\langle \xi \rangle^2 \langle \delta \xi \rangle^2} \right\}.
\end{align*}
Since $|x| \sim |\xi|$ holds on $\supp b$, we obtain $(i)$. On $\supp b \cap \Omega _{\epsilon, r, R, in}$, $\beta (x, \xi) \sim -1$ and $k>N, l>M$ yields $(ii)$.
\end{proof}
We remark that $H_{p_0} \beta \gtrsim 1$ holds on $\supp b \cap \{(x, \xi) \in T^* \R^d \mid \beta (x, \xi) \in [-1+\epsilon_0, 1+\epsilon_0] \}$ for any sufficiently small $\epsilon_0 >0$. This follows from $H_{p_0} \beta = 2\left( \frac{|x|}{|\xi|} + \frac{|\xi|}{|x|} \right) (1-\beta^2)$. For simplicity we write $S^{k, l} = S \left(\langle \xi \rangle^{k} \langle x \rangle^{l}, \frac{dx^2}{\langle x \rangle^2} + \frac{d\xi^2}{\langle \xi \rangle^2} \right)$. The next lemma gives estimates for cut-off functions. We say that $a \in S^{k, l}$ is microlocally negligible outside $\Omega \subset T^* \R^d$ if there exists $r \in S^{-\infty, -\infty}$ such that $\supp (a+r) \subset \Omega$. We say that $a \in S^{k, l}$ is microlocally contained in $(\Omega_1, \Omega_2)$ if $a$ is elliptic in $\Omega_1$ and microlocally negligible outside $\Omega_2$. 
\begin{lemma}\label{2406180906}
$(i)$ Let $-1 < \beta _1 < \beta_2 <1$. For any $L>1$ and sufficiently small $\epsilon >0$, there exist $a, b_1 \in S^{0, 0}$, $b_2 \in S^{0, -\infty}$ and $e \in S^{-\infty, 0}$ such that $a$ is microlocally contained in $(\Omega_{\beta_1 -\epsilon, \beta_2 +\epsilon, 2r, 2R, mid}, \Omega_{\beta_1 -2\epsilon, \beta_2 +2\epsilon, r, R, mid})$, $b_1$ is microlocally negligible outside $\Omega_{\beta_1-2\epsilon, \beta_1-\epsilon, r, R, mid}$ and $H_{p_0} a^2 \le -La^2 +b^2 _1 +b^2 _2 +e^2$ holds on $\supp b$.

\noindent $(ii)$ For any $\epsilon \in (0, \frac{1}{10})$, there exist $a \in S^{0, 0}$ and $e \in S^{-\infty, 0}$ such that $a$ is microlocally contained in $(\Omega_{\epsilon, 2r, 2R, in}, \Omega_{2\epsilon, r, R, in})$ and $H_{p_0} a^2 \le e^2$ holds on $\supp b$.
\end{lemma} 
\begin{proof}
The construction is essentially the same as in \cite{Ta3}. We fix $\chi \in C^{\infty} (\R; [0, 1])$ such that $\chi (t) =1$ if $t \le 1$, $\chi (t) =0$ if $t \ge 2$ and $\chi ' \le 0$. We set $\bar{\chi}_{R} (x) = 1-\chi_{R} (x) = 1-\chi (|x|/R)$. 

$(i)$ We fix $\rho_{mid} \in C^{\infty} (\R; [0, 1])$ such that $\rho_{mid} (t) =1$ on $\{t \in [\beta_1 -\epsilon, \beta_2 +\epsilon]\}$, $\rho_{mid} (t) =0$ on $\{t \le \beta_1 - 2\epsilon \} \cup \{t \ge \beta_2 + 2\epsilon\}$ and $\rho' _{mid} (t) \le 0$ on $\{t \in [\beta_2 + \epsilon, \beta_2 +2\epsilon]\}$. Then
\begin{align*}
H_{p_0} e^{-2M\beta} \le -Le^{-2M\beta} \quad on \quad \supp b \cap \Omega_{\beta_1 -2\epsilon, \beta_2 +2\epsilon, r, R, mid}
\end{align*}
holds for sufficiently large $M>0$. We set $a(x, \xi) = e^{-M\beta} \rho_{mid} (\beta) \bar{\chi}_{R} (x) \bar{\chi}_r (\xi) \in S^{0, 0}$. Then we have
\begin{align*}
H_{p_0} a^2 \le -La^2 + H_{p_0} (\rho^2 _{mid} (\beta) \bar{\chi}^2 _{R} (x) \bar{\chi}^2 _r (\xi)) e^{-2M\beta}
\end{align*}
on $\supp b$. Since $H_{p_0} \rho^2 _{mid} (\beta) \le C \rho^2 _{1} (\beta)$ holds on $\supp b$ for some $C>0$ and $\rho_1 \in C^{\infty} (\R; [0, 1])$ supported in $[\beta_1 -3\epsilon, \beta_1 - \frac{\epsilon}{2} ]$, we define
\begin{align*}
b_1 (x, \xi) = \sqrt{C} e^{-M\beta} \rho_1 (\beta) \bar{\chi}_{R} (x) \bar{\chi}_r (\xi) \tilde{\chi} \left( \frac{\tilde{R}}{2} \cdot \frac{|\xi|^2 - |x|^2}{|\xi|^2 + |x|^2} \right) \in S^{0, 0}.
\end{align*}
$b_2$ and $e$ are similarly constructed from $H_{p_0} \bar{\chi}^2 _{R}$ and $H_{p_0} \bar{\chi}^2 _r$ respectively. Then we have $H_{p_0} a^2 \le -La^2 +b^2 _1 +b^2 _2 +e^2$ on $\supp b$. Note that $b_2$ is compactly supported with respect to $x$ variable and $e$ is compactly supported with respect to $\xi$ variable.

$(ii)$ We fix $\rho_{in} \in C^{\infty} (\R; [0, 1])$ such that $\rho_{in} (t) =1$ if $t \le -1+\epsilon$, $\rho_{in} (t) =0$ if $t \ge -1+2\epsilon$ and $\rho' _{in} \le 0$. We set $a(x, \xi) = \rho_{in} (\beta) \bar{\chi}_{R} (x) \bar{\chi}_r (\xi)$. Then
\begin{align*}
H_{p_0} \rho_{in} (\beta) \le 0 \quad and \quad H_{p_0} \bar{\chi}_{R} = -\frac{2}{R} |\xi| \beta \chi' (\frac{|x|}{R}) \le 0
\end{align*}
holds on $\supp b \cap \supp \rho_{in} (\beta)$ since $H_{p_0} \beta \gtrsim 1$ holds on $\supp b \cap \{ \beta \in [-1+\epsilon, -1+2\epsilon ] \}$. This yields
\begin{align*}
H_{p_0} a^2 \le \rho^2 _{in} (\beta) \bar{\chi}^2 _{R} (x) H_{p_0} \bar{\chi}^2 _{r} (\xi) \le \rho^2 _{in} (\beta) \bar{\chi}^2 _{R} (x) \tilde{e}^2 (x, \xi)
\end{align*}
on $\supp b$, where $\tilde{e} \in S^{-\infty, 1}$ is compactly supported with respect to $\xi$ variable. By setting
\begin{align*}
e(x, \xi) = \rho_{in} (\beta) \bar{\chi}_{R} (x) \tilde{e} (x, \xi) \tilde{\chi} \left( \frac{\tilde{R}}{2} \cdot \frac{|\xi|^2 - |x|^2}{|\xi|^2 + |x|^2} \right) \in S^{-\infty, 0}
\end{align*}
we obtain the desired estimate. Note that $e$ is also compactly supported with respect to $\xi$ variable.
\end{proof}
We are ready to prove the propagation of regularity and decay $(i)$ and radial source estimates $(ii)$.
In the following theorem we fix $z \in \C^{+}$. We set $H=H_0 + cV (= P_0 +cV)$ for $c \in (0, c_0)$, where $V \in \tilde{S}^0 (\R^d)$ and $c_0$ is as in Lemma \ref{2406162127}. We say $u \in H^{l, k}$ microlocally on $\Omega$ if there exists $a \in S^{0, 0}$ such that $a$ is elliptic in $\Omega$ and $a^w (x, D)u \in H^{l, k}$, where $H^{l, k} := \langle x \rangle^{-k} \langle D \rangle^{-l} L^2 (\R^d)$ is the weighted Sobolev space.
\begin{thm}\label{2406181741}
$(i)$ Let $k, l>0$ and $-1 < \beta_1 < \beta_2 <1$. We assume $u, Bu \in L^2 \cap C^{\infty} (\R^d)$ and $(H-z)u, (H-z)Bu \in \mathcal{S}$, where $B=b^w (x, D)$. If $Bu \in H^{l, k}$ microlocally on $\Omega_{r, R} (\beta_1)$, we have $Bu \in H^{l, k}$ microlocally on $\Omega_{r', R'} (\beta_2)$ for some $r', R' >0$.

\noindent $(ii)$ We assume $u, Bu \in L^2 \cap C^{\infty} (\R^d)$ and $(H-z)u, (H-z)Bu \in \mathcal{S}$. Then, for any $k, l>0$, we have $Bu \in H^{l, k}$ microlocally on $\Omega_{r', R'} (-1)$ for some $r', R'>0$.
\end{thm}
Note that Theorem \ref{2406181741} $(i)$ implies propagation of regularity and decay along bicharacteristics. In our situation, the direction of propagation is affected by the sign of $\im z$ since $p_0 (x, \xi) =0$ occurs if $|x| = |\xi|$. The direction is consistent with ordinary propagation of singularities theorem, e.g. \cite{Hor}.
\begin{proof}
By Lemma \ref{2406172116} and \ref{2406180906}, we obtain
\begin{align*}
(b(H_{p_0} \lambda^2 ) b) (x, \xi) \lesssim (b(-a^2 +b^2 _1 +b^2 _2 +e^2) \lambda^2 b) (x, \xi)
\end{align*} 
for all $(x, \xi) \in T^* \R^d$. In the case of $(ii)$, we set $b_1 =b_2 =0$. Now we define
\begin{align*}
A = a^w (x, D), \Lambda = \tilde{\lambda}^w (x, D), \Lambda_{\delta} = Op^w (\langle \delta x \rangle^{-N} \langle \delta \xi \rangle^{-M}), B_j =b^w _j (x, D), E=e^w (x, D).
\end{align*}
Then by the sharp G\aa rding inequality, we obtain
\begin{align*}
B[H, iA \Lambda_{\delta} \Lambda^2 \Lambda_{\delta} A]B &\lesssim -BA\Lambda_{\delta} \Lambda^2 \Lambda_{\delta} AB + \sum_{j=1}^{2} BB_j \Lambda_{\delta} \Lambda^2 \Lambda_{\delta} B_j B +BE\Lambda_{\delta} \Lambda^2 \Lambda_{\delta} EB \\
& \quad  \quad + \tilde{B} \tilde{A} \Lambda_{\delta} \langle x \rangle^{-\frac{1}{2}} \langle D \rangle^{-\frac{1}{2}} \Lambda^2 \langle D \rangle^{-\frac{1}{2}} \langle x \rangle^{-\frac{1}{2}} \Lambda_{\delta} \tilde{A} \tilde{B}
\end{align*}
mod $OpS^{-\infty, -\infty}$, where $\tilde{A} = \tilde{a}^{w} (x, D) \in S^{0, 0}$ is elliptic in $ \mathcal{C} := \supp a \cup \supp b_1 \cup \supp _2 \cup \supp e$ but supported in sufficiently close neighborhood of $\mathcal{C}$. $\tilde{B} = \tilde{b}^{w} (x, D) \in S^{0, 0}$ satisfies $\tilde{b} (x, \xi) =1$ in $\supp b$ but supported in sufficiently close neighborhood of $\supp b$. By
\begin{align*}
- \im \langle (H-z)u, AC^* C Au \rangle = -\frac{1}{2} \langle u, [H, iAC^* CA]u \rangle + \im z \|CAu\|^2 _2,
\end{align*}
the Cauchy-Schwarz inequality and $\im z >0$, we have
\begin{align*}
- \langle u, [H, iAC^* CA]u \rangle \le \epsilon \|CAu\|^2 _2 + C_{\epsilon} \|CA(P-z)u\|^2 _2.
\end{align*}
From these estimates we obtain
\begin{align*}
\|\Lambda_{\delta} ABu\|^2 _{H^{l, k}} &\lesssim \|\Lambda_{\delta} A (H-z)Bu\|^2 _{H^{l, k}} + \sum_{j=1}^{2} \|\Lambda_{\delta} B_j Bu\|^2 _{H^{l, k}} + \|\Lambda_{\delta} EBu\|^2 _{H^{l, k}} \\
& \quad \quad + \|\Lambda_{\delta} \tilde{A} (\tilde{B} -B)u\|^2 _{H^{l-1/2, k-1/2}} + \|\Lambda_{\delta} \tilde{A} Bu\|^2 _{H^{l-1/2, k-1/2}} + \|u\|^2 _{H^{-L, -L}}
\end{align*}
for any large $L>0$ and $u \in \mathcal{S}$. In the case of $(i)$, by a bootstrap argument, there exist $A_1, \tilde{B}_j, \tilde{E} \in S^{0, 0}$ whose symbols are supported in a sufficiently close neighborhood of $\supp a, \supp b_j, \supp e$ respectively and satisfy
\begin{align}
\|\Lambda_{\delta} ABu\|^2 _{H^{l, k}} &\lesssim \|\Lambda_{\delta} A_1 (H-z)Bu\|^2 _{H^{l, k}} + \sum_{j=1}^{2} \|\Lambda_{\delta} \tilde{B}_j Bu\|^2 _{H^{l, k}} + \|\Lambda_{\delta} \tilde{E} Bu\|^2 _{H^{l, k}} \notag \\ 
& \quad \quad + \|\Lambda_{\delta} A_1 (\tilde{B} -B)u\|^2 _{H^{l-1/2, k-1/2}} + \|u\|^2 _{H^{-L, -L}} \label{2406191524}
\end{align}
for $u \in \mathcal{S}$. In the case of $(ii)$, by interpolation inequalities, 
\begin{align*}
\|\Lambda_{\delta} \tilde{A} Bu\|^2 _{H^{l-1/2, k-1/2}} &\lesssim \|\Lambda_{\delta} ABu\|^2 _{H^{l-1/2, k-1/2}} + \|\Lambda_{\delta} A_2 Bu\|^2 _{H^{l-1/2, k-1/2}} + \|u\|^2 _{H^{-L, -L}} \\
& \lesssim \tilde{\epsilon} \|\Lambda_{\delta} A Bu\|^2 _{H^{l, k}} + \|\Lambda_{\delta} A_2 Bu\|^2 _{H^{l-1/2, k-1/2}} + C_{\tilde{\epsilon}} \|u\|^2 _{H^{-L, -L}}
\end{align*}
holds for any small $\tilde{\epsilon} >0$, where $\sigma (A_2)$ is supported in $\Omega_{-1 + \frac{\epsilon}{2}, -1+ 3\epsilon, \tilde{r}, \tilde{R}, mid}$. By taking $A_3 \in OpS^{0, 0}$ whose symbol $\sigma (A_3)$ is microlocally contained in $(\Omega_{-1 + \frac{\epsilon}{2}, -1+ 3\epsilon, \tilde{r}, \tilde{R}}, \Omega_{-1 + \frac{\epsilon}{4}, -1+ 4\epsilon, \tilde{r}/2, \tilde{R}/2})$, we obtain
\begin{align*}
\|\Lambda_{\delta} A_2 Bu\|^2 _{H^{l-1/2, k-1/2}} &\underbrace{\lesssim}_{\mathrm{support} \; \mathrm{property}} \|\Lambda_{\delta} A_3 Bu\|^2 _{H^{l-1/2, k-1/2}} + \|u\|^2 _{H^{-L, -L}} \\
& \lesssim \|\Lambda_{\delta} A_4 (H-z)Bu\|^2 _{H^{l-1/2, k-1/2}} + \sum_{j=1}^{2} \|\Lambda_{\delta} \tilde{B}_{j, 1} Bu\|^2 _{H^{l, k}} + \|\Lambda_{\delta} \tilde{E}_1 Bu\|^2 _{H^{l, k}} \notag \\ 
& \quad \quad + \|\Lambda_{\delta} A_4 (\tilde{B} -B)u\|^2 _{H^{l-1/2, k-1/2}} + \|u\|^2 _{H^{-L, -L}} \\
& \lesssim \|\Lambda_{\delta} A_4 (H-z)Bu\|^2 _{H^{l-1/2, k-1/2}} + \|\Lambda_{\delta} ABu\|^2 _{H^{l-1/2, k-1/2}} \\
& \quad + \|\Lambda_{\delta} A_4 (\tilde{B} -B)u\|^2 _{H^{l-1/2, k-1/2}} + \|u\|^2 _{H^{-L, -L}} 
\end{align*}
for some $A_4 \in OpS^{0, 0}$, where in the second line we have used propagation of regularity and decay (\ref{2406191524}), in the third line we have used support properties of symbols of $B, \tilde{B}_{2, 1}, \tilde{E}_1$, $\supp \sigma (\tilde{B}_{1, 1}) \subset \Omega_{-1 +\frac{\epsilon}{4}, -1+\frac{\epsilon}{2}, r', R', mid}$ and $a$ is elliptic on $\Omega_{-1 +\frac{\epsilon}{4}, -1+\frac{\epsilon}{2}, r', R', mid}$. By interpolation inequalities, we obtain
\begin{align*}
\|\Lambda_{\delta} A_2 Bu\|^2 _{H^{l-1/2, k-1/2}} &\lesssim \|\Lambda_{\delta} A_4 (H-z)Bu\|^2 _{H^{l-1/2, k-1/2}} + \epsilon' \|\Lambda_{\delta} ABu\|^2 _{H^{l, k}} \\
&\quad \quad + \|\Lambda_{\delta} A_4 (\tilde{B} -B)u\|^2 _{H^{l-1/2, k-1/2}} + C_{\epsilon'} \|u\|^2 _{H^{-L, -L}}
\end{align*}
for any small $\epsilon' >0$. Hence
\begin{align}
\|\Lambda_{\delta} ABu\|^2 _{H^{l, k}} \lesssim \|\Lambda_{\delta} A_5 (H-z)Bu\|^2 _{H^{l, k}} +\|\Lambda_{\delta} A_5 (\tilde{B} -B)u\|^2 _{H^{l-1/2, k-1/2}} + \|u\|^2 _{H^{-L, -L}} \label{2406191802}
\end{align}
holds for some $A_5 \in OpS^{0, 0}$. In $(i)$, by taking $N \gg k$, $M \gg l$ and standard approximation argument (see pp.653 in \cite{Ta3}), (\ref{2406191524}) holds for any $u$ as in the statement $(i)$. Then by taking $\delta \to 0$, we obtain
\begin{align*}
\|ABu\|^2 _{H^{l, k}} &\lesssim \|A_1 (H-z)Bu\|^2 _{H^{l, k}} + \sum_{j=1}^{2} \|\tilde{B}_j Bu\|^2 _{H^{l, k}} \\ 
& \quad + \|\tilde{E} Bu\|^2 _{H^{l, k}} + \|A_1 (\tilde{B} -B)u\|^2 _{H^{l-1/2, k-1/2}} + \|u\|^2 _{H^{-L, -L}} \\
& < \infty
\end{align*} 
by assumptions in $(i)$. Here we have used $u \in \mathcal{S}$ microlocally on the support of $\sigma (\tilde{B} -B)$, which follows from Lemma \ref{2406161842}. Therefore $Bu \in H^{l, k}$ microlocally on $\Omega_{r', R'} (\beta_2)$ for some $r', R' >0$. Next we consider $(ii)$. We take $N, M >0$ such that $(k-N) + (l-M) < 2\alpha$, where $\alpha$ is from Lemma \ref{2406162127}. Since $u \in H^{\alpha, \alpha}$ (and hence $Bu \in H^{\alpha, \alpha}$), again by a standard approximation argument, we obtain (\ref{2406191802}) for any $u$ as in the statement $(ii)$. By taking $\delta \to 0$,
\begin{align*}
\|ABu\|^2 _{H^{l, k}} \lesssim \|A_5 (H-z)Bu\|^2 _{H^{l, k}} + \|A_5 (\tilde{B} -B)u\|^2 _{H^{l-1/2, k-1/2}} + \|u\|^2 _{H^{-L, -L}} < \infty
\end{align*}
holds by assumptions in $(ii)$ and $(\tilde{B} -B)u \in \mathcal{S}$. Hence $Bu \in H^{l, k}$ microlocally on $\Omega_{r', R'} (-1)$ for some $r', R'>0$.
\end{proof}
Next we prove that $(H_0 -z)^{-1}$ does not maps $\mathcal{S}$ into $\mathcal{S}$. In the proof, we consider $U^* H_0 U$ rather than $H_0$ since resolvents of the former are easy to treat.
\begin{proposition}\label{2406161956}
For any $z \in \C^{+}$, there exists $f \in \mathcal{S}$ such that $(H_0 -z)^{-1} f \notin \mathcal{S}$.
\end{proposition}
\begin{proof}
Since $U^* H_0 U = xD + Dx -1=: 2B-1$ holds, we have $U^* (H_0 -z)^{-1} U = (2B-z-1)^{-1}$. Since $U$ and $U^*$ are homeomorphisms on $\mathcal{S}$, it suffices to show that there exists $f \in \mathcal{S}$ such that $(2B-1-z)^{-1} f \notin \mathcal{S}$. For $f \in \mathcal{S}$, we define
\begin{align*}
g(x) = \frac{i}{2} \int_{0}^{1} s^{-\frac{1+z}{2} i-1+ \frac{d}{2}} f(xs)ds.
\end{align*}  
Since $\im z >0$, $g \in C^{\infty} (\R^d)$. Furthermore, by integration by parts, 
\begin{align*}
2xDg = \int_{0}^{1} s^{-\frac{1+z}{2} i+ \frac{d}{2}} \frac{d}{ds} (f(xs)) ds = f-(-z-1-di)g
\end{align*}
holds, which implies $(2B-z-1)g =f$. Since
\begin{align*}
\|g\|_2 \le \frac{1}{2} \int_{0}^{1} s^{\frac{d}{2} -1 + \im z} \|f(s \cdot)\|_{L^2 _x} ds = \frac{\|f\|_2}{2} \int_{0}^{1} s^{-1 + \im z} ds < \infty ,
\end{align*}
$g \in D(B)$ and $g= (2B-z-1)^{-1} f$ holds. We suppose $(2B-z-1)^{-1} f \in \mathcal{S}$ for all $f \in \mathcal{S}$. Then, for radial $f(x) = h(|x|)$, we have
\begin{align*}
\mathcal{S} \ni (2B-z-1)^{-1} f = \frac{i}{2} |x|^{\frac{1+z}{2} i- \frac{d}{2}} \int_{0}^{|x|} s^{-\frac{1+z}{2} i-1+ \frac{d}{2}} h(s) ds.
\end{align*}
In particular we obtain
\begin{align*}
\int_{0}^{\infty} s^{-\frac{1+z}{2} i-1+ \frac{d}{2}} h(s) ds = \lim_{|x| \to \infty} \left(\frac{2}{i} |x|^{- \frac{1+z}{2} i + \frac{d}{2}} (2B-z-1)^{-1} f(x) \right) = 0
\end{align*}
for any $h \in C^{\infty} _0 ((0, \infty))$. Since $s^{-\frac{1+z}{2} i-1+ \frac{d}{2}} \in L^1 _{loc} ((0, \infty))$, by the fundamental lemma of the calculus of variations, this yields contradiction.
\end{proof}
Now we are ready to prove Theorem \ref{2406161602}. By Lemma \ref{2406161842}, we do not need to consider compact regions in the phase space.
\begin{proof}[Proof of Theorem \ref{2406161602}]
We first note that if $(H-z)u \in \mathcal{S}$ and $u \in L^2 (\R^d)$, then $Bu \in C^{\infty} (\R^d)$ holds. This follows from $(H-z) Bu = (H-z)u - (H-z) (1-B)u \in \mathcal{S}$ and Lemma \ref{2406161703}, where the former follows from $(1-B)u \in \mathcal{S}$ (see Lemma \ref{2406161842}). Then we first apply Theorem \ref{2406181741} $(ii)$ and next apply $(i)$. As a result we obtain Theorem \ref{2406161602} $(i)$. Theorem \ref{2406161602} $(ii)$ follows from $(i)$ and Proposition \ref{2406161956}.  
\end{proof}
\section{Nonlinear Schr\"odinger equations for infinitely many particles under the effect of Aharonov-Bohm potentials}\label{2406230934}
In this section we consider the Hartree equations for infinitely many particles:
\begin{align}
\left\{
\begin{array}{l}
i\partial_t \gamma=[H_{A, a} +w*\rho_{\gamma},\gamma] \\
\gamma(0)=\gamma_0
\end{array}
\right.
\tag{H}\label{2405192052}
\end{align}
Here $w: \R^d \to \R$ is a function of interaction and $\gamma : \R \to \mathcal{B} (\mathcal{H}) = \mathcal{B} (L^2 (\R^2))$ is an operator-valued function. For $A \in \mathcal{B}(\mathcal{H})$, $\rho_A (x) : = k_A (x, x)$ denotes the density function of $A$, where $k_A (x, y)$ is the integral kernel of $A$. We sometimes write $\rho (A)$. Note that (\ref{2405192052}) is an infinitely many particle version of the following $N$ particle system:
\[
\left\{
\begin{array}{l}
i\partial_t u_j=H_{A, a} u_j+w*(\displaystyle \sum_{k=1}^{k=N} |u_k|^2)u_j \\
u_j(0)= u_{0,j}
\end{array}
\right.
\]     
for $j=1,2,\dots,N$.
We refer to \cite{CHP1}, \cite{CHP2}, \cite{LS}, \cite{LS2}, \cite{Ha}, \cite{HH} and references therein for well-posedness and scattering for infinitely many particle systems, though all of these results are considered for the free Hamiltonian. To state our result we recall the definition of the Schatten class.
\begin{defn}[Schatten space]
Let $\mathcal{H}_1$ and $\mathcal{H}_2$ be a Hilbert space. For a compact operator 
$A : \mathcal{H}_1 \rightarrow \mathcal{H}_2$, the singular values $\{ \mu_n \}$ of $A$  is defined as the set of all the eigenvalues of $(A^*A)^{1/2}$. The Schatten space $\mathfrak{S}^{\alpha} (\mathcal{H}_1 \rightarrow \mathcal{H}_2)$ for $\alpha \in [1, \infty]$ is the set of all the compact operators: $\mathcal{H}_1 \rightarrow \mathcal{H}_2$  such that its singular values belong to $\ell^{\alpha}$. Its norm is defined by the $\ell^\alpha$ norm of the singular values. We also use the following notation for simplicity: $\|S\|_{\mathfrak{S}^{\alpha} (\mathcal{H}_0)} := \|S\|_{\mathfrak{S}^{\alpha} (\mathcal{H}_0 \rightarrow \mathcal{H}_0)}$.
\end{defn}
Concerning the global existence for large initial data, we have the following result.
\begin{thm}\label{2405232007}
Suppose $H_{A, a}$ is as in Example \ref{2405072241}. Let $(p, q)$ satisfy $1 \le p \le \infty, 1 \le q <3$ and $\frac{1}{p} + \frac{1}{q} =1$. We assume $w \in L^{q'}$ and $\gamma _0 \in \mathfrak{S}^{2q/q+1}$. Then there exists a unique global solution $\gamma$ to (\ref{2405192052}) satisfying $\gamma \in C(\R; \mathfrak{S}^{2q/q+1})$ and $\rho _{\gamma} \in L^p _{\mathrm{loc}} (\R; L^q _x)$.
\end{thm}
The proof is based on the orthonormal Strichartz estimate for $H_{A, a}$ and analogous to \cite{H1} and \cite{FS}. Hence we omit it here. Next we prove that, for a small initial data, unique solution $\gamma$ scatters. There seems to have been no result on the scattering for the Schr\"odinger operator with electromagnetic potentials.
\begin{thm}\label{2405232027}
Suppose $H_{A, a}$ is as in Example \ref{2405072241} and $w \in L^1 (\R^2)$. Then there exists $\epsilon _0 >0$ such that if $\|\gamma _0\|_{\mathfrak{S}^{4/3}} < \epsilon _0$, there exists a unique global solution $\gamma$ to (\ref{2405192052}) satisfying $\gamma \in C(\R; \mathfrak{S}^{4/3})$ and $\rho _{\gamma} \in L^2 _t L^2 _x$. Furthermore there exists $\gamma _{\pm} \in \mathfrak{S}^{4/3}$ such that
\begin{align*}
\lim_{t \to \pm \infty} \| \gamma (t) - e^{-itH_{A, 0}} \gamma _{\pm} e^{itH_{A, 0}} \|_{\mathfrak{S}^{4/3}} =0
\end{align*}
holds.
\end{thm} 
Before proving Theorem \ref{2405232027}, we collect some lemmas. The first one is concerned with a propagator of linear Schr\"odinger equations with time-dependent potentials. For the free Schr\"odinger operator, the following lemma is proved by Yajima \cite{Y1} (see also \cite{FLLS}).
\begin{lemma}\label{2405252344}
Let $V \in L^2 _{t, x}$ be a real-valued function. Then there exists a family of unitary operators $\{U_V (t, s)\}$ such that $U_V (s, s) =1$, $U_V (t, s) U_V (s, r) = U_V (t, r)$ and $U_V$ is a propagator associated to $i\partial_t U_V (t, s) = (H_{A, a} +V)U_V (t, s)$, i.e. 
\begin{align}
U_V (t, s)u_0 = e^{-i(t-s)H_{A, a}} u_0 -i \int_{s}^{t} e^{-i(t-\tau)H_{A, a}} V(\tau)U_V (\tau, s) u_0 d\tau \label{2405270156}
\end{align}
holds for all $u_0 \in L^2 (\R^2)$.
\end{lemma}
\begin{proof}
The proof is similar to Lemma 4.15 and Corollary 4.17 in \cite{H1} (see also \cite{Y1}). Therefore we only give modifications needed. By the Strichartz estimate we have
\begin{align*}
\left\| \int_{0}^{t} e^{-i(t-s)H_{A, a}} V(s) u(s)ds \right\|_{C(I;L^2) \cap L^4 (I; L^4)} &\lesssim \|Vu\|_{L^{4/3} (I; L^{4/3})} \\
&\lesssim \|V\|_{L^2 (I; L^2)} \|u\|_{C(I;L^2) \cap L^4 (I; L^4)}
\end{align*}
where $I=[-a, a]$ with $0<a< \infty$. Since $\|V\|_{L^2 _{t, x}} < \infty$, there exists a small $a>0$ such that 
\begin{align*}
C(I;L^2) \cap L^4 (I; L^4) \ni u \to e^{-itH_{A, a}} u_0 -i \int_{0}^{t} e^{-i(t-s)H_{A, a}} V(s) u(s)ds \in C(I;L^2) \cap L^4 (I; L^4)
\end{align*}
is a contraction on $C(I;L^2) \cap L^4 (I; L^4)$. We set $U_V (t, 0)u_0 :=u(t)$, where $u$ is a unique fixed point. We split $\R$ into a finite sum of intervals on which the above argument succeed (this is possible by $\|V\|_{L^2 _{t, x}} < \infty$). Then by glueing solutions we obtain $U_V (t, s)$. See the proof of Lemma 4.15 and Corollary 4.17 in \cite{H1} or \cite{Y1} for more details. The proof of $\|U_V (t, s) u_0\|_2 = \|u_0\|_2$ is totally the same so we omit its proof
\end{proof}
We define $U_V (t):= U_V (t, 0)$ for the sake of simplicity. The next lemma is the orthonormal Strichartz estimates for $U_V$.
\begin{lemma}\label{2405251312}
Let $V \in L^2 _{t, x}$ be a real-valued function and $U_{V}$ be the propagator associated to $i\partial _t U_{V} (t, s) = (H_{A, a} +V) U_{V} (t, s), U_V (s, s) = 1$. Then 
\begin{align*}
\left\| \sum_{n=0}^ \infty{\nu_n| U_V (t)f_n|^2} \right\|_{L^{q/2} _t L^{r/2} _x} \lesssim \| \nu\|_{\ell^\beta}
\end{align*}
holds, where $(q, r, \beta)$ is as in Example \ref{2405072241} and the implicit constant is bounded if $\|V\|_{L^2 _{t, x}}$ is bounded. Furthermore this is equivalent to 
\begin{align}
\|f U_V\|_{\mathfrak{S}^{2\beta '} (L^2 _x \to L^2 _{t, x})} \lesssim \|f\|_{L^{\tilde{q}} _t L^{\tilde{r}} _x} \label{2405251342}.
\end{align}
\end{lemma}
\begin{proof}
The proof is just repeating the argument of Lemma 3.7 in \cite{Ha} with $e^{it\Delta}$ replaced by $e^{-itH_{A, a}}$. Note that the argument there is based on the orthonormal Strichartz estimate for $e^{it\Delta}$ and since we have Example \ref{2405072241}, we do not need other modifications. (\ref{2405251342}) is a consequence of Lemma 3.1 in \cite{Ha}.
\end{proof}
By using Lemma \ref{2405252344} and \ref{2405251312}, we prove the following estimates, which are used in the contraction mapping argument.
\begin{lemma}\label{2405251355}
Let $\gamma_0 \in \mathfrak{S}^{4/3}$. Then we have
\begin{align*}
&\|\rho (U_V (t) \gamma_0 U_V (t)^*)\|_{L^2 _{t, x}} \lesssim \|\gamma_0\|_{\mathfrak{S}^{4/3}} \\
&\|\rho (U_V (t) \gamma_0 U_V (t)^*) - \rho (U_W (t) \gamma_0 U_W (t)^*)\|_{L^2 _{t, x}} \lesssim \|V-W\|_{L^2 _{t, x}} \|\gamma_0\|_{\mathfrak{S}^{4/3}}
\end{align*}
if $\|V\|_{L^2 _{t, x}}$ and $\|W\|_{L^2 _{t, x}}$ are bounded.
\end{lemma}
\begin{proof}
We compute as
\begin{align*}
\left|\int_{\R \times \R^d} f(t, x) \rho (U_V (t) \gamma_0 U_V (t)^*) (x) dx dt \right|&= \left| \int_{\R} \tr [f(t) U_V (t) \gamma_0 U_V (t)^*] dt \right| \\
& = \left| \int_{\R} \tr [U_V (t)^* f^{1/2} f^{1/2} U_V (t) \gamma_0]dt \right| \\
& \lesssim \|\gamma_0\|_{\mathfrak{S}^{4/3}} \||f|^{1/2} U_V\|^2 _{\mathfrak{S}^{8} (L^2 _x \to L^2 _{t, x})} \\
& \lesssim \|\gamma_0\|_{\mathfrak{S}^{4/3}} \||f|^{1/2}\|^2 _{L^4 _{t, x}} = \|\gamma_0\|_{\mathfrak{S}^{4/3}} \|f\|_{L^2 _{t, x}}
\end{align*}
where we have used (\ref{2405251342}). By a duality argument we obtain $\|\rho (U_V (t) \gamma_0 U_V (t)^*)\|_{L^2 _{t, x}} \lesssim \|\gamma_0\|_{\mathfrak{S}^{4/3}}$. Next we transform
\begin{align}
&\rho (U_V (t) \gamma_0 U_V (t)^*) - \rho (U_W (t) \gamma_0 U_W (t)^*) \notag \\
&= \rho ((U_V (t) - U_W (t))\gamma_0 U_V (t)^*) + \rho (U_W (t) \gamma_0 (U_V (t)^* - U_W (t)^*)). \label{2405251541}
\end{align}
To estimate a difference, we use
\begin{align}
U_V (t) - U_W (t) = -i \int_{0}^{t} U_V (t, \tau) (V(\tau) - W(\tau)) U_W (\tau) d\tau. \label{2405270147}
\end{align}
If $U_V (t)$ are differentiable, this is easily verified (e.g.\cite{Ha}). However in our case $U_V (t)$ are not necessarily differentiable since it is constructed by a solution to an integral equation. Hence we need a little technical computation. See \cite{Y1} for a sufficient condition for differentiability of $U_V (t)$. First we note that
\begin{align}
U_V (t, s) = e^{-i(t-s)H_{A, a}} -i \int_{s}^{t} U_V (t, \tau) V(\tau) e^{i(s-\tau)H_{A, a}} d\tau \label{2405270205}
\end{align}
holds. This can be obtained by taking the adjoint of (\ref{2405270156}) and changing $s$ and $t$. Then
\begin{align*}
U_V (t) - U_W (t) = -i \underbrace{\left\{\int_{0}^{t} U_V (t, s) V(s)e^{-is H_{A, a}} ds - \int_{0}^{t} e^{-i(t-s)H_{A, a}} W(s) U_{W} (s) ds \right\}}_{=: I(t)}
\end{align*}
holds. By using (\ref{2405270156}) and (\ref{2405270205}), we obtain
\begin{align*}
I(t) &= \int_{0}^{t} U_V (t, s) V(s) \left\{U_W (s) +i \int_{0}^{s} e^{-i(s-r)H_{A, a}} W(r) U_{W} (r)dr \right\} ds \\
& \quad \quad - \int_{0}^{t} \left\{ U_V (t, s) +i \int_{s}^{t} U_V (t, \tau) V(\tau) e^{i(s-\tau)H_{A, a}} d\tau \right\} W(s)U_W (s) ds \\
& = \int_{0}^{t} U_V (t, s) (V(s) - W(s)) U_W (s) ds \\
& \quad \quad + \underbrace{i\int_{0}^{t} U_V (t, s) V(s) \left\{ \int_{0}^{s} e^{-i(s-r)H_{A, a}} W(r) U_{W} (r)dr \right\} ds}_{=:I_1 (t)} \\
& \quad \quad -\underbrace{i\int_{0}^{t} \int_{s}^{t} U_V (t, \tau) V(\tau) e^{i(s-\tau)H_{A, a}} d\tau W(s) U_W (s) ds}_{=:I_2 (t)}.
\end{align*}
By changing the order of integration in $I_2 (t)$, it is easy to see that $I_1 (t) = I_2 (t)$ holds. Therefore we have proved (\ref{2405270147}).
Then
\begin{align*}
&\left| \int_{\R \times \R^d} f(t, x) \rho ((U_V (t) - U_W (t))\gamma_0 U_V (t)^*) (t, x) dtdx \right| \\
&=\left| \int_{\R \times \R^d} f(t, x) \rho \left(\int_{0}^{t} U_V (t, \tau) (V(\tau) - W(\tau)) U_W (\tau) d\tau \gamma_0 U_V (t)^* \right) dxdt \right| \\
& = \left| \int_{\R} \tr \left[U_V (t)^* f^{1/2} f^{1/2} \left(\int_{0}^{t} U_V (t, \tau) (V(\tau) - W(\tau)) U_W (\tau) d\tau \right) \gamma_0 \right] dt \right| \\
& \lesssim \|\gamma_0\|_{\mathfrak{S}^{4/3}} \left\| \int_{\R} U_V (t)^* f^{1/2} f^{1/2} \left(\int_{0}^{t} U_V (t, \tau) (V(\tau) - W(\tau)) U_W (\tau) d\tau \right) dt \right\|_{\mathfrak{S}^{4}} \\
& \lesssim \|\gamma_0\|_{\mathfrak{S}^{4/3}} \|f\|^{1/2} _{L^2 _t L^2 _x} \left\| f^{1/2} \left(\int_{0}^{t} U_V (t, \tau) (V(\tau) - W(\tau)) U_W (\tau) d\tau \right) \right\|_{\mathfrak{S}^{8} (L^2 _x \to L^2 _{t, x})}
\end{align*}
holds, where we have used the adjoint of (\ref{2405251342}). By the Strichartz estimate we have
\begin{align*}
\||V-W|^{1/2} U_W u \|_{L^2 _{t, x}} \lesssim \||V-W|^{1/2}\|_{L^4 _{t, x}} \|U_W u\|_{L^4 _{t, x}} \lesssim \|V-W\|^{1/2} _{L^2 _{t, x}} \|u\|_2
\end{align*}
and this yields $S: u \mapsto |V-W|^{1/2} \sgn (V-W) U_W u \in \mathcal{B} (L^2 _x ; L^2 _{t, x})$ and $\|S\|_{ \mathcal{B} (L^2 _x ; L^2 _{t, x})} \lesssim \|V-W\|^{1/2} _{L^2 _{t, x}}$. Furthermore
\begin{align*}
&\left\| f^{1/2} \left(\int_{0}^{t} U_V (t, \tau) |V(\tau) - W(\tau)|^{1/2} d\tau \right) \right\|_{\mathfrak{S}^{8} (L^2 _{t, x} \to L^2 _{t, x})} \\
&\lesssim \left\| f^{1/2} \left(\int_{0}^{\infty} U_V (t, \tau) |V(\tau) - W(\tau)|^{1/2} d\tau \right) \right\|_{\mathfrak{S}^{8} (L^2 _{t, x} \to L^2 _{t, x})} \\
&\lesssim \|f^{1/2} U_V\|_{\mathfrak{S}^{8} (L^2 _x \to L^2 _{t, x})} \left\| \int_{0}^{\infty} U_V (\tau)^* |V-W|^{1/2} d\tau \right\|_{\mathcal{B} (L^2 _{t, x} ; L^2 _x)} \\
& \lesssim \|f\|^{1/2} _{L^2 _{t, x}} \|V-W\|^{1/2} _{L^2 _{t, x}}
\end{align*}
holds, where we have used (\ref{2405251342}), Theorem 3.1 in \cite{Ha} and the adjoint of $\|S\|_{ \mathcal{B} (L^2 _x ; L^2 _{t, x})} \lesssim \|V-W\|^{1/2} _{L^2 _{t, x}}$. As a result we obtain
\begin{align*}
\left\| f^{1/2} \left(\int_{0}^{t} U_V (t, \tau) (V(\tau) - W(\tau)) U_W (\tau) d\tau \right) \right\|_{\mathfrak{S}^{8} (L^2 _x \to L^2 _{t, x})} \lesssim \|f\|^{1/2} _{L^2 _{t, x}} \|V-W\|_{L^2 _{t, x}}
\end{align*}
and this yields $\|\rho ((U_V (t) - U_W (t))\gamma_0 U_V (t)^*)\|_{L^2 _{t, x}} \lesssim \|\gamma_0\|_{\mathfrak{S}^{4/3}} \|V-W\|_{L^2 _{t, x}}$. Similarly we can prove $\|\rho (U_W (t) \gamma_0 (U_V (t)^* - U_W (t)^*))\|_{L^2 _{t, x}} \lesssim \|\gamma_0\|_{\mathfrak{S}^{4/3}} \|V-W\|_{L^2 _{t, x}}$ and we obtain the desired estimates.
\end{proof}
Now we are ready to prove Theorem \ref{2405232027}. The proof is based on the argument in \cite{LS}.
\begin{proof}[Proof of Theorem \ref{2405232027}]
We define
\begin{align*}
\mathcal{L} [g] (t, x) = \rho (U_{w*g} (t) \gamma_0 U_{w*g} (t) ^*) (t, x)
\end{align*}
for $g \in L^2 _{t, x}$. Since $w \in L^1 _x$ implies $w*g \in L^2 _{t, x}$, $\mathcal{L}$ is well-defined. Furthermore Lemma \ref{2405251312} (or Lemma \ref{2405251355}) yields $\mathcal{L} : L^2 _{t, x} \to L^2 _{t, x}$ is bounded. We set
\begin{align*}
X:= \{ f \in L^2 _{t, x} \mid \|f\|_{L^2 _{t, x}} \le 1\}
\end{align*}
and prove that $\mathcal{L}$ is a contraction on $X$ if $\|\gamma_0\|_{\mathfrak{S}^{4/3}}$ is sufficiently small. By Lemma \ref{2405251355} we have
\begin{align*}
&\|\mathcal{L} [g]\|_{L^2 _{t, x}} \lesssim \|\gamma_0\|_{\mathfrak{S}^{4/3}}, \\
&\|\mathcal{L} [g] - \mathcal{L} [f]\|_{L^2 _{t, x}} \lesssim \|g-f\|_{L^2 _{t, x}} \|\gamma_0\|_{\mathfrak{S}^{4/3}}
\end{align*}
for all $f, g \in X$, where the implicit constants do not depend on $f$ and $g$ since $\|f\|_{L^2 _{t, x}}, \|g\|_{L^2 _{t, x}} \le 1$. Therefore, there exists $\epsilon_0 >0$ such that if $\|\gamma_0\|_{\mathfrak{S}^{4/3}} < \epsilon_0$, $\mathcal{L}$ is a contraction on $X$. Let $\phi \in X$ be a unique fixed point of $\mathcal{L}$ and set
\begin{align*}
\gamma (t) = U_{w*\phi} (t) \gamma_0 U_{w*\phi} (t)^*.
\end{align*}
Then we have $\rho (\gamma (t)) = \rho (U_{w*\phi} (t) \gamma_0 U_{w*\phi} (t)^*) = \phi (t)$. Since $U_{w* \rho (\gamma)}$ is strongly continuous and $\gamma_0 \in \mathfrak{S}^{4/3}$, by Gr\"umm's convergence theorem (Theorem 2.19 in \cite{S}), we have $\gamma \in C (\R; \mathfrak{S}^{4/3})$. Furthermore
\begin{align*}
&\|e^{itH_{A, a}} \gamma (t) e^{-itH_{A, a}} - e^{isH_{A, a}} \gamma (s) e^{-isH_{A, a}}\|_{\mathfrak{S}^{4/3}} \\
& \lesssim \|e^{itH_{A, a}} U_{w*\rho_{\gamma}} (t) - e^{isH_{A, a}} U_{w*\rho_{\gamma}} (s)\|_{\mathcal{B} (L^2 _x)} \|\gamma_0\|_{\mathfrak{S}^{4/3}} \\
& \quad \quad \quad + \|U_{w*\rho_{\gamma}} (t)^* e^{-itH_{A, a}} - U_{w*\rho_{\gamma}} (s)^* e^{-isH_{A, a}} \|_{\mathcal{B} (L^2 _x)} \|\gamma_0\|_{\mathfrak{S}^{4/3}}
\end{align*}  
and
\begin{align*}
\|e^{itH_{A, a}} U_{w*\rho_{\gamma}} (t) u - e^{isH_{A, a}} U_{w*\rho_{\gamma}} (s)u\|_2 &= \left\|\int_{s}^{t} e^{i\tau H_{A, a}} w*\rho_{\gamma} (\tau) U_{w*\rho_{\gamma}} (\tau) ud\tau \right\|_2 \\
& \le \left\| \int_{s}^{r} e^{-i(r- \tau) H_{A, a}} w*\rho_{\gamma} (\tau) U_{w*\rho_{\gamma}} (\tau) ud\tau \right\|_{L^{\infty} ([s, t]; L^2 _x)} \\
& \lesssim \|w*\rho_{\gamma} U_{w*\rho_{\gamma}} u\|_{L^{4/3} ([s, t]; L^{4/3} _x)} \\
& \lesssim \|\rho_{\gamma}\|_{L^2 ([s, t]; L^2 _x)} \|U_{w*\rho_{\gamma}} u\|_{L^4 _{t, x}} \lesssim \|\rho_{\gamma}\|_{L^2 ([s, t]; L^2 _x)} \|u\|_2
\end{align*}
implies $\|e^{itH_{A, a}} \gamma (t) e^{-itH_{A, a}} - e^{isH_{A, a}} \gamma (s) e^{-isH_{A, a}}\|_{\mathfrak{S}^{4/3}} \to 0$ as $s, t \to \infty$. Therefore there exists $\eta_{\pm} \in \mathfrak{S}^{4/3}$ such that $\|\gamma (t) - e^{-itH_{A, a}} \eta_{\pm} e^{itH_{A, a}}\|_{\mathfrak{S}^{4/3}} \to 0$ as $t \to \pm \infty$. Then by the completeness of $W_{\pm}$ (see the proof of Corollary \ref{2405081750}) and Gr\"umm's convergence theorem, there exists $\gamma_{\pm} \in \mathfrak{S}^{4/3}$ such that
\begin{align*}
\lim_{t \to \pm \infty} \| \gamma (t) - e^{-itH_{A, 0}} \gamma _{\pm} e^{itH_{A, 0}} \|_{\mathfrak{S}^{4/3}} =0
\end{align*}
holds and hence we are done.
\end{proof}
\appendix
\section{Interpolation spaces}\label{2405082014}
Here we prove that real interpolation spaces of Sobolev spaces associated to $H_{A, 0}$ are the Besov spaces associated to $H_{A, 0}$. This is used in the proof of Corollary \ref{2405081750}.
\begin{lemma}\label{2405082057}
Let $H_{A, 0}$ be the operator appeared in Corollary \ref{2405081750}. Then we have 
\begin{align*}
(\dot{H}^{s_0} (\sqrt{H_{A, 0}}), \dot{H}^{s_1} (\sqrt{H_{A, 0}}))_{\theta, q} = \dot{B}^{s} _{2, q} (\sqrt{H_{A, 0}})
\end{align*}
for $s_0 < s < s_1$, $s= (1-\theta)s_0 + \theta s_1$, $\theta \in (0, 1)$ and $q \in (2, \infty)$.
\end{lemma}
This lemma would hold for any $q \in [1, \infty]$ but we only consider $q \in (2, \infty)$ since we do not use other cases. We mimic the argument for the ordinary Besov spaces in \cite{T}.
\begin{proof}
For simplicity we set $\dot{H}^t = \dot{H}^t (\sqrt{H_{A, 0}})$ and $\dot{B}^t _{2, r} = \dot{B}^t _{2, r} (\sqrt{H_{A, 0}})$ for $t \in \R$ and $r \in (2, \infty)$. First we prove $\|f\|_{\dot{B}^s _{2, q}} \lesssim \|f\|_{(\dot{H}^{s_0}, \dot{H}^{s_1})_{\theta, q}}$ for all $f \in E((\epsilon, \frac{1}{\epsilon})) L^2 _x$ and $\epsilon \in (0, 1)$, where $E$ denotes the spectral measure associated with $H_{A, 0}$. For such $f$, we set $K(t, f) = \inf_{f=f_0 +f_1} (\|f_0\|_{\dot{H}^{s_0}} +t\|f_1\|_{\dot{H}^{s_1}})$ for any $t>0$, where the infimum is taken for all $f = f_0 + f_1 \in \dot{H}^{s_0} + \dot{H}^{s_1}$. Let $\{\phi_j\}$ be a homogeneous Littlewood-Paley decomposition. Then
\begin{align*}
\|\phi_j (\sqrt{H_{A, 0}}) f\|_2 &\le \|\phi_j (\sqrt{H_{A, 0}}) f_0\|_2 + \|\phi_j (\sqrt{H_{A, 0}}) f_1\|_2 \\
& \lesssim 2^{-s_0 j} \|(\sqrt{H_{A, 0}})^{s_0} f_0\|_2 + 2^{-s_1 j} \|(\sqrt{H_{A, 0}})^{s_1} f_1\|_2
\end{align*}
holds. By taking the infimum, we obtain $\|\phi_j (\sqrt{H_{A, 0}}) f\|_2 \lesssim 2^{-s_0 j} K(2^{j(s_0 - s_1)}, f)$. This implies
\begin{align*}
\|\{ 2^{js} \|\phi_j (\sqrt{H_{A, 0}}) f\|_2\}\|_{l^q} \lesssim \|\{ 2^{-\theta j (s_0 -s_1)} K(2^{j(s_0 - s_1)}, f)\}\|_{l^q}
\end{align*}
Since $K(t, f)$ increases monotonically with respect to $t>0$, we obtain
\begin{align*}
2^{\theta (s_0 -s_1)} 2^{-\theta (j+1)(s_0 -s_1)} K(2^{(j+1)(s_0 -s_1)}, f) \le t^{-\theta} K(t, f) \le 2^{-\theta (s_0 -s_1)} 2^{-\theta j(s_0 -s_1)} K(2^{j(s_0 -s_1)}, f)
\end{align*}
for all $t \in [2^{(j+1)(s_0 -s_1)}, 2^{j(s_0 -s_1)}]$. Therefore 
\begin{align*}
&(s_1 -s_0) \log 2\cdot 2^{\theta (s_0 -s_1)q} \{ 2^{-\theta (j+1)(s_0 -s_1)} K(2^{(j+1)(s_0 -s_1)}, f)\}^q \\
&\le \int_{2^{(j+1)(s_0 -s_1)}}^{2^{j(s_0 -s_1)}} (t^{-\theta} K(t, f))^q \frac{dt}{t} \\
&\le (s_1 -s_0) \log 2 \cdot 2^{-\theta (s_0 -s_1)q} \{ 2^{-\theta j(s_0 -s_1)} K(2^{j(s_0 -s_1)}, f)\}^q
\end{align*}
holds. for all $j \in \Z$. By taking the sum, we obtain
\begin{align*}
 \|\{ 2^{-\theta j (s_0 -s_1)} K(2^{j(s_0 - s_1)}, f)\}\|_{l^q} \approx \left( \int_{0}^{\infty} (t^{-\theta} K(t, f))^q \frac{dt}{t}\right)^{\frac{1}{q}} = \|f\|_{(\dot{H}^{s_0}, \dot{H}^{s_1})_{\theta, q}}.
\end{align*}
Hence we have proved $\|f\|_{\dot{B}^s _{2, q}} \lesssim \|f\|_{(\dot{H}^{s_0}, \dot{H}^{s_1})_{\theta, q}}$. Next we prove the other inequality: $\|f\|_{(\dot{H}^{s_0}, \dot{H}^{s_1})_{\theta, q}} \lesssim \|f\|_{\dot{B}^s _{2, q}}$. By the above argument we have
\begin{align*}
\|f\|^q _{(\dot{H}^{s_0}, \dot{H}^{s_1})_{\theta, q}} = \int_{0}^{\infty} t^{-\theta q} K(t, f)^q \frac{dt}{t} \approx \sum_{k \in \Z} 2^{-q\theta k(s_0 -s_1)}K(2^{k(s_0 -s_1)}, f)^q.
\end{align*}
We set $f_{0, k} = \sum_{j=k+1}^{\infty} \phi_j (\sqrt{H_{A, 0}}) f$ and $f_{1, k} = \sum_{j= -\infty}^{k} \phi_j (\sqrt{H_{A, 0}}) f$. Then by a support property of $\{\phi_j\}$, we obtain
\begin{align*}
\|f_{0, k}\|^2 _{\dot{H}^{s_0}} \lesssim \sum_{j=k+1}^{\infty} 2^{2js_0} \|\phi_j (\sqrt{H_{A, 0}})f\|^2 _2,\quad \|f_{1, k}\|^2 _{\dot{H}^{s_1}} \lesssim \sum_{j=-\infty}^{k} 2^{2js_1} \|\phi_j (\sqrt{H_{A, 0}})f\|^2 _2.
\end{align*}
This yields
\begin{align*}
&\|f\|^q _{(\dot{H}^{s_0}, \dot{H}^{s_1})_{\theta, q}} \\
&\lesssim \sum_{k \in \Z} 2^{qk(s-s_0)} \left[\sum_{j=k+1}^{\infty} 2^{2js_0} \|\phi_j (\sqrt{H_{A, 0}})f\|^2 _2 + \sum_{j=-\infty}^{k} 2^{2k(s_0 -s_1)} \cdot 2^{2js_1} \|\phi_j (\sqrt{H_{A, 0}})f\|^2 _2 \right]^{\frac{q}{2}}
\end{align*}
by the definition of $K$. By taking $s_0 < \chi_0 < s< \chi_1 < s_1$ and $\sigma >0$ such that $\frac{1}{q} + \frac{1}{\sigma} = \frac{1}{2}$, we obtain
\begin{align*}
\|f\|^q _{(\dot{H}^{s_0}, \dot{H}^{s_1})_{\theta, q}} &\lesssim \sum_{k \in \Z} 2^{qk(s-s_0)} \left( \sum_{j=k+1}^{\infty} 2^{(s_0 -\chi_0)\sigma j}\right)^{\frac{q}{\sigma}} \left( \sum_{j=k+1}^{\infty} 2^{\chi_0 qj} \|\phi_j (\sqrt{H_{A, 0}})f\|^q _2\right) \\
& \quad + \sum_{k \in \Z} 2^{qk(s-s_1)} \left( \sum_{j=-\infty}^{k} 2^{(s_1 -\chi_1)\sigma j}\right)^{\frac{q}{\sigma}} \left( \sum_{j=-\infty}^{k} 2^{\chi_1 qj} \|\phi_j (\sqrt{H_{A, 0}})f\|^q _2\right) \\
& \lesssim \sum_{j \in \Z} 2^{\chi_0 qj} \|\phi_j (\sqrt{H_{A, 0}})f\|^q _2 \sum_{k \le j-1} 2^{qk(s-\chi_0)} + \sum_{j \in \Z} 2^{\chi_1 qj} \|\phi_j (\sqrt{H_{A, 0}})f\|^q _2 \sum_{k \ge j} 2^{qk(s-\chi_1)} \\
& \lesssim \sum_{j \in \Z} 2^{qjs} \|\phi_j (\sqrt{H_{A, 0}})f\|^q _2 = \|f\|^q _{\dot{B}^s _{2, q}}.
\end{align*}
Therefore we have $\|f\|_{(\dot{H}^{s_0}, \dot{H}^{s_1})_{\theta, q}} \approx \|f\|_{\dot{B}^s _{2, q}}$ for all $f \in E((\epsilon, \frac{1}{\epsilon})) L^2 _x$ and $\epsilon \in (0, 1)$. By the density of such $f$, we have proved $(\dot{H}^{s_0} (\sqrt{H_{A, 0}}), \dot{H}^{s_1} (\sqrt{H_{A, 0}}))_{\theta, q} = \dot{B}^{s} _{2, q} (\sqrt{H_{A, 0}})$.
\end{proof}

\section{Lemma for Mourre theory}\label{202405222307}
Here we give a supplementary lemma which is used when we apply the weakly conjugate operator method (see the proof of Theorem \ref{2405200111}). We follow the argument in \cite{Mo} with some modifications. See also \cite{ABG} and \cite{GG}.
\begin{lemma}\label{2405222315}
Let $H$ and $A$ be as in Theorem \ref{2405200111}. We set $R_{\epsilon} (z) = (H-z-i\epsilon M_1)^{-1}$, where $M_1$ is as in the proof of Theorem \ref{2405200111} . Then we have
\begin{align*}
R_{\epsilon} (z) (D(A)) \subset D(A).
\end{align*}
\end{lemma}
\begin{proof}
First we notice that $e^{itA} (D(H)) \subset D(H)$ holds for all $t \in \R$. This is a consequence of  $e^{itA} (D(H_0)) \subset D(H_0)$, which follows from $H_0 \in C^1 (A)$ (see \cite{BCHM} Lemma 3.9) and $D(H) = D(H_0)$. Then by Proposition 3.2.5 in \cite{ABG}, $e^{itA} \in \mathcal{B} (D(H))$ and
\begin{align*}
\sup_{|t| <1} \|He^{itA} u\|_2 < \infty
\end{align*}
holds for all $u \in D(H)$. Since $M_1$ and $M_2$ are bounded, the assumptions $(a) \sim (d)$ in \cite{Mo} are satisfied. Therefore by Proposition II.3 in \cite{Mo}, $(A-z)^{-1} (D(H)) \subset D(H)$ holds if $|\im z|$ is sufficiently large. We set $A(\lambda) =i\lambda A(A+i\lambda)^{-1} = i\lambda + \lambda ^2 (A+i\lambda)^{-1}$ for $\lambda >0$. Since $\im R_{\epsilon} (z) \subset D(H)$ holds, we have
\begin{align*}
&\langle (A+i\lambda)^{-1} H R_{\epsilon} (z) u, v \rangle - \langle H (A+i\lambda)^{-1} R_{\epsilon} (z) u, v \rangle \\
& = \langle -iM_1 (A+i\lambda)^{-1} R_{\epsilon} (z) u, (A-i\lambda)^{-1} v \rangle
\end{align*}
and this yields
\begin{align*}
R_{\epsilon} (z) [A(\lambda), H] R_{\epsilon} (z)=i R_{\epsilon} (z) i\lambda(A+i\lambda)^{-1} M_1  i\lambda(A+i\lambda)^{-1} R_{\epsilon} (z).
\end{align*}
Furthermore 
\begin{align*}
[A(\lambda), M_1] = i\lambda(A+i\lambda)^{-1} iM_2 i\lambda(A+i\lambda)^{-1}
\end{align*}
holds. By using $[R_{\epsilon} (z), A(\lambda)] = R_{\epsilon} (z) [A(\lambda), H-z-i\epsilon M_1 ]R_{\epsilon} (z)$, for $u \in D(A)$, 
\begin{align*}
A(\lambda) R_{\epsilon} (z)u &= [A(\lambda), R_{\epsilon} (z)]u + R_{\epsilon} (z) A(\lambda)u \\
& \to -iR_{\epsilon} (z) M_1 R_{\epsilon} (z) -\epsilon R_{\epsilon} (z) M_2 R_{\epsilon} (z) +R_{\epsilon} (z) Au
\end{align*}
as $\lambda \to \infty$. Since $i\lambda (A+i\lambda)^{-1} R_{\epsilon} (z)u \to R_{\epsilon} (z) u$, we obtain $R_{\epsilon} (z) u \in D(A)$.
\end{proof}

\section*{Acknowledgement}
The author would like to thank his supervisor Kenichi Ito for valuable comments and discussions. He is also thankful to Kouichi Taira for explaining radial estimates. He is partially supported by FoPM, WINGS Program, the University of Tokyo.

\end{document}